\documentclass[11pt,oneside]{article}

\usepackage{amsthm, amsmath, amssymb, amsfonts,epsfig}
\usepackage{mathtools}
\usepackage{esint}
\usepackage{epstopdf}

\usepackage{graphicx}
\usepackage[font=sl,labelfont=bf]{caption}
\usepackage{subcaption}

\usepackage{enumerate}

\usepackage[colorlinks=true, pdfstartview=FitV, linkcolor=blue,
            citecolor=blue, urlcolor=blue]{hyperref}
\usepackage[usenames]{color}
\definecolor{Red}{rgb}{0.7,0,0.1}
\definecolor{Green}{rgb}{0,0.7,0}

\usepackage[numbers]{natbib}

\usepackage{url}
\makeatletter\def\url@leostyle{%
 \@ifundefined{selectfont}{\def\UrlFont{\sf}}{\def\UrlFont{\scriptsize\ttfamily}}} \makeatother

\usepackage{accents, comment}

\usepackage{bbm}

\usepackage{mathrsfs}

\newtheorem{theorem}{Theorem}
\newtheorem{proposition}[theorem]{Proposition}
\newtheorem{lemma}[theorem]{Lemma}
\newtheorem{corollary}[theorem]{Corollary}

\theoremstyle{definition}
\newtheorem{definition}[theorem]{Definition}

\theoremstyle{remark}
\newtheorem{remark}[theorem]{Remark}

\numberwithin{equation}{section}
\numberwithin{theorem}{section}

\def\cA{\mathcal{A}}
\def\cB{\mathcal{B}}
\def\cC{\mathcal{C}}
\def\cD{\mathcal{D}}

\def\cI{\mathcal{I}}
\def\cJ{\mathcal{J}}

\def\cR{\mathcal{R}}

\def\cT{\mathcal{T}}

\def\cZ{\mathcal{Z}}

\def\bC{\mathbb{C}}

\def\bN{\mathbb{N}}

\def\bR{\mathbb{R}}

\newcommand{\1}{\mathbbm{1}}                     
\newcommand{\spm}{\scriptscriptstyle}

\title{Asymptotic Criticality of the Navier-Stokes Regularity Problem\footnote{to appear in \emph{JMFM}}}

\author{
Zoran Gruji\'c  \small  \url{zgrujic@uab.edu}
\and 
 Liaosha Xu \small  \url{liaosha.xu@ucr.edu}
}

\begin{document}

\maketitle

\begin{abstract}

\footnotesize{The problem of global-in-time regularity for the 3D Navier-Stokes equations, i.e., the question of whether a smooth flow can exhibit spontaneous formation of singularities, is a fundamental open problem in mathematical physics. Due to the super-criticality of the equations, the problem has been super-critical in the sense that there has been a `scaling gap' between any regularity criterion and the corresponding \emph{a priori} bound (regardless of the functional setup utilized). The purpose of this work is to present a mathematical framework--based on a suitably defined `scale of sparseness' of the super-level sets of the positive and negative parts of the components of the higher-order spatial derivatives of the velocity field--in which the scaling gap between the regularity class and the corresponding \emph{a priori} bound vanishes as the order of the derivative goes to infinity.}
\end{abstract}

\vspace{.3in}

\section{Introduction}

3D Navier-Stokes equations (NSE) -- describing a flow of 3D incompressible, viscous, Newtonian fluid -- read
\begin{align*}
 u_t+(u\cdot \nabla)u=-\nabla p + \Delta u,
\end{align*}
supplemented with the incompressibility condition $\, \mbox{div} \, u = 0$, where $u$ is the velocity of the fluid and $p$ is the pressure (here, the viscosity is set to 1 and the external force to zero).
Taking the curl yields the vorticity formulation,
\begin{align*}
\omega_t+(u\cdot \nabla)\omega= (\omega \cdot \nabla)u + \Delta \omega
\end{align*}
where $\omega = \, \mbox{curl} \, u$ is the vorticity of the fluid. The LHS is the transport of the vorticity
by the velocity, the first term on the RHS is the vortex-stretching term, and the second one the diffusion.

There is a unique scaling that leaves the NSE invariant. Let $\lambda > 0$ be a scaling parameter;
it is transparent that if $u=u(x,t) \, \mbox{and} \, p=p(x,t)$ solve the NSE, then
\begin{align*}
u_\lambda(x,t)=\lambda \, u(\lambda x, \lambda^2 t) \ \ \ \mbox{and} \ \ \
p_\lambda(x,t)=\lambda^2 \, p(\lambda x, \lambda^2 t)
\end{align*}
solve the NSE as well (corresponding to the rescaled initial condition, and over the rescaled time
interval).

3D NS regularity problem has been super-critical in the sense that there has been
a `scaling gap' between any known regularity criterion and the corresponding
\emph{a priori} bound. An illustrative example is the classical Ladyzhenskaya-Prodi-Serrin
regularity criterion, $u \in L^p (0,T; L^q)$,
\begin{align*}
\frac{3}{q} + \frac{2}{p} = 1
\end{align*}
\emph{vs.} the corresponding \emph{a priori} bound $u \in L^p (0,T; L^q)$,
\begin{align*}
\frac{3}{q} + \frac{2}{p} = \frac{3}{2}
\end{align*}
(for a suitable range of the parameters). As a matter of fact, all the known regularity criteria are (at best) scaling-invariant, while all the \emph{a priori} bounds had been on the scaling level of the energy bound, $u \in L^\infty (0,T; L^2)$, reflecting the super-criticality of the system \emph{per se}.

\medskip

\emph{Spatial intermittency} of the regions of intense fluid activity has been well-documented in the computational simulations of the 3D NSE. This phenomenon inspired a mathematical study of turbulent dissipation in the 3D NS flows based on the concept of \emph{sparseness at scale} whose local-1D version was introduced in \citet{Grujic2013} -- some key notions are recalled below.
Let $S$ be an open subset of $\bR^d$ and $\mu$ the $d$-dimensional Lebesgue measure.
\begin{definition}\label{def:1DSparse}
For a spatial point $x_0$ and $\delta\in (0,1)$, an open set $S$ is 1D $\delta$-sparse around $x_0$ at scale $r$ if there exists a unit vector $\nu$ such that
\begin{align*}
\frac{\mu\left(S\cap (x_0-r\nu, x_0+r\nu)\right)}{2r} \le \delta\ .
\end{align*}
\end{definition}
The volumetric version is the following.
\vspace{-0.1in}
\begin{definition}\label{def:3DSparse}
For a spatial point $x_0$ and $\delta\in (0,1)$, an open set $S$ is $d$-dimensional $\delta$-sparse around $x_0$ at scale $r$ if
\begin{align*}
\frac{\mu\left(S\cap B_r(x_0)\right)}{\mu(B_r(x_0))} \le \delta\ .
\end{align*}
$S$ is said to be $r$-semi-mixed with ratio $\delta$ if the above inequality holds for every $x_0\in\bR^d$. (It is straightforward to check that for any $S$, $d$-dimensional $\delta$-sparseness at scale $r$ implies 1D $(\delta)^\frac{1}{d}$-sparseness at scale $r$ around any spatial point $x_0$; however the converse is false, i.e. local-1D sparseness is in general a weaker condition.)
\end{definition}

The main idea in this approach is simple. Local-in-time analytic smoothing (in the spatial variables), measured in $L^\infty$, represents a very strong manifestation of the diffusion in the 3D NS system. This provides a suitable environment for the application of the harmonic measure majorization principle -- shortly, if the regions of the intense fluid activity are `sparse enough', the associated harmonic measure will be `small enough' to prevent any further growth of the $L^\infty$-norm and -- in turn -- any singularity formation. Essentially, it suffices that the scale of sparseness of the super-level sets of the field of interest, cut at a fraction of the $L^\infty$-norm, be dominated by a fraction of the scale of the analyticity radius.

In what follows, let us denote the positive and the negative parts of the vectorial components of $f$ by $f_i^\pm$, and compute the norm of a vector $v=(a_1, a_2, a_3, \dots, a_d)$ as $|v| = \max_{1 \le i \le d} \{|a_i|\}$.

\begin{definition}[\citet{Farhat2017} and \citet{Bradshaw2019}]\label{def:VecCompLSet}
For a positive exponent $\alpha$, and a selection of parameters $\lambda$ in $(0,1)$, $\delta$ in $(0,1)$ and $c_0>0$, the class of functions $Z_\alpha(\lambda, \delta; c_0)$ consists of bounded, continuous functions $f : \mathbb{R}^d \to \mathbb{R}^d$ subjected to the following uniformly local condition. For $x_0$ in $\mathbb{R}^d$, select the/a component $f_i^\pm$ such that $f_i^\pm(x_0) = |f(x_0)|$, and require that the set
\begin{align*}
S_i^\pm(x_0):=\biggl\{ x \in \mathbb{R}^d: \, f_i^\pm(x) > \lambda \|f\|_\infty\biggr\}
\end{align*}
be $d$-dimensional $\delta$-sparse around $x_0$ at scale $c \, \frac{1}{\|f\|_\infty^\alpha}$, for some $c$ comparable to $c_0$. Enforce this for all $x_0$ in $\mathbb{R}^d$. Here, $\alpha$ is the scaling parameter, $c_0$ is the size-parameter, and $\lambda$ and $\delta$ are the (interdependent) `tuning parameters'.
\end{definition}

\begin{remark}
On one hand, it is plain that $f \in L^p_w$ implies $f \in Z_\frac{p}{d}$ (here $L^p_w$ denotes the weak Lebesgue space). On the other hand, in the geometrically worst case scenario for sparseness, the super-level set being a single ball, being in $Z_\frac{p}{d}$ is consistent with being in $L^p_w$ (of course, in general, $f \in Z_\frac{p}{d}$ gives no information on the decay of the distribution function of $f$).
\end{remark}

\medskip

Applying this methodology to the vorticity field $\omega$ led to the reduction of the scaling gap within the framework (\citet{Bradshaw2019}), shortly, the class of
\emph{a priori} sparseness is $Z_\frac{2}{5}$ while the regularity class is $Z_\frac{1}{2}$ (this stems from the special structure of the vorticity formulation of the
3D NS system; if one worked with the full gradient the regularity class would be $Z_\frac{3}{5}$ reflecting the standard scaling gap, $Z_\frac{2}{5}$ \emph{vs.}
$Z_\frac{3}{5}$).

\medskip

To illustrate the gain in a bit more tangible way, consider an isolated singularity
of a Leray solution at $(x_0, T)$,
and assume a simple buildup (the super-level sets being approximately balls) of the vorticity singular profile compatible with $\frac{1}{|x-x_0|^\delta}$. Then, the standard $L^p$-theory confines the possible values of $\delta$ to the interval $[2, 3)$, while the $Z_\alpha$-theory confines them to the interval $[2, \frac{5}{2})$, eliminating the $[\frac{5}{2}, 3)$-range.

\medskip

At this point, a natural question arises of whether a further reduction of the scaling gap within the $Z_\alpha$ framework might be possible or whether there might be an obstruction in the way. Before presenting the main result, let us briefly mention two instances of criticality and (a slight) sub-criticality of the NS regularity problem within the framework based on the vorticity as the underlying field of interest.

\medskip

The first concerns a simple geometric scenario in which one arrives at criticality. Namely, suppose that the structure of the vorticity super-level sets is dominated by an ensemble of $O(1)$-long vortex filaments (formation and persistence of  $O(1)$-long filaments has been observed in computational simulations of turbulent flows). Then the \emph{a priori} bound $\omega \in L^\infty(0, T; L^1)$ and Chebyshev's inequality imply that the solution in view is in $Z_\frac{1}{2}$ (in this case, the transversal scale of the filament is comparable to the scale of sparseness)
 (cf. \citet{Grujic2016}).

\medskip

The second concerns a  `non-filamentary' scenario -- more precisely -- a flow initiated at the Kida vortex constrained with the maximal number of symmetries on the periodic cube. A careful computational study of the scale of sparseness in this case was performed in \citet{Rafner2019} revealing that -- within a time interval leading to the peak of the vorticity magnitude -- the solution stabilized in $Z_{\frac{1}{2}+\epsilon}$, explaining the eventual slump as a consequence of turbulent dissipation and revealing a slight sub-criticality of the Kida flow within the framework.

\medskip

In what follows, consider the higher-order spatial fluctuations of the velocity field -- more specifically -- the sequence of functional classes $Z^{(k)}_{\alpha_k}$ defined by the following rule. For a positive, decreasing sequence $\{\alpha_k\}$,
\begin{align*}
 u \in Z^{(k)}_{\alpha_k} \ \ \ \mbox{if} \ \ \ D^{(k)} u \in Z_{\alpha_k}.
\end{align*}
Then, the main result of this paper can be summarized in the following table ($T^*$ denotes a possible singular time).

\begin{table}[h]
{
\begin{center}
\begin{tabular}{ | p{4.5cm} | p{4.5cm} | }
\hline
 ~&~
\\
Regularity class
& A priori bound
\\  ~&~ \\ \hline ~&~
\\   $u(\tau) \in \bigcap_{k \ge k^*} Z^{(k)}_\frac{1}{k+1}$ on a suitable $(T^* - \epsilon, T^*)$, small size-parameters
(uniform in time),  $k^*$ can be taken arbitrary large
&  $u(\tau) \in \bigcap_{k \ge 0} Z^{(k)}_\frac{1}{k+\frac{3}{2}}$  on a suitable $(T^* - \epsilon, T^*)$,
the size-parameters uniform in time
\\ ~ & ~ 
\\\hline
\end{tabular}
\end{center}
}
\end{table}

\noindent (The precise statements are given in Theorem~\ref{th:MainResult} and Theorem~\ref{th:LerayZalpha}, respectively.)

\medskip

\noindent It is informative to present the level-$k$ scales of sparseness realizing the above functional classes.

{
\begin{table}[h!]
{
\begin{center}
\begin{tabular}{ | p{3.9cm} | p{3.7cm} |}
\hline
 ~&~
\\
Regularity class-scale
& A priori bound-scale
\\  ~&~ \\ \hline ~&~
\\  $\frac{1}{C_1(k)} \frac{1}{\|D^{(k)} u\|_\infty^\frac{1}{k+1}}$
&  $C_2(\|u_0\|_2, k) \frac{1}{\|D^{(k)} u\|_\infty^\frac{1}{k+\frac{3}{2}}}$
\\ ~ & ~ 
\\\hline
\end{tabular}
\end{center}
}
\end{table}
}

\newpage

\noindent A closer look at the scaling of dynamic quantities in the table reveals the following.
{
\begin{table}[h!]
{
\begin{center}
\begin{tabular}{ | p{3.9cm} | p{3.7cm} |}
\hline
 ~&~
\\
Regularity class-scale
& A priori bound-scale
\\  ~&~ \\ \hline ~&~
\\  $\frac{1}{\|D^{(k)} u\|_\infty^\frac{1}{k+1}} \approx r$
&  $\frac{1}{\|D^{(k)} u\|_\infty^\frac{1}{k+\frac{3}{2}}} \approx r^\frac{k+1}{k+\frac{3}{2}}$
\\  ~ & ~ 
\\\hline
\end{tabular}
\end{center}
}
\end{table}
}

\noindent Since
\begin{align*}
r^\frac{k+1}{k+\frac{3}{2}}  \to  r, \ \ \  k \to \infty
\end{align*}
and $k^*$ can be taken arbitrary large, we term this phenomenon  \emph{asymptotic criticality}.

\begin{remark}
The 3D NS system features one (known) fundamental cancellation,
\begin{align*}
\int (u \cdot \nabla)u \cdot u \, dx = 0,
\end{align*}
which -- in turn -- implies the \emph{a priori} boundedness of the kinetic energy, i.e., $u \in L^\infty(0,T; L^2)$. This is
away from the level at which the nonlinearity and the diffusion equilibrate -- the scaling-invariant level (e.g., $u \in L^\infty(0,T; L^3)$) -- illustrating the scaling gap. 
In the $Z^{(k)}_{\alpha_k}$ framework, as $k$ increases, the energy bound provides enhanced
sparseness (Theorem~\ref{th:LerayZalpha}) which -- via the harmonic measure majorization principle -- yields the improved bounds on the $L^\infty$-norm of $D^{(k)}u$.
\end{remark}

\medskip

The main results are detailed in Section 3. Here we present a bit of heuristics behind the proof, 
identifying a principal source of the scaling gain.

\medskip

Note that the \emph{a priori} scale of sparseness  (Theorem~\ref{th:LerayZalpha}) \emph{vs.} 
the lower bound on the scale 
of the analyticity radius at level-$k$ (Theorem~\ref{th:MainThmVel}) is
\[
 \|D^{(k)}u\|_\infty^{-\frac{1}{k+\frac{3}{2}}}  \ \ \ \ \ \  vs. \ \ \ \ \ \ \  \|D^{(k)}u\|_\infty^{-\frac{3}{2}\frac{1}{k+\frac{3}{2}}},
\]
and recall that in the general approach based on sparseness of the regions of the intense fluid activity, a possible formation 
of singularities will be prevented as long as the scale of the analyticity radius dominates the scale of sparseness (of the field 
in view; here $D^{(k)} u$). Transparently, the gap here is independent of $k$, reflecting the super-criticality of the system.

\medskip

The scaling gain stems from the 
observation that certain monotonicity properties (in $k$) of the sequence 
$\displaystyle{ \{ \|D^{(k)}u\|_\infty \}_{k=0}^\infty }$, either increasing/ascending
or decreasing/descending, yield a much stronger bound on the analyticity radius, of the order of
\[
  \|D^{(k)}u\|_\infty^{-\frac{1}{k+1}}.
\]
The ascending
scenario is treated in Theorem~\ref{le:AscendDer}, while the descending one in Theorem~\ref{le:DescendDer}.
The utility of the ascending property (which is a more plausible road to a singularity) is
in replacing the classical Gagliardo-Nirenberg interpolation inequalities in estimating the Leibniz
expansion of the nonlinearity, over a suitable range of indices.

\medskip

Once these two key scenarios are well understood, 
the task shifts to deconstructing the local dynamics accordingly. Synchronization of all the moving parts turned out to be intricate, 
and is presented in the proof of the main result (Theorem~\ref{th:MainResult}, including the three
lemmas). The mechanism behind the argument can be thought of as `dynamic interpolation'.

\medskip

In what follows, we will consider the general NS system in $\bR^d$,

\begin{align}
&u_t -\Delta u+ u\cdot\nabla u+\nabla p=f, &&\textrm{ in }\bR^d\times (0,T) \label{eq:NSE1}
\\
&\textrm{div}\ u=0,                                &&\textrm{ in }\bR^d\times (0,T) \label{eq:NSE2}
\\
&u(\cdot,0)=u_0(\cdot),                            &&\textrm{ in }\bR^d\times \{t=0\} \label{eq:NSE3}
\end{align}
where $u$ is the velocity of the fluid, $p$ is the pressure, $f$ is the external force, and $u_0$ is the given initial velocity  (here, the viscosity is set to 1 and the external force $f(\cdot,t)$ is a real-analytic vector field in space).
More precisely, all the
velocity-based results will be set up in $\bR^d$, while all the vorticity-based results will be set up in $\bR^3$.

\section{Spatial analyticity initiated at level $k$}\label{sec:PrfThm}

Since the notion of sparseness is utilized via the harmonic measure maximum principle for subharmonic functions, and the lower bound on the radius of spatial analyticity of solutions plays a key role in its application, the primary purpose of this section is to develop spatial analyticity results for the higher order derivatives.
We start by recalling the results on the spatial analyticity of velocity and vorticity
obtained in \citet{Guberovic2010} and \citet{Bradshaw2019}, respectively, inspired by the method for determining
a lower bound on the uniform radius of spatial analyticity of solutions in $L^p$ spaces introduced in \citet{Grujic1998}.

\begin{theorem}[\citet{Guberovic2010} and \citet{Bradshaw2019}]\label{th:LinftyIVP}
Let the initial datum $u_0\in L^\infty$ (resp. $\omega_0\in L^\infty\cap L^1$). Then, for any $M>1$, there exists a constant $c(M)$ such that there is a unique mild solution $u$ (resp. $\omega$) in $C_w([0,T], L^\infty)$ where $T\ge\frac{1}{c(M)^2\|u_0\|_\infty^2}$ (resp. $T\ge\frac{1}{c(M)\|\omega_0\|_\infty}$), which has an analytic extension $U(t)$ (resp. $W(t)$) to the region
\begin{align*}
\cD_t:=\left\{x+iy\in\bC^3\ :\ |y|\le \sqrt{t}/c(M)\ \left(\textrm{resp. }|y|\le \sqrt{t}/\sqrt{c(M)}\right)\right\}
\end{align*}
for all $t\in[0,T]$, and
\begin{align*}
\sup_{t\le T}\|U(t)\|_{L^\infty(\cD_t)}\le M\|u_0\|_\infty\ \left(\textrm{resp. }\sup_{t\le T}\|W(t)\|_{L^\infty(\cD_t)}\le M\|\omega_0\|_\infty\right).
\end{align*}
\end{theorem}

\medskip

The following two lemmas are included for the reader's convenience.

\begin{lemma}[\citet{Nirenberg1959} or \citet{Gagliardo1959}]\label{le:GNIneq}
Suppose $p,q,r\in[1,\infty]$, $s\in\bR$ and $m,j,d\in\bN$ satisfy
\begin{align*}
\frac{1}{p} = \frac{j}{d} + \left(\frac{1}{r}-\frac{m}{d}\right)s + \frac{1-s}{q} \ , \qquad \frac{j}{m}\le s\le 1.
\end{align*}
Then there exists constant $C$ only depending on $m,d,j,q,r,s$ such that for any function $f: \bR^d\to\bR^d$
\begin{align*}
\|D^j f\|_{L^p}\le C\|D^m f\|_{L^r}^s\|f\|_{L^q}^{1-s}.
\end{align*}
\end{lemma}

\medskip

\begin{lemma}[Montel's]\label{le:Montel}
Let $q\in[1,\infty]$ and let $\mathscr{F}$ be a family of analytic functions $f$ on an open set $\Omega\subset \bC^d$ such that
\vspace{-0.15in}
\begin{align*}
\underset{f\in\mathscr{F}}{\sup}\ \|f\|_{L^q(\Omega)}<\infty\ .
\end{align*}
Then $\mathscr{F}$ is a normal family.
\end{lemma}

\medskip

The main result of this section is as follows.

\begin{theorem}\label{th:MainThmVel}
Assume $u_0\in L^\infty(\bR^d)\cap L^p(\bR^d)$ for some $p \ge 2$ and $f(\cdot,t)$ is divergence-free and real-analytic in the space variable with the analyticity radius at least $\delta_f$ for all $t\in[0,\infty)$ with the analytic extension $f+ig$ satisfying
\begin{align*}
\Gamma_\infty^k(t) &:= \sup_{s<t}\sup_{|y|<\delta_f} \left(\|D^kf(\cdot,y,s)\|_{L^\infty}+\|D^kg(\cdot,y,s)\|_{L^\infty}\right)<\infty\ ,
\\
\Gamma_p(t) &:= \sup_{s<t}\sup_{|y|<\delta_f} \left(\|f(\cdot,y,s)\|_{L^p}+\|g(\cdot,y,s)\|_{L^p}\right)<\infty\ .
\end{align*}
Fix $k\in\bN$, $M>1$ and $t_0>0$ and let
\begin{align}
T_*&=\min\left\{\left(C_1(M) 2^{2k} \left(\|u_0\|_p + \Gamma_p(t_0)\right)^{2k/(k+\frac{d}{p})} \left(\|D^ku_0\|_\infty + \Gamma_\infty^k(t_0)\right)^{\frac{2d}{p}/(k+\frac{d}{p})} \right)^{-1}, \right. \notag
\\
&\qquad \left. \left(C_2(M) \left(\|u_0\|_p + \Gamma_p(T)\right)^{(k-1)/(k+\frac{d}{p})} \left(\|D^ku_0\|_\infty + \Gamma_\infty^k(T)\right)^{(1+\frac{d}{p})/(k+\frac{d}{p})} \right)^{-1} \right\} \label{eq:TimeLength}
\end{align}
where $C_i(M)$ are constants depending only on $M$. Then there exists a solution
\begin{align*}
u\in C([0,T_*),L^p(\bR^d)^d) \cap C([0,T_*),C^\infty(\bR^d)^d)
\end{align*}
of the NSE \eqref{eq:NSE1}-\eqref{eq:NSE3} such that for every $t\in (0,T_*)$, $u$ is a restriction of an analytic function $u(x,y,t)+iv(x,y,t)$ in the region
\begin{align}\label{eq:AnalDom}
\cD_t=: \left\{(x,y)\in\bC^d\ \big|\ |y|\le \min\{c(M)t^{1/2},\delta_f\}\right\}\ .
\end{align}
Moreover, $D^ku\in C([0,T_*),L^\infty(\bR^d)^d)$ and
\begin{align}
&\underset{t\in(0,T_*)}{\sup}\ \underset{y\in\cD_t}{\sup} \|u(\cdot,y,t)\|_{L^p} + \underset{t\in(0,T_*)}{\sup}\ \underset{y\in\cD_t}{\sup} \|v(\cdot,y,t)\|_{L^p}\le M\left(\|u_0\|_p + \Gamma_p(T_*)\right)\ ,
\\
&\underset{t\in(0,T_*)}{\sup}\ \underset{y\in\cD_t}{\sup} \|D^ku(\cdot,y,t)\|_{L^\infty} + \underset{t\in(0,T_*)}{\sup}\ \underset{y\in\cD_t}{\sup} \|D^kv(\cdot,y,t)\|_{L^\infty}\le M\left(\|D^ku_0\|_\infty + \Gamma_\infty^k(T_*)\right)\ .
\end{align}
A simplified version of the above result holds for real solutions, in which case, the time span $T_*$ is larger for the same constant $M$ while $\tilde{\Gamma}_p(T_*)$ and $\tilde{\Gamma}_\infty^k(T_*)$ do not contain the imaginary part $g$.
\end{theorem}

\noindent{\textit{Proof.}} \ We construct an approximating sequence as follows:
\begin{align*}
&u^{(0)}=0\ ,\quad \pi^{(0)}=0\ ,
\\
&\partial_t u^{(n)} -\Delta u^{(n)} = -\left(u^{(n-1)}\cdot \nabla\right)u^{(n-1)} - \nabla\pi^{(n-1)} +f\ ,
\\
&u^{(n)}(x,0)=u_0(x)\ ,\quad \nabla\cdot u^{(n)}=0\ ,
\\
&\Delta\pi^{(n)}=-\partial_j\partial_k\left(u_j^{(n)} u_k^{(n)}\right)\ .
\end{align*}
By the induction argument as in \citet{Guberovic2010}, $u^{(n)}(t) \in C([0,T], L^\infty(\bR^d))$ and each $u^{(n)}(t)$ is real analytic for every $t\in (0,T]$. Let $u^{(n)}(x,y,t)+iv^{(n)}(x,y,t)$ and $\pi^{(n)}(x,y,t)+i\rho^{(n)}(x,y,t)$ be the analytic extensions of $u^{(n)}$ and $\pi^{(n)}$ respectively. Inductively we have analytic extensions for all approximate solutions and the real and imaginary parts satisfy
\begin{align}
\partial_t u^{(n)}-\Delta u^{(n)} &= -\left(u^{(n-1)}\cdot \nabla\right)u^{(n-1)} + \left(v^{(n-1)}\cdot \nabla\right)v^{(n-1)} - \nabla\pi^{(n-1)} + f\ ,
\\
\partial_t v^{(n)}-\Delta v^{(n)} &= -\left(u^{(n-1)}\cdot \nabla\right)v^{(n-1)} - \left(v^{(n-1)}\cdot \nabla\right)u^{(n-1)} - \nabla\rho^{(n-1)} + g\ ,
\end{align}
where
\begin{align*}
\Delta \pi^{(n)} &= -\partial_j\partial_k\left(u^{(n)}_ju^{(n)}_k-v^{(n)}_jv^{(n)}_k\right),\qquad \Delta\rho^{(n)} = -2\partial_j\partial_k\left(u^{(n)}_jv^{(n)}_k\right)\ .
\end{align*}
Now define
\begin{align*}
U_\alpha^{(n)}(x,t)&= u^{(n)}(x,\alpha t,t), && \Pi_\alpha^{(n)}(x,t)= \pi^{(n)}(x,\alpha t,t), & F_\alpha(x,t)=f(x,\alpha t, t),
\\
V_\alpha^{(n)}(x,t)&= v^{(n)}(x,\alpha t,t), && R_\alpha^{(n)}(x,t)= \rho^{(n)}(x,\alpha t,t), & G_\alpha(x,t)=g(x,\alpha t, t);
\end{align*}
then the approximation scheme becomes (for simplicity we drop the subscript $\alpha$)
\begin{align*}
\partial_t U^{(n)}- \Delta U^{(n)} &= -\alpha\cdot \nabla V^{(n)}-\left(U^{(n-1)}\cdot \nabla\right)U^{(n-1)} + \left(V^{(n-1)}\cdot \nabla\right)V^{(n-1)} - \nabla\Pi^{(n-1)} + F\ ,
\\
\partial_t V^{(n)}- \Delta V^{(n)} &= -\alpha\cdot \nabla U^{(n)}-\left(U^{(n-1)}\cdot \nabla\right)V^{(n-1)} - \left(V^{(n-1)}\cdot \nabla\right)U^{(n-1)} - \nabla R^{(n-1)} + G\ ,
\\
\Delta \Pi^{(n)} &= -\partial_j\partial_k\left(U^{(n)}_jU^{(n)}_k-V^{(n)}_jV^{(n)}_k\right),\qquad \Delta R^{(n)} = -2\partial_j\partial_k\left(U^{(n)}_jV^{(n)}_k\right)\ ,
\end{align*}
supplemented with the initial conditions
\begin{align*}
U^{(n)}(x,0)=u_0(x),\qquad V^{(n)}(x,0)=0 \qquad\textrm{for all }x\in\bR^d\ ,
\end{align*}
leading to the following set of iterations,
\begin{align}
U^{(n)}(x,t)& =e^{t\Delta} u_0 - \int_0^te^{(t-s)\Delta}\left(U^{(n-1)}\cdot\nabla\right)U^{(n-1)} ds + \int_0^te^{(t-s)\Delta}\left(V^{(n-1)}\cdot\nabla\right)V^{(n-1)} ds \notag
\\
&\quad -\int_0^te^{(t-s)\Delta}\nabla\Pi^{(n-1)}ds + \int_0^te^{(t-s)\Delta} F\ ds - \int_0^te^{(t-s)\Delta}\alpha\cdot\nabla V^{(n)}ds\ , \label{eq:IterationU}
\\
V^{(n)}(x,t)& = - \int_0^te^{(t-s)\Delta}\left(U^{(n-1)}\cdot\nabla\right)V^{(n-1)} ds - \int_0^te^{(t-s)\Delta}\left(V^{(n-1)}\cdot\nabla\right)U^{(n-1)} ds \notag
\\
&\quad -\int_0^te^{(t-s)\Delta}\nabla R^{(n-1)}ds + \int_0^te^{(t-s)\Delta} G\ ds - \int_0^te^{(t-s)\Delta}\alpha\cdot\nabla U^{(n)}ds \label{eq:IterationV}
\end{align}
where
\begin{align*}
\Pi^{(n)}(x,t)& =-(\Delta)^{-1}\sum \partial_j\partial_k\left(U^{(n)}_jU^{(n)}_k-V^{(n)}_jV^{(n)}_k\right),
\\
R^{(n)}(x,t) &=-2(\Delta)^{-1}\sum \partial_j\partial_k\left(U^{(n)}_jV^{(n)}_k\right).
\end{align*}
In view of Theorem~\ref{th:LinftyIVP}, without loss of generality assume $u_0\in C^\infty$ and
$\left\|D^ku_0\right\|\lesssim k!\left\|u_0\right\|^k$.
Taking the $k$-th derivative on both sides of \eqref{eq:IterationU} and \eqref{eq:IterationV} yields
\begin{align}
D^kU^{(n)}(x,t)& =e^{t\Delta} D^ku_0 - \int_0^te^{(t-s)\Delta} D^k\left(U^{(n-1)}\cdot\nabla\right)U^{(n-1)} ds + \int_0^te^{(t-s)\Delta} D^k\left(V^{(n-1)}\cdot\nabla\right)V^{(n-1)} ds \notag
\\
&\quad -\int_0^te^{(t-s)\Delta}\nabla D^k\Pi^{(n-1)}ds + \int_0^te^{(t-s)\Delta} D^kF\ ds - \int_0^te^{(t-s)\Delta}\alpha\cdot\nabla D^kV^{(n)}ds\ , \label{eq:IterationDU}
\\
D^kV^{(n)}(x,t)& = - \int_0^te^{(t-s)\Delta} D^k\left(U^{(n-1)}\cdot\nabla\right)V^{(n-1)} ds - \int_0^te^{(t-s)\Delta} D^k\left(V^{(n-1)}\cdot\nabla\right)U^{(n-1)} ds \notag
\\
&\quad -\int_0^te^{(t-s)\Delta}\nabla D^k R^{(n-1)}ds + \int_0^te^{(t-s)\Delta} D^kG\ ds - \int_0^te^{(t-s)\Delta}\alpha\cdot\nabla D^kU^{(n)}ds \label{eq:IterationDV}.
\end{align}
We claim that
\begin{align*}
&L_n:=\underset{t<T}{\sup}\ \|U^{(n)}\|_{L^p}+\underset{t<T}{\sup}\ \|V^{(n)}\|_{L^p}\ ,\quad L_n':=\underset{t<T}{\sup}\ \|D^kU^{(n)}\|_{L^\infty}+\underset{t<T}{\sup}\ \|D^kV^{(n)}\|_{L^\infty}\
\end{align*}
are all bounded by a constant determined only by $k$, $\|u_0\|_{L^\infty}$, $\|u_0\|_{L^p}$, $F$ and $G$.

\bigskip

\textit{Proof of the claim:} At the initial step of the iteration, i.e.
\begin{align}
U^{(0)}(x,t)& =e^{t\Delta} u_0 -\int_0^te^{(t-s)\Delta} F\ ds - \int_0^te^{(t-s)\Delta}\alpha\cdot\nabla V^{(0)}ds\ , \label{eq:IniIterU}
\\
V^{(0)}(x,t)& = \int_0^te^{(t-s)\Delta} G\ ds - \int_0^te^{(t-s)\Delta}\alpha\cdot\nabla U^{(0)}ds\ , \label{eq:IniIterV}
\end{align}
the $L^p$-estimates are as follows:
\begin{align}
\|U^{(0)}\|_{L^p}& \lesssim \|e^{t\Delta}u_0\|_{L^p} + \int_0^t\| e^{(t-s)\Delta} F \|_{L^p}ds + \int_0^t\| e^{(t-s)\Delta} \alpha\cdot \nabla V^{(0)} \|_{L^p}ds \notag
\\
& \lesssim \|u_0\|_{L^p} + \int_0^t \|F\|_{L^p}ds + |\alpha|\int_0^t\| \nabla e^{(t-s)\Delta} V^{(0)} \|_{L^p}ds \notag
\\
& \lesssim \|u_0\|_{L^p} + t\ \underset{s<t}{\sup}\ \|F\|_{L^p} + |\alpha|\int_0^t(t-s)^{-1/2}\|V^{(0)} \|_{L^p}ds \notag
\\
& \lesssim \|u_0\|_{L^p} + t\ \underset{s<t}{\sup}\ \|F\|_{L^p} + |\alpha|t^{1/2}\underset{s<t}{\sup}\ \|V^{(0)} \|_{L^p}\ . \label{eq:bmoEstU0}
\end{align}
Similarly,
\begin{align}\label{eq:bmoEstV0}
\|V^{(0)}\|_{L^p} \lesssim\|u_0\|_{L^p} + t\ \underset{s<t}{\sup}\ \|G\|_{L^p} + |\alpha|t^{1/2}\underset{s<t}{\sup}\ \|U^{(0)}\|_{L^p}\ .
\end{align}
If $\alpha$ is a vector such that $C|\alpha|t^{1/2}<1/2$ for all $t<T$ (with a suitable $C$),
then combining \eqref{eq:bmoEstU0} and \eqref{eq:bmoEstV0} gives
\begin{align}\label{eq:InitialBMO}
\underset{t<T}{\sup}\ \|U^{(0)}\|_{L^p}+\underset{t<T}{\sup}\ \|V^{(0)}\|_{L^p}\lesssim\|u_0\|_{L^p} + T \left(\underset{t<T}{\sup}\|F\|_{L^p}+\underset{t<T}{\sup}\|G\|_{L^p}\right)\ .
\end{align}
Taking the $L^\infty$-norms of the $j$th derivative on both sides of \eqref{eq:IniIterU} and \eqref{eq:IniIterV} yields
\begin{align*}
\|D^kU^{(0)}\|_{L^\infty}& \lesssim \|e^{t\Delta}D^ku_0\|_{L^\infty} + \int_0^t\| e^{(t-s)\Delta} D^kF \|_{L^\infty}ds + \int_0^t\| e^{(t-s)\Delta} \alpha\cdot \nabla D^kV^{(0)} \|_{L^\infty}ds
\\
& \lesssim \|D^ku_0\|_{L^\infty} + \int_0^t\|D^kF \|_{L^\infty}ds + |\alpha|\int_0^t (t-s)^{-1/2}\|D^kV^{(0)} \|_{L^\infty}ds
\\
& \lesssim \|D^ku_0\|_{L^\infty} + t\|D^kF \|_{L^\infty} + |\alpha|t^{1/2}\|D^kV^{(0)} \|_{L^\infty}\ .
\end{align*}
Similarly, if $\alpha$ is such that $C|\alpha|t^{1/2}<1/2$ for all $t<T$ (with a suitable $C$), then
\begin{align}\label{eq:InitialLinfty}
\underset{t<T}{\sup}\ \|D^kU^{(0)}\|_{L^\infty}+\underset{t<T}{\sup}\ \|D^kV^{(0)}\|_{L^\infty}\lesssim \|D^ku_0\|_{L^\infty} + T \left(\underset{t<T}{\sup} \|D^kF\|_{L^\infty}+\underset{t<T}{\sup} \|D^kG\|_{L^\infty}\right)\ .
\end{align}
Collecting the estimates \eqref{eq:InitialBMO} and \eqref{eq:InitialLinfty},
\begin{align*}
L_0 &\lesssim \|u_0\|_{L^p} + T \left(\underset{t<T}{\sup}\|F\|_{L^p}+\underset{t<T}{\sup}\|G\|_{L^p}\right)
\\
L_0' &\lesssim \|D^ju_0\|_{L^\infty} + T \left(\underset{t<T}{\sup} \|D^jF\|_{L^\infty}+\underset{t<T}{\sup} \|D^jG\|_{L^\infty}\right)\ .
\end{align*}
To control the rest of $L_n$ and $L_n'$ in the iteration scheme, the nonlinear and the pressure estimates play the
essential role. We demonstrate the $L^\infty$-estimates on the three representative terms, namely,

\[
\int_0^t e^{(t-s)\Delta}(U^{(n)}\cdot \nabla)U^{(n)}ds, \  \int_0^t e^{(t-s)\Delta}(U^{(n)}\cdot \nabla)V^{(n)}ds
 \ \ \mbox{and}  \  \ \int_0^t e^{(t-s)\Delta}\nabla\Pi^{(n)}ds.
\]
First, observe that by Lemma~\ref{le:GNIneq},
\begin{align}
\|D^jU^{(n)}\|_\infty \lesssim \|D^kU^{(n)}\|_\infty^{(j+\frac{d}{p})/(k+\frac{d}{p})}\|U^{(n)}\|_p^{(k-j)/(k+\frac{d}{p})}\ ,\qquad\forall\ 0\le j \le k.
\end{align}
In addition, since $\nabla\cdot U^{(n)}=0$, $(U^{(n)}\cdot\nabla) U^{(n)}=\nabla \cdot (U^{(n)}\otimes U^{(n)})$. The first term is then estimated as follows,
\begin{align*}
&\left\|\int_0^te^{(t-s)\Delta}\nabla D^k\left(U^{(n-1)}\otimes U^{(n-1)}\right) ds\right\|_\infty \lesssim \int_0^t (t-s)^{-\frac{1}{2}} ds \left(\sum_{i=0}^k \binom k i \|D^iU^{(n-1)}\|_\infty\|D^{k-i}U^{(n-1)}\|_\infty \right)
\\
&\qquad\qquad\qquad\lesssim t^{\frac{1}{2}} \left(\sum_{i=0}^k \binom k i \|D^kU^{(n-1)}\|_\infty^{(i+\frac{d}{p})/(k+\frac{d}{p})}\|U^{(n-1)}\|_p^{(k-i)/(k+\frac{d}{p})} \right.
\\
&\qquad\qquad\qquad\qquad\qquad\qquad\qquad\qquad \left. \|D^kU^{(n-1)}\|_\infty^{(k-i+\frac{d}{p})/(k+\frac{d}{p})}\|U^{(n-1)}\|_p^{i/(k+\frac{d}{p})} \right)
\\
&\qquad\qquad\qquad\lesssim t^{\frac{1}{2}} \sum_{i=0}^k \binom k i  \|U^{(n-1)}\|_p^{k/(k+\frac{d}{p})}  \|D^kU^{(n-1)}\|_\infty^{(k+\frac{2d}{p})/(k+\frac{d}{p})}
\\
&\qquad\qquad\qquad\lesssim t^{\frac{1}{2}} 2^k \|U^{(n-1)}\|_p^{k/(k+\frac{d}{p})} \|D^kU^{(n-1)}\|_\infty^{(k+\frac{2d}{p})/(k+\frac{d}{p})}.
\end{align*}
Similarly,
\begin{align*}
&\left\|\int_0^te^{(t-s)\Delta}\nabla D^k\Pi^{(n-1)}ds\right\|_\infty \lesssim \int_0^t \left\|\nabla e^{(t-s)\Delta} D^k\Pi^{(n-1)} \right\|_\infty ds
\\
&\qquad \lesssim \int_0^t (t-s)^{-\frac{1}{2}} \left(\left\|D^k\left(U^{(n-1)}\otimes U^{(n-1)}\right)\right\|_{BMO} + \left\|D^k\left(V^{(n-1)}\otimes V^{(n-1)}\right)\right\|_{BMO} \right)ds
\\
&\qquad \lesssim t^{\frac{1}{2}} 2^k \left(\|U^{(n-1)}\|_p^{k/(k+\frac{d}{p})} \|D^kU^{(n-1)}\|_\infty^{(k+\frac{2d}{p})/(k+\frac{d}{p})} + \|V^{(n-1)}\|_p^{k/(k+\frac{d}{p})} \|D^kV^{(n-1)}\|_\infty^{(k+\frac{2d}{p})/(k+\frac{d}{p})}\right).
\end{align*}
For the mixed product term,
\begin{align*}
&\left\|\int_0^te^{(t-s)\Delta} D^k\left(U^{(n-1)}\cdot\nabla\right)V^{(n-1)} ds\right\|_\infty
\lesssim \int_0^t (t-s)^{-\frac{1}{2}} ds \left\|D^{k-1}\left(U^{(n-1)}\cdot\nabla\right)V^{(n-1)} \right\|_\infty
\\
&\qquad\qquad\qquad \lesssim \int_0^t (t-s)^{-\frac{1}{2}} ds \left(\sum_{i=0}^{k-1} \binom {k-1} i \|D^iU^{(n-1)}\|_\infty\|D^{k-i}V^{(n-1)}\|_\infty \right)
\\
&\qquad\qquad\qquad \lesssim t^{\frac{1}{2}} \left(\sum_{i=0}^{k-1} \binom {k-1} i \|D^kU^{(n-1)}\|_\infty^{(i+\frac{d}{p})/(k+\frac{d}{p})}\|U^{(n-1)}\|_p^{(k-i)/(k+\frac{d}{p})} \right.
\\
&\qquad\qquad\qquad \qquad\qquad\qquad \left. \|D^kV^{(n-1)}\|_\infty^{(k-i+\frac{d}{p})/(k+\frac{d}{p})}\|V^{(n-1)}\|_p^{i/(k+\frac{d}{p})} \right)
\\
&\qquad\qquad\qquad \lesssim t^{\frac{1}{2}} 2^{k-1} \|U^{(n-1)}\|_p^{(k-i)/(k+\frac{d}{p})}\|V^{(n-1)}\|_p^{i/(k+\frac{d}{p})}
\\
&\qquad\qquad\qquad \qquad\qquad\qquad \|D^kU^{(n-1)}\|_\infty^{(i+\frac{d}{p})/(k+\frac{d}{p})} \|D^kV^{(n-1)}\|_\infty^{(k-i+\frac{d}{p})/(k+\frac{d}{p})}.
\end{align*}
The $L^p$-estimates are demonstrated on the pressure term,
\begin{align*}
\left\|\int_0^te^{(t-s)\Delta}\nabla \Pi^{(n-1)}ds\right\|_p &\lesssim \int_0^t \left\|e^{(t-s)\Delta} (\Delta)^{-1}\sum \partial_j\partial_k \nabla \left(U^{(n-1)}_jU^{(n-1)}_k-V^{(n-1)}_jV^{(n-1)}_k\right)\right\|_p ds
\\
&\lesssim \int_0^t \left\|\nabla \left(U^{(n-1)}_jU^{(n-1)}_k-V^{(n-1)}_jV^{(n-1)}_k\right)\right\|_p ds
\\
&\lesssim t \left(\|\nabla U^{(n-1)}\|_\infty \|U^{(n-1)}\|_p + \|\nabla V^{(n-1)}\|_\infty \|V^{(n-1)}\|_p \right)
\\
&\lesssim t \left(\|U^{(n-1)}\|_p^{(2k-1+\frac{d}{p})/(k+\frac{d}{p})}  \|D^k U^{(n-1)}\|_\infty^{(1+\frac{d}{p})/(k+\frac{d}{p})} \right.
\\
&\qquad \quad \left. + \|V^{(n-1)}\|_p^{(2k-1+\frac{d}{p})/(k+\frac{d}{p})}  \|D^k V^{(n-1)}\|_\infty^{(1+\frac{d}{p})/(k+\frac{d}{p})}.\right)
\end{align*}
In conclusion, the above argument implies
\begin{align*}
\|D^kU^{(n)}\|_{L^\infty} &\lesssim \|D^ku_0\|_{L^\infty} + T^{\frac{1}{2}} 2^k (L_{n-1})^{k/(k+\frac{d}{p})} \left(L_{n-1}'\right)^{(k+\frac{2d}{p})/(k+\frac{d}{p})}
\\
&\qquad + T\ \underset{t<T}{\sup}\|D^kF\|_{L^\infty} + |\alpha|t^{\frac{1}{2}} \|D^kU^{(n)}\|_{L^\infty}\
\end{align*}
and
\begin{align*}
\|U^{(n)}\|_{L^p} &\lesssim \|u_0\|_{L^p} + T\ (L_{n-1})^{(2k-1+\frac{d}{p})/(k+\frac{d}{p})} \left(L_{n-1}'\right)^{(1+\frac{d}{p})/(k+\frac{d}{p})}
\\
&\qquad + T\ \underset{t<T}{\sup}\|F\|_{L^p} + |\alpha|t^{\frac{1}{2}} \|U^{(n)}\|_{L^p} \ .
\end{align*}
Hence, with $|\alpha|t^{1/2}<1/2$,
\begin{align*}
L_{n}' &\lesssim \|D^ku_0\|_{L^\infty} + T^{\frac{1}{2}} 2^k (L_{n-1})^{k/(k+\frac{d}{p})} \left(L_{n-1}'\right)^{(k+\frac{2d}{p})/(k+\frac{d}{p})}+ \Gamma_\infty^k(T)\ ,
\\
L_n &\lesssim \|u_0\|_{L^p} + T\ (L_{n-1})^{(2k-1+\frac{d}{p})/(k+\frac{d}{p})} \left(L_{n-1}'\right)^{(1+\frac{d}{p})/(k+\frac{d}{p})}+ \Gamma_p(T)\ .
\end{align*}
In particular, if
\begin{align*}
T\le C\left(2^{2k}\cdot M^4/(M-1)^2 \cdot \left(\|u_0\|_p + \Gamma_p(T)\right)^{2k/(k+\frac{d}{p})} \cdot \left(\|D^ku_0\|_\infty + \Gamma_\infty^k(T)\right)^{\frac{2d}{p}/(k+\frac{d}{p})} \right)^{-1}
\end{align*}
and
\begin{align*}
T\le C\left(M^2/(M-1) \cdot \left(\|u_0\|_p + \Gamma_p(T)\right)^{(k-1)/(k+\frac{d}{p})} \cdot \left(\|D^ku_0\|_\infty + \Gamma_\infty^k(T)\right)^{(1+\frac{d}{p})/(k+\frac{d}{p})} \right)^{-1}
\end{align*}
then an induction argument yields
\begin{align*}
L_n' \le M\left(\|D^ku_0\|_\infty + \Gamma_\infty^k(T)\right),\qquad L_n \le M\left(\|u_0\|_p + \Gamma_p(T)\right) ,
\end{align*}
completing the proof of the claim.

\bigskip

Now the standard convergence argument based on Lemma~\ref{le:Montel} (applied for each $t$ with $q=p$ or $q=\infty$)
completes the proof that the limit function $u$ (i.e. the complexified solution of the NSE \eqref{eq:NSE1}-\eqref{eq:NSE3}) exists and is bounded locally uniformly in time (the time interval depends only on $k$, $\|u_0\|_{L^\infty}$, $\|u_0\|_{L^p}$, $F$ and $G$) and uniformly in $y$-variables over the complex domain
\begin{align*}
\cD_t=: \left\{(x,y)\in\bC^d\ \big|\ |y|\le \min\{ct^{1/2},\delta_f\}\right\}\
\end{align*}
with the upper bound only depending on $k$, $\|u_0\|_{L^\infty}$, $\|u_0\|_{L^p}$, $F$ and $G$.  The analyticity properties of $u$, are justified by the uniform convergence on any compact subset of $\cD_t$, following from Lemma~\ref{le:Montel} (see \citet{Grujic1998} and \citet{Guberovic2010} for more details). This ends the proof of Theorem~\ref{th:MainThmVel}.

\bigskip

An analogous result for the vorticity is the following.

\begin{theorem}\label{th:MainThmVor}
Assume the initial value $\omega_0\in L^\infty(\bR^3)\cap L^p(\bR^3)$ where $1\le p<3$. Fix $k\in\bN$, $M>1$ and $t_0>0$ and let
\begin{align}
T_*&=C(M)\cdot \min\left\{2^{-k} \left( \|\omega_0\|_p^{k/(k+\frac{d}{p})} \cdot \|D^k\omega_0\|_\infty^{\frac{d}{p}/(k+\frac{d}{p})} + \|\omega_0\|_p \right)^{-1}, \right. \notag
\\
&\qquad \left. \left(\|D^k\omega_0\|_\infty^{\frac{d/p}{k+\frac{d}{p}}}\|\omega_0\|_p^{\frac{k}{k+\frac{d}{p}}} + \|D^k\omega_0\|_\infty^{\frac{1+\frac{d}{p}}{k+\frac{d}{p}}}\|\omega_0\|_p^{\frac{k-1}{k+\frac{d}{p}}} + \|D^k\omega_0\|_\infty^{\frac{1}{k+\frac{d}{p}}}\|\omega_0\|_p^{\frac{k-1+\frac{d}{p}}{k+\frac{d}{p}}} \right)^{-1} \right\}
\end{align}
where $C_i(M)$ is a constant only depending on $M$ and $d=3$. Then there exists a solution
\begin{align*}
\omega\in C([0,T_*),L^p(\bR^3)^3) \cap C([0,T_*),C^\infty(\bR^3)^3)
\end{align*}
of the NSE \eqref{eq:NSE1}-\eqref{eq:NSE3} such that for every $t\in (0,T_*)$, $\omega$ is a restriction of an analytic function $\omega(x,y,t)+i\zeta(x,y,t)$ in the region
\begin{align}\label{eq:AnalDomVor}
\cD_t=: \left\{(x,y)\in\bC^3\ \big|\ |y|\le ct^{1/2}\right\}\ .
\end{align}
Moreover, $D^k\omega\in C([0,T_*),L^\infty(\bR^3)^3)$ and
\begin{align}
&\underset{t\in(0,T)}{\sup}\ \underset{y\in\cD_t}{\sup} \|\omega(\cdot,y,t)\|_{L^p} + \underset{t\in(0,T)}{\sup}\ \underset{y\in\cD_t}{\sup} \|\zeta(\cdot,y,t)\|_{L^p}\le M\|\omega_0\|_p \ ,
\\
&\underset{t\in(0,T)}{\sup}\ \underset{y\in\cD_t}{\sup} \|D^k\omega(\cdot,y,t)\|_{L^\infty} + \underset{t\in(0,T)}{\sup}\ \underset{y\in\cD_t}{\sup} \|D^k\zeta(\cdot,y,t)\|_{L^\infty}\le M\|D^k\omega_0\|_\infty \ .
\end{align}
Similar results hold for real solutions.
\end{theorem}

\noindent{\textit{Sketch of the proof.} \ Similarly as in the proof of Theorem~\ref{th:MainThmVel}, we construct an approximating sequence for the vorticity-velocity formulation
\begin{align}\label{eq:VelVorForm}
\partial_t\omega+(u\cdot \nabla)\omega= (\omega \cdot \nabla)u + \Delta \omega\ , \qquad \omega(\cdot,0)=\omega_0
\end{align}
as follows:
\begin{align*}
\partial_t\omega^{(n)} -\Delta \omega^{(n)} &=\omega^{(n-1)}\nabla u^{(n-1)} - u^{(n-1)}\nabla\omega^{(n-1)},\qquad \omega^{(n)}(0,x)=\omega_0\ ,
\\
u_j^{(n-1)}(x,t) &=c\int_{\mathbb{R}^3} \epsilon_{j,k,\ell}\ \partial_{y_k}\frac{1}{|x-y|}\omega_\ell^{(n-1)}(y,t)dy\ .
\end{align*}
We let $u^{(n)}+iv^{(n)}$ and $\omega^{(n)}+i\zeta^{(n)}$ be the analytic extension of the approximating sequence and let
\begin{align*}
&U^{(n)}(x,t)=u^{(n)}(x,\alpha t,t)\ , &&W^{(n)}(x,t)=w^{(n)}(x,\alpha t,t)\ ,
\\
&V^{(n)}(x,t)=v^{(n)}(x,\alpha t,t)\ , &&Z^{(n)}(x,t)=\zeta^{(n)}(x,\alpha t,t)\ ;
\end{align*}
then taking the $k$-th derivative (for the same reason as in the proof of Theorem~\ref{th:MainThmVel} we can assume $\omega_0\in C^\infty$) leads to the complexified iterations:
\begin{align*}
D^kW^{(n+1)}(x,t) &= e^{t\Delta}D^k\omega_0 +\int_0^t e^{(t-s)\Delta}D^k\left(W^{(n)}\nabla U^{(n)}\right)ds -\int_0^t e^{(t-s)\Delta}D^k\left(Z^{(n)}\nabla V^{(n)}\right)ds
\\
-\int_0^t & e^{(t-s)\Delta}D^k\left(U^{(n)}\nabla W^{(n)}\right)ds +\int_0^t e^{(t-s)\Delta}D^k\left(V^{(n)}\nabla Z^{(n)}\right)ds +\int_0^t e^{(t-s)\Delta}\alpha\cdot\nabla D^kZ^{(n+1)} ds
\\
D^kZ^{(n+1)}(x,t) &= \int_0^t e^{(t-s)\Delta}D^k\left(Z^{(n)}\nabla U^{(n)}\right)ds +\int_0^t e^{(t-s)\Delta}D^k\left(W^{(n)}\nabla V^{(n)}\right)ds
\\
-\int_0^t & e^{(t-s)\Delta}D^k\left(V^{(n)}\nabla W^{(n)}\right)ds -\int_0^t e^{(t-s)\Delta}D^k\left(U^{(n)}\nabla Z^{(n)}\right)ds -\int_0^t e^{(t-s)\Delta}\alpha\cdot\nabla D^kW^{(n+1)} ds
\end{align*}
where
\begin{align}
U_j^{(n)}(x,t) &=c\int_{\mathbb{R}^3} \epsilon_{j,k,\ell}\ \partial_{y_k}\frac{1}{|x-y|}W_\ell^{(n)}(y,t)dy\ , \label{eq:BiotSU}
\\
V_j^{(n)}(x,t) &=c\int_{\mathbb{R}^3} \epsilon_{j,k,\ell}\ \partial_{y_k}\frac{1}{|x-y|} Z_\ell^{(n)}(y,t)dy\ . \label{eq:BiotSV}
\end{align}
We claim that
\begin{align*}
&K_n:=\underset{t<T}{\sup}\ \|D^kW^{(n)}\|_{L^\infty}+\underset{t<T}{\sup}\ \|D^kZ^{(n)}\|_{L^\infty}\ ,\quad L_n:=\underset{t<T}{\sup}\ \|W^{(n)}\|_{L^p}+\underset{t<T}{\sup}\ \|Z^{(n)}\|_{L^p}\
\end{align*}
are all bounded by a constant only determined by $k$, $\|\omega_0\|_{L^\infty}$, $\|\omega_0\|_{L^p}$.

Let $W_x^{(n)}(y)$ denote the translation $W^{(n)}(x-y)$ and $B$ be the unit ball centered at 0. Then, from \eqref{eq:BiotSU}
it follows
\begin{align*}
\left|D^jU^{(n)}(x,t)\right| &\lesssim \int_B \frac{1}{|y|^2} \left|D^jW_x^{(n)}(y,t)\right| dy + \int_{B^c} \frac{1}{|y|^2} \left|D^jW_x^{(n)}(y,t)\right| dy
\\
&\lesssim \|D^jW_x^{(n)}\|_{{L^\infty}}\int_B |y|^{-2}dy + \|D^jW_x^{(n)}\|_{L^p}\||y|^{-2}\1_{B^c}\|_{L^{p'}}
\end{align*}
where we used that $p'>\frac{3}{2}$ (since $p<3$); hence, by Lemma~\ref{le:GNIneq}
\begin{align*}
\|D^jU^{(n)}\|_{L^\infty} &\lesssim \|D^jW^{(n)}\|_{{L^\infty}} + \|D^jW^{(n)}\|_{L^p}
\\
&\lesssim \|D^kW^{(n)}\|_\infty^{(j+\frac{d}{p})/(k+\frac{d}{p})}\|W^{(n)}\|_p^{(k-j)/(k+\frac{d}{p})} + \|D^kW^{(n)}\|_\infty^{j/(k+\frac{d}{p})}\|W^{(n)}\|_p^{1-j/(k+\frac{d}{p})}.
\end{align*}
Note that the map
\begin{align*}
(Tf)_j(x,t):=c\ \nabla\int_{\mathbb{R}^3} \epsilon_{j,k,\ell}\ \partial_{y_k}\frac{1}{|x-y|}f_\ell(y,t)dy
\end{align*}
defines a $\cC$-$\cZ$ operator. The $L^\infty$-estimates on the nonlinear terms are then as follows,
\begin{align*}
\left\|\int_0^t e^{(t-s)\Delta}D^k\left(W^{(n)}\nabla U^{(n)}\right)ds\right\|_\infty &\lesssim \sum_{i=0}^k \binom k i \int_0^t \|D^{k-i}W^{(n)}\|_\infty \left\|e^{(t-s)\Delta}|T D^iW^{(n)}|\right\|_\infty\ ds
\\
\lesssim \sum_{i=0}^k & \binom k i \int_0^t  \|D^{k-i}W^{(n)}\|_\infty \left(\|D^iW^{(n)}\|_\infty + \|D^iW^{(n)}\|_{L^p}\right)ds
\\
\lesssim t\sum_{i=0}^k \binom k i K_n^{(k-i+\frac{d}{p})/(k+\frac{d}{p})} & L_n^{i/(k+\frac{d}{p})} \left(K_n^{(i+\frac{d}{p})/(k+\frac{d}{p})}L_n^{(k-i)/(k+\frac{d}{p})} + K_n^{i/(k+\frac{d}{p})}L_n^{1-i/(k+\frac{d}{p})}\right)
\\
\lesssim t \sum_{i=0}^k \binom k i & \left(K_n^{(k+\frac{2d}{p})/(k+\frac{d}{p})} L_n^{k/(k+\frac{d}{p})} + K_n L_n\right)
\end{align*}
and
\begin{align*}
\left|\int_0^t \nabla e^{(t-s)\Delta}D^k\left(U^{(n)}W^{(n)}\right) ds\right|
&\lesssim\int_0^t \sum_{i=0}^k \binom k i \|D^{k-i}W^{(n)}\|_\infty \left\|\nabla G_{t-s}(x-\cdot)D^iU^{(n)}(\cdot)\right\|_{L^1}ds
\\
\lesssim \sum_{i=0}^k \binom k i & \int_0^t \|D^{k-i}W^{(n)}\|_\infty \left(\left\|G_{t-s}(x-\cdot)|\nabla |D^iU^{(n)}
(\cdot)||\right\|_{L^1}+ \|D^iU^{(n)}\|_\infty\right) ds
\\
\lesssim \sum_{i=0}^k \binom k i & \int_0^t \|D^{k-i}W^{(n)}\|_\infty \left(\left\|e^{(t-s)\Delta}|T D^iW^{(n)}|\right\|_\infty + \|D^iU^{(n)}\|_\infty\right) ds
\end{align*}
\begin{align*}
&\lesssim \sum_{i=0}^k \binom k i  \int_0^t \|D^{k-i}W^{(n)}\|_\infty \left(\|D^iW^{(n)}\|_{{L^\infty}} + \|D^iW^{(n)}\|_{L^p}\right)ds
\\
&\lesssim t \sum_{i=0}^k \binom k i  \left(K_n^{(k+\frac{2d}{p})/(k+\frac{d}{p})} L_n^{k/(k+\frac{d}{p})} + K_n L_n\right)\ .
\end{align*}
The $L^p$-estimates are summarized as
\begin{align*}
\left\|\int_0^t e^{(t-s)\Delta} W^{(n)}\nabla U^{(n)} ds\right\|_p \lesssim \int_0^t \|W^{(n)}\|_\infty \left\|e^{(t-s)\Delta}|T W^{(n)}|\right\|_p ds \lesssim t \cdot K_n^{\frac{d}{p}/(k+\frac{d}{p})}L_n^{(2k+\frac{d}{p})/(k+\frac{d}{p})}
\end{align*}
and
\begin{align*}
\left\|\int_0^t e^{(t-s)\Delta}U^{(n)}\nabla W^{(n)} ds\right\|_p &\lesssim \int_0^t \|U^{(n)}\|_\infty \left\|e^{(t-s)\Delta}\nabla W^{(n)}\right\|_p ds
\\
&\lesssim t \left(K_n^{\frac{d}{p}/(k+\frac{d}{p})}L_n^{k/(k+\frac{d}{p})} + L_n\right) K_n^{1/(k+\frac{d}{p})}L_n^{1-1/(k+\frac{d}{p})}.
\end{align*}
\vspace{-0.05in}
Thus, with $|\alpha|t^{1/2}<1/2$,
\begin{align*}
K_n &\lesssim \|D^k\omega_0\|_\infty + T\ 2^k \left(K_n^{\frac{d}{p}/(k+\frac{d}{p})} L_n^{k/(k+\frac{d}{p})} + L_n\right) K_n\ ,
\\
L_n &\lesssim \|\omega_0\|_p + T \left(K_n^{\frac{d}{p}/(k+\frac{d}{p})}L_n^{k/(k+\frac{d}{p})} + K_n^{(1+\frac{d}{p})/(k+\frac{d}{p})}L_n^{(k-1)/(k+\frac{d}{p})} + K_n^{1/(k+\frac{d}{p})}L_n^{1-1/(k+\frac{d}{p})} \right) L_n\ .
\end{align*}
In particular, if
\begin{align*}
T\le C\left(2^{k}\cdot M^2/(M-1) \left( \|\omega_0\|_p^{k/(k+\frac{d}{p})} \cdot \|D^k\omega_0\|_\infty^{\frac{d}{p}/(k+\frac{d}{p})} + \|\omega_0\|_p\right)\right)^{-1}
\end{align*}
\vspace{-0.05in}
and
\vspace{-0.05in}
\begin{align*}
T\le C\left(\frac{M^2}{M-1} \left(\|D^k\omega_0\|_\infty^{\frac{d/p}{k+\frac{d}{p}}}\|\omega_0\|_p^{\frac{k}{k+\frac{d}{p}}} + \|D^k\omega_0\|_\infty^{\frac{1+\frac{d}{p}}{k+\frac{d}{p}}}\|\omega_0\|_p^{\frac{k-1}{k+\frac{d}{p}}} + \|D^k\omega_0\|_\infty^{\frac{1}{k+\frac{d}{p}}}\|\omega_0\|_p^{\frac{k-1+\frac{d}{p}}{k+\frac{d}{p}}} \right) \right)^{-1},
\end{align*}
then an induction argument gives
\begin{align*}
K_n \le M \|D^k\omega_0\|_\infty\ ,\qquad L_n \le M \|\omega_0\|_p\ ,
\end{align*}
completing the proof.

%
%

\section{Asymptotic Zero Scaling Gap}\label{sec:Asym0ScGp}

In the first part of this section we compile some notions and ideas introduced in the prologue (including definitions \ref{def:1DSparse}-\ref{def:VecCompLSet} and the $Z_\alpha$-framework) with several results about sparseness of the regions of intense fluid activity whose mathematical setup was initiated in \citet{Grujic2001} and further developed and applied for various purposes in \citet{Grujic2013}, \citet{Farhat2017} and \citet{Bradshaw2019}, as well as present their level-$k$ generalizations based on the analyticity results derived in the previous section. In the second part we present a novel technique based on local-in-time dynamics of `chains of derivatives' in preparation not only for the proof of the main theorem but also for a more general theory of the blow-up (or the lack thereof) in the super-critical parabolic problems featuring a baseline \emph{a priori} bound, and then state and prove the main theorem.

\medskip

In the aforementioned articles, the ideas of the spatial intermittency were realized via the harmonic measure maximum principle for subharmonic functions as recorded, e.g., in \citet{Ahlfors1973} and \citet{Ransford1995}. Here we recall a result utilized in \citet{Bradshaw2019} ($h(z,\Omega,K)$ denotes the harmonic measure of $K$ with respect to $\Omega$, evaluated
at $z$).

\begin{proposition}[\citet{Ransford1995}]\label{prop:HarMaxPrin}
Let $\Omega$ be an open, connected set in $\bC$ such that its boundary has nonzero Hausdorff dimension, and let $K$ be a Borel subset of the boundary. Suppose that $u$ is a subharmonic function on $\Omega$ satisfying
\vspace{-0.05in}
\begin{align*}
u(z) &\le M\ , \quad \textrm{for }z\in\Omega
\\
\limsup_{z\to\zeta} u(z) &\le m\ , \quad \textrm{for }\zeta\in K.
\end{align*}
Then
\vspace{-0.1in}
\begin{align*}
u(z)\le m \, h(z,\Omega,K) + M(1-h(z,\Omega,K)) \ , \quad \textrm{for }z\in\Omega.
\end{align*}
\end{proposition}

The following extremal property of the harmonic measure in the unit disk $\mathbb{D}$
will be helpful in the calculations to follow.

\medskip

\begin{proposition}[\citet{Solynin1999}]
Let $\lambda$ be in $(0, 1)$, $K$ a closed subset of $[-1,1]$
such that $\mu (K) = 2\lambda$,
and suppose that the origin is in $\mathbb{D} \setminus K$. Then
\[
 h(0,\mathbb{D},K) \ge h(0,\mathbb{D}, K_\lambda) =
 \frac{2}{\pi} \arcsin \frac{1-(1-\lambda)^2}{1+(1-\lambda)^2}
\]
where $K_\lambda = [-1, -1+\lambda] \cup [1-\lambda, 1]$.
\end{proposition}

\medskip

As demonstrated in \citet{Farhat2017} and \citet{Bradshaw2019}, the concept of `escape time' allows for a more
streamlined presentation.

\begin{definition}
Let $u$ (resp. $\omega$) be in $C([0,T^*], L^\infty)$ where $T^*$ is the first possible blow-up time. A time $t\in(0,T^*)$ is an escape time if $\|u(s)\|_\infty>\|u(t)\|_\infty$ (resp. $\|\omega(s)\|_\infty>\|\omega(t)\|_\infty$) for any $s\in(t,T^*)$. (Local-in-time continuity of the $L^\infty$-norm implies there are continuum-many escape times.)
\end{definition}

Here we recall a regularity criterion based on the spatial intermittency of the velocity presented in \citet{Farhat2017} and an analogous result for the vorticity presented in \citet{Bradshaw2019}.
\vspace{-0.05in}
\begin{theorem}[\citet{Farhat2017} and \citet{Bradshaw2019}]\label{th:SparsityRegVel}
Let $u$ (resp. $\omega$) be in $C([0,T^*), L^\infty)$ where $T^*$ is the first possible blow-up time, and assume, in addition, that $u_0\in L^\infty$ (resp. $\omega_0\in L^\infty\cap L^2$). Let $t$ be an escape time of $u(t)$ (resp. $\omega(t)$), and suppose that there exists a temporal point
\begin{align*}
&\qquad\quad s=s(t)\in \left[t+\frac{1}{4c(M)^2\|u(t)\|_\infty^2},\ t+\frac{1}{c(M)^2\|u(t)\|_\infty^2}\right]
\\
&\left(\textrm{resp. } s=s(t)\in \left[t+\frac{1}{4c(M)\|\omega(t)\|_\infty},\ t+\frac{1}{c(M)\|\omega(t)\|_\infty}\right]\ \right)
\end{align*}
such that for any spatial point $x_0$, there exists a scale $\rho\le \frac{1}{2c(M)^2\|u(s)\|_\infty}$ $\left(\textrm{resp. }\rho\le \frac{1}{2c(M)\|\omega(s)\|_\infty^{\frac{1}{2}}}\right)$ with the property that the super-level set
\begin{align*}
&\qquad\quad V_\lambda^{j,\pm}=\left\{x\in\bR^d\ |\ u_j^\pm(x,s)>\lambda \|u(s)\|_\infty\right\}
\\
&\left(\textrm{resp. }\Omega_\lambda^{j,\pm}=\left\{x\in\bR^3\ |\ \omega_j^\pm(x,s)>\lambda \|\omega(s)\|_\infty\right\}\ \right)
\end{align*}
is 1D $\delta$-sparse around $x_0$ at scale $\rho$; here the index $(j,\pm)$ is chosen such that $|u(x_0,s)|=u_j^\pm(x_0,s)$ (resp. $|\omega(x_0,s)|=\omega_j^\pm(x_0,s)$), and the pair $(\lambda,\delta)$ is chosen such that the followings hold:
\begin{align*}
\lambda h+(1-h)=2\lambda\ ,\qquad h=\frac{2}{\pi}\arcsin\frac{1-\delta^2}{1+\delta^2}\ , \qquad \frac{1}{1+\lambda}<\delta<1\ .
\end{align*}
(Note that such pair exists and a particular example is that when $\delta=\frac{3}{4}$, $\lambda>\frac{1}{3}$.) Then, there exists $\gamma>0$ such that $u\in L^\infty((0,T^*+\gamma); L^\infty)$, i.e. $T^*$ is not a blow-up time.
\end{theorem}

With Theorem~\ref{th:MainThmVel} (setting $p=2$) and Theorem~\ref{th:MainThmVor} (setting $p=1$) we are able to generalize the above results as follows.
\vspace{-0.05in}
\begin{theorem}\label{th:SparsityRegDk}
Let $u$ (resp. $\omega$) be in $C([0,T^*), L^\infty)$ where $T^*$ is the first possible blow-up time, and assume, in addition, that $u_0\in L^\infty\cap L^2$ (resp. $\omega_0\in L^\infty\cap L^1$). Let $t$ be an escape time of $D^ku(t)$ (resp. $D^k\omega(t)$), and suppose that there exists a temporal point
\begin{align*}
&\qquad\quad s=s(t)\in \left[t+\frac{1}{4^{k+1}c(M,\|u_0\|_2)^2\|D^ku(t)\|_\infty^{2d/(2k+d)}},\ t+\frac{1}{4^kc(M,\|u_0\|_2)^2\|D^ku(t)\|_\infty^{2d/(2k+d)}}\right]
\\
&\left(\textrm{resp. } s=s(t)\in \left[t+\frac{1}{4^{k+1}c(M,\|\omega_0\|_1)\|D^k\omega(t)\|_\infty^{3/(k+3)}},\ t+\frac{1}{4^kc(M,\|\omega_0\|_1)\|D^k\omega(t)\|_\infty^{3/(k+3)}}\right]\ \right)
\end{align*}
such that for any spatial point $x_0$, there exists a scale $\rho\le \frac{1}{2^kc(M)\|D^ku(s)\|_\infty^{\frac{d}{2k+d}}}$ $\left(\textrm{resp. }\rho\le \frac{1}{2^kc(M)\|D^k\omega(s)\|_\infty^{\frac{3/2}{k+3}}}\right)$ with the property that the super-level set
\begin{align*}
&\qquad\quad V_\lambda^{j,\pm}=\left\{x\in\bR^d\ |\ (D^ku)_j^\pm(x,s)>\lambda \|D^ku(s)\|_\infty\right\}
\\
&\left(\textrm{resp. }\Omega_\lambda^{j,\pm}=\left\{x\in\bR^3\ |\ (D^k\omega)_j^\pm(x,s)>\lambda \|D^k\omega(s)\|_\infty\right\}\ \right)
\end{align*}
is 1D $\delta$-sparse around $x_0$ at scale $\rho$; here the index $(j,\pm)$ is chosen such that $|D^ku(x_0,s)|=(D^ku)_j^\pm(x_0,s)$ (resp. $|D^k\omega(x_0,s)|=(D^k\omega)_j^\pm(x_0,s)$), and the pair $(\lambda,\delta)$ is chosen as in Theorem~\ref{th:SparsityRegVel}. Then, there exists $\gamma>0$ such that $u\in L^\infty((0,T^*+\gamma); L^\infty)$, i.e. $T^*$ is not a blow-up time.
\end{theorem}

\begin{proof}
The proof is analogous to the proof of Theorem~\ref{th:SparsityRegVel}.
\end{proof}

The following lemma is the Sobolev $W^{-k,p}$-version of the volumetric sparseness results presented in \citet{Farhat2017} and \citet{Bradshaw2019},\
all vectorial analogs of the semi-mixedness lemma in \citet{Iyer2014}.

\begin{lemma}\label{le:HkSparse}
Let $r\in(0,1]$ and $f$ a bounded function from $\bR^d$ to $\bR^d$ with continuous partial derivatives
of order $k$. Then, for any tuple $(\zeta,\lambda, \delta, p)$, $\zeta\in\bN^d$ with $|\zeta|=k$, $\lambda\in (0,1)$, $\delta\in(\frac{1}{1+\lambda},1)$ and $p>1$, there exists $c^*(\zeta,\lambda,\delta,d,p)>0$ such that if
\begin{align}\label{eq:ZalphaCond}
\|D^\zeta f\|_{W^{-k,p}} \le c^*(\zeta,\lambda,\delta,d,p)\ r^{k+\frac{d}{p}} \|D^\zeta f\|_\infty
\end{align}
then each of the super-level sets
\begin{align*}
S_{\zeta,\lambda}^{i,\pm}=\left\{x\in\bR^d\ |\ (D^\zeta f)_i^\pm(x)>\lambda \|D^\zeta f\|_\infty\right\}\ , \qquad 1\le i\le d, \quad \zeta\in\bN^d
\end{align*}
is $r$-semi-mixed with ratio $\delta$.
\end{lemma}

\begin{proof}
Assume the opposite, i.e. there is either $S_{\zeta,\lambda}^{i,+}$ or $S_{\zeta,\lambda}^{i,-}$ which is not $r$-semi-mixed with the ratio $\delta$. 
Suppose -- without loss of generality -- it is $S_{\zeta,\lambda}^{1,+}$. Then there exists a spatial point $x_0$ such that
\begin{align}\label{eq:OppSparse}
\mu\left(S_{\zeta,\lambda}^{1,+}\cap B_r(x_0)\right) >  \varpi \delta r^d
\end{align}
where $\varpi$ denotes the volume of the unit ball in $\bR^d$. Let $\phi$ be a smooth, radially symmetric and radially decreasing function such that
\begin{align*}
\phi =\left\{\begin{array}{ccc} 1  &\textrm{ on } B_r(x_0) \\ 0 &\textrm{ on } \left(B_{(1+\eta)r}(x_0)\right)^c\end{array}\right. \qquad\textrm{and}\qquad |D^i \phi|\lesssim 2^{|i|}(\eta\cdot r)^{-|i|}\qquad\textrm{for all }|i|\le k\ .
\end{align*}
By duality
\begin{align}
\left|\int_{\bR^d} (D^\zeta f)_1(y)\phi(y)dy\right| &\lesssim \|D^\zeta f\|_{W^{-k,p}} \|\phi\|_{W^{k,q}}\ . \label{eq:L2locHolder}
\end{align}
For sufficiently small $\eta\cdot r$, an explicit calculation yields
\begin{align*}
\|\phi\|_{W^{k,q}} \lesssim \left(\int_{\bR^d} \sum_{|i|\le k} |D^i\phi|^q dy\right)^{1/q} \lesssim \left((1+\eta)^d-1\right)^{1/q} (\eta/2)^{-k} r^{-k+\frac{d}{q}}\ .
\end{align*}
To develop a contradictive result to \eqref{eq:ZalphaCond}, we write
\begin{align}\label{eq:DecompSparse}
\left|\int_{\bR^d} (D^\zeta f)_1(y)\phi(y)dy\right| &\ge \int_{\bR^d} (D^\zeta f)_1(y)\phi(y)dy \ge I-J-K
\end{align}
where
\begin{align*}
I &= \int_{S_{\zeta,\lambda}^{1,+}\cap B_{r}(x_0)} (D^\zeta f)_1(y)\phi(y)dy
\\
J &= \left|\int_{B_{r}(x_0)\setminus S_{\zeta,\lambda}^{1,+}} (D^\zeta f)_1(y)\phi(y)dy\right|
\\
K &= \left|\int_{B_{(1+\eta)r}(x_0)\setminus B_{r}(x_0)} (D^\zeta f)_1(y)\phi(y)dy\right|\ .
\end{align*}
Similar to the proof of Lemma~3.3 in \citet{Farhat2017}, the estimates \eqref{eq:L2locHolder} and \eqref{eq:DecompSparse} together with the opposite assumption about $\delta$-sparseness lead to
\begin{align}
\left((1+\eta)^d-1\right)^{1/q} (\eta/2)^{-k} r^{-k+\frac{d}{q}} \|D^\zeta f\|_{W^{-k,p}} \gtrsim \varpi r^d \|D^\zeta f\|_\infty\left(\lambda\delta +\delta -(1+\eta)^d\right);
\end{align}
in other words, for some constant $c$,
\begin{align}
\|D^\zeta f\|_{W^{-k,p}} > c\ \frac{(\eta/2)^{k}\left(\lambda\delta +\delta -(1+\eta)^d\right)}{\left((1+\eta)^d-1\right)^{1/q}}\ \varpi r^{k+\frac{d}{p}} \|D^\zeta f\|_\infty\ .
\end{align}
Since $\delta>\frac{1}{1+\lambda}$, if we set $(1+\eta)^d = \frac{\delta(1+\lambda)+1}{2}$, then
\begin{align*}
\|D^\zeta f\|_{W^{-k,p}} > c^*(\lambda, \delta, d, p)\ (\eta/2)^{k}r^{k+\frac{d}{p}} \|D^\zeta f\|_\infty
\end{align*}
where $c^*(\lambda, \delta, d, p)=\frac{1}{2} c \varpi (\delta(1+\lambda)-1)^{1/p}$ with $(1+\eta)^d = \frac{\delta(1+\lambda)+1}{2}$, producing a contradiction.
\end{proof}

This leads to the following \emph{a priori} sparseness result announced in the prologue.

\begin{theorem}\label{th:LerayZalpha}
Let $u$ be a Leray solution (a global-in-time weak solution satisfying the global energy inequality), and assume that $u$ is in $C((0,T^*), L^\infty)$ for some $T^*>0$. Then for any $t\in (0,T^*)$ the super-level sets
\begin{align*}
&\qquad\quad S_{\zeta,\lambda}^{i,\pm}=\left\{x\in\bR^d\ |\ (D^\zeta u)_i^\pm(x)>\lambda \|D^\zeta u\|_\infty\right\}\ , \qquad 1\le i\le d, \quad \zeta\in\bN^d
\\
&\left(\textrm{resp. } S_{\zeta,\lambda}^{i,\pm}=\left\{x\in\bR^3\ |\ (D^\zeta \omega)_i^\pm(x)>\lambda \|D^\zeta \omega\|_\infty\right\}\ , \qquad 1\le i\le 3, \quad \zeta\in\bN^3\ \right)
\end{align*}
are $d$-dimensional (resp. $3D$) $\delta$-sparse around any spatial point $x_0$ at scale
\begin{align}\label{eq:kNaturalScale}
r^*=c(\|u_0\|_2) \frac{1}{\|D^\zeta u(t)\|_\infty^{2/(2k+d)}}\quad \left(\textrm{resp. } r^*=c(\|u_0\|_2) \frac{1}{\|D^\zeta \omega(t)\|_\infty^{2/(2k+5)}}\ \right)
\end{align}
provided $r^* \in (0, 1]$ and with the same restrictions on $\lambda$ and $\delta$ as in the preceeding lemma.
In other words, $D^\zeta u(t)\in Z_\alpha(\lambda,\delta,c_0)$ with $\alpha=1/(k+d/2)$ (resp. $D^\zeta \omega(t)\in Z_\alpha(\lambda,\delta,c_0)$ with $\alpha=1/(k+5/2)$). Moreover, for any $p>2$, if we assume
\begin{align*}
u\in C((0,T^*), L^\infty)\cap L^\infty((0,T^*], L^p)
\end{align*}
then for any $t\in (0,T^*)$ the super-level sets $S_{\zeta,\lambda}^{i,\pm}$ are $d$-dimensional $\delta$-sparse around any spatial point $x_0$ at scale
\vspace{-0.06in}
\begin{align*}
r^*=c\left(\sup_{t<T^*}\|u(t)\|_{L^p}\right) \frac{1}{\|D^\zeta u(t)\|_\infty^{1/(k+d/p)}}.
\end{align*}
provided $r^* \in (0, 1]$ and the same conditions on $\lambda$ and $\delta$ as in Lemma~\ref{le:HkSparse} are satisfied, i.e. $D^\zeta u(t)\in Z_\alpha(\lambda,\delta,c_0)$ with $\alpha=1/(k+d/p)$.
\end{theorem}

\begin{proof}
Note that, for any $p\ge 2$ and $\zeta\in\bN^d$ with $|\zeta|=k$,
\begin{align*}
\|D^\zeta u(t)\|_{W^{-k,p}}\lesssim \|u(t)\|_{L^p}
\end{align*}
If $u\in L^\infty((0,T^*], L^p)$, in order to meet the assumption in Lemma~\ref{le:HkSparse}, it suffices to postulate
\begin{align*}
\sup_{t<T^*}\|u(t)\|_{L^p} \lesssim c^*(\zeta,\lambda,\delta,p)\ r^{k+\frac{d}{p}} \|D^\zeta f\|_\infty
\end{align*}
with $c^*$ given in \eqref{eq:ZalphaCond}, which forces the scale of sparseness required by the theorem. The proof for the vorticity is similar.
\end{proof}

In the rest of the paper, for simplicity, we will assume $D^\zeta=\partial_{x_1}^k$ (the proofs for other derivatives
of order $k$ are analogous). In addition, $\| \cdot \|$ will denote the $L^\infty$-norm.
The following four results, two theorems, a lemma and a corollary, as well as the last theorem, provide the foundation for a novel blow-up argument based on local-in-time dynamics of chains of derivatives.

\begin{theorem}[Ascending Chain]\label{le:AscendDer}
Let $u$ be a Leray solution initiated at $u_0$ and
suppose that 
\vspace{-0.06in}
\begin{align}\label{eq:AscDerOrd}
\frac{\|D^ju_0\|^{\frac{1}{j+1}}}{c^{\frac{j}{j+1}}(j!)^{\frac{1}{j+1}}} \le \frac{\|D^ku_0\|^{\frac{1}{k+1}}}{c^{\frac{k}{k+1}}(k!)^{\frac{1}{k+1}}}\, \qquad \forall \ell\le j\le k
\end{align}
where $c$, $\ell$ and $k$ satisfy
\vspace{-0.1in}
\begin{align}\label{eq:AscDerCond}
c \|u_0\|_2 \|u_0\|^{d/2-1} \frac{(\ell!)^{1/2}\ell}{(\ell/2)!} \lesssim (k!)^{1/(k+1)}.
\end{align}
Let $T^{1/2}\lesssim C(\|u_0\|,\ell,k)^{-1} \|D^ku_0\|^{-\frac{1}{k+1}}$; here $C(\|u_0\|,\ell,k)$ depends only on $u_0$, $\ell$, $k$ and a threshold $M$ introduced below; as we shall see later, the constant $c=c(k)$ in \eqref{eq:AscDerOrd} will be chosen according to the formation of the ascending chains in Lemma~\ref{cor:DerOrdConfig} and Corollary~\ref{cor:ScaleBound}, originally determined by the assumption~\eqref{eq:ParaAdjMaxP} in Theorem~\ref{le:DescendDer}. Then for any $\ell\le j \le k$ the complex solution of \eqref{eq:NSE1}-\eqref{eq:NSE3} has the following upper bounds:
\begin{align}
&\underset{t\in(0,T)}{\sup}\ \underset{y\in\cD_t}{\sup} \|D^ju(\cdot,y,t)\|_{L^\infty} + \underset{t\in(0,T)}{\sup}\ \underset{y\in\cD_t}{\sup} \|D^jv(\cdot,y,t)\|_{L^\infty} \notag
\\
&\qquad\le M\|D^ju_0\| + \left(j+ c\|u_0\|_2 \|u_0\|^{d/2-1} \frac{(\ell!)^{1/2}\ell}{(\ell/2)!}\right) \frac{c^{j+1}\ j!\ \|D^ku_0\|^{\frac{j+1}{k+1}}}{\left(c^{\frac{k}{k +1}}(k !)^{\frac{1}{k+1}}\right)^{j+2}} \label{eq:AscDerUpBdd}
\end{align}
where the multiplicative constant $M>1$ can be set as desired, and $\cD_t$ is given by \eqref{eq:AnalDom}. For the real solutions the above result becomes
\begin{align}
\underset{t\in(0,\tilde{T})}{\sup} \|D^ju(\cdot,t)\|_{L^\infty} \le \|D^ju_0\| + \left(j+ c\|u_0\|_2 \|u_0\|^{d/2-1} \frac{(\ell!)^{1/2}\ell}{(\ell/2)!}\right) \frac{c^{j+1}\ j!\ \|D^ku_0\|^{\frac{j+1}{k+1}}}{\left(c^{\frac{k}{k +1}}(k !)^{\frac{1}{k+1}}\right)^{j+2}} \label{eq:AscDerUpBddReal}
\end{align}
where $\tilde{T}$ does not depend on $M$.
\end{theorem}

\begin{proof}
For simplicity assume the system \eqref{eq:NSE1}-\eqref{eq:NSE3} is homogeneous, i.e. $f=0$. As in the proof of Theorem~\ref{th:MainThmVel} we have the iteration formulas \eqref{eq:IterationDU}-\eqref{eq:IterationDV}. The utility of the assumption \eqref{eq:AscDerOrd} is reducing the nonlinear effect at level-$k$  by replacing the standard
Gagliardo-Nirenberg interpolation (Lemma~\ref{le:GNIneq}) over the wide enough range of indexes, up to order $k$.

Let
\vspace{-0.1in}
\begin{align*}
K_n &:=\underset{t<T}{\sup}\ \|U^{(n)}\|_{L^2}+\underset{t<T}{\sup}\ \|V^{(n)}\|_{L^2}
\end{align*}
and
\vspace{-0.05in}
\begin{align*}
L_n^{(j)} &:= \underset{t<T}{\sup}\ \|D^jU^{(n)}\|_{L^\infty}+\underset{t<T}{\sup}\ \|D^jV^{(n)}\|_{L^\infty}\ , \qquad\forall\ \ell\le j \le k\ .
\end{align*}
We will demonstrate the idea of the proof on the nonlinear term $U^{(n)}\otimes U^{(n)}$ (the rest of the nonlinear terms can be estimated in a similar way) via an induction argument. For $j>2\ell$,
\begin{align*}
&\left\|\int_0^t e^{(t-s)\Delta}D^j(U^{(n)}\cdot \nabla)U^{(n)}ds\right\| \lesssim t^{\frac{1}{2}} \sum_{i=0}^j \binom j i \sup_{s<T}\|D^iU^{(n-1)}(s)\| \sup_{s<T}\|D^{j-i}U^{(n-1)}(s)\|
\\
&\quad \lesssim t^{\frac{1}{2}} \left(\sum_{0\le i\le \ell} + \sum_{\ell \le i\le j-\ell} + \sum_{j-\ell\le i\le j}\right) \left(\binom j i \sup_{s<T}\|D^iU^{(n-1)}(s)\| \sup_{s<T}\|D^{j-i}U^{(n-1)}(s)\| \right)
\\
&\quad \lesssim t^{\frac{1}{2}} \left(\sum_{\ell \le i\le j-\ell} \binom j i \sup_{s<T}\|D^iU^{(n-1)}(s)\| \sup_{s<T}\|D^{j-i}U^{(n-1)}(s)\| \right.
\\
&\quad \left.+ 2\sum_{0\le i\le \ell} \binom j i \left(\sup_{s<T}\|U^{(n-1)}(s)\|_2\right)^{\frac{\ell-i}{\ell+d/2}} \left(\sup_{s<T}\|D^\ell U^{(n-1)}(s)\|\right)^{\frac{i+d/2}{\ell+d/2}} \sup_{s<T}\|D^{j-i}U^{(n-1)}(s)\|\right)
\\
&\quad \lesssim t^{\frac{1}{2}} \left(\sum_{\ell \le i\le j-\ell} \binom j i L_{n-1}^{(i)}L_{n-1}^{(j-i)}+ 2\sum_{0\le i\le \ell} \binom j i K_{n-1}^{\frac{\ell-i}{\ell+d/2}}\left(L_{n-1}^{(\ell)}\right)^{\frac{i+d/2}{\ell+d/2}}L_{n-1}^{(j-i)} \right) := t^{\frac{1}{2}} \left(I+2J\right).
\end{align*}
Via the induction hypothesis and the assumption~\eqref{eq:AscDerOrd}, if $c$, $\ell$ and $k$ are chosen as in \eqref{eq:AscDerCond},
\begin{align*}
I & \lesssim \sum_{\ell \le i\le j-\ell} \binom j i \left(1+ \frac{i+ c\|u_0\|_2 \|u_0\|^{1/2} \frac{(\ell!)^{1/2}\ell}{(\ell/2)!}}{c^{-\frac{1}{k +1}}(k !)^{\frac{1}{k+1}}}\right) \frac{c^i\ i!}{\left(c^{\frac{k}{k +1}}(k !)^{\frac{1}{k+1}}\right)^{i+1}} \|D^ku_0\|^{\frac{i+1}{k+1}}
\\
&\qquad\qquad \times \left(1+ \frac{j-i+ c\|u_0\|_2 \|u_0\|^{1/2} \frac{(\ell!)^{1/2}\ell}{(\ell/2)!}}{c^{-\frac{1}{k +1}}(k !)^{\frac{1}{k+1}}}\right) \frac{c^{j-i}(j-i)!}{\left(c^{\frac{k}{k +1}}(k !)^{\frac{1}{k+1}}\right)^{j-i+1}} \|D^ku_0\|^{\frac{j-i+1}{k+1}}
\\
&\lesssim \sum_{\ell \le i\le j-\ell} \binom j i \frac{3\cdot c^i\ i!}{\left(c^{\frac{k}{k +1}}(k !)^{\frac{1}{k+1}}\right)^{i+1}} \|D^ku_0\|^{\frac{i+1}{k+1}} \cdot \frac{3\cdot c^{j-i}(j-i)!}{\left(c^{\frac{k}{k +1}}(k !)^{\frac{1}{k+1}}\right)^{j-i+1}} \|D^ku_0\|^{\frac{j-i+1}{k+1}}
\\
&\lesssim \frac{c^{j+1}\ j!(j-2\ell)}{\left(c^{\frac{k}{k +1}}(k !)^{\frac{1}{k+1}}\right)^{j+2}}\ c^{-1}\|D^ku_0\|^{\frac{j+2}{k+1}}.
\end{align*}
Also, by Lemma~\ref{le:GNIneq} and the assumption~\eqref{eq:AscDerOrd}
\begin{align*}
\|u_0\|\lesssim \|u_0\|_2^{\frac{\ell}{\ell+d/2}}\|D^\ell u_0\|^{\frac{d/2}{\ell+d/2}} \lesssim \|u_0\|_2^{\frac{\ell}{\ell+d/2}}  \left(\frac{c^{\frac{\ell}{\ell +1}}(\ell !)^{\frac{1}{\ell+1}}}{c^{\frac{k}{k +1}}(k !)^{\frac{1}{k+1}}}\right)^{\frac{d/2}{\ell +d/2}(\ell+1)} \|D^ku_0\|^{\frac{\ell+1}{k+1} \frac{d/2}{\ell +d/2}}.
\end{align*}
Without loss of generality, we can assume $\|D^ku_0\|\lesssim k!\|u_0\|^{k+1}$. (If this is not satisfied at the initial time, perform the (spatially) complexified local-in-time algorithm in $L^\infty$ 
resulting in the interval of existence $(0, \tau)$ where $\tau \approx \frac{1}{\|u_0\|^2}$ and -- consequently -- the lower bound on the radius of spatial analyticity of the solution at
the end point $\tau$ of the
order of $\frac{1}{\|u_0\|}$. If $\|u(\tau)\| \ge \|u_0\|$ then the inequality will hold at $t=\tau$ (by the generalized Cauchy formula) and we simply reset the initial time. If the local-in-time iterates at the end points never rise above the initial level, the solution will stay bounded in $L^\infty$ for all times). Then -- under the assumption~\eqref{eq:AscDerCond} -- the above estimate implies
\begin{align*}
\|u_0\|&\lesssim \|u_0\|_2  \left(c^{\frac{\ell}{\ell +1}}(\ell !)^{\frac{1}{\ell+1}}\big/ c^{\frac{k}{k +1}}(k !)^{\frac{1}{k+1}}\right)^{\frac{d/2}{\ell +d/2}(\ell+1)} \left(k!\ \|u_0\|^k\right)^{\frac{(d/2-1)\ell}{\ell +d/2} \frac{1}{k+1}} \|D^ku_0\|^{\frac{1}{k+1}}
\\
&\lesssim \|u_0\|_2  \left(c^{\frac{1}{k+1}-\frac{1}{\ell+1}}\right)^{\frac{d/2(\ell+1)}{\ell+d/2}} (\ell !)^{\frac{d/2}{\ell+d/2}} (k !)^{-\frac{1}{k+1}} \|u_0\|^{\frac{(d/2-1)\ell}{\ell+d/2}\frac{k}{k+1}} \|D^ku_0\|^{\frac{1}{k+1}} \lesssim c^{-\frac{d/2}{\ell+1}}\|D^ku_0\|^{\frac{1}{k+1}}.
\end{align*}
Consequently, following the proof of Theorem~\ref{th:MainThmVel} (with $k=0$ and $p=2$), if $t^{1/2}\lesssim c\|D^ku_0\|^{-\frac{1}{k+1}}$ one can show
\vspace{-0.06in}
\begin{align*}
\sup_{s<t}\|U_n(s)\|_2 \lesssim \|u_0\|_2 \qquad\textrm{for all }n.
\end{align*}
Thus, by induction and \eqref{eq:AscDerOrd}-\eqref{eq:AscDerCond} we deduce
\begin{align*}
J &\lesssim \sum_{0\le i\le \ell} \binom j i \|u_0\|_2^{\frac{\ell-i}{\ell+d/2}} \left(2+ \frac{\ell}{(k !)^{\frac{1}{k+1}}}\right)^{\frac{i+d/2}{\ell+d/2}} \left(\frac{c^{\frac{\ell}{\ell +1}}(\ell !)^{\frac{1}{\ell+1}}}{c^{\frac{k}{k +1}}(k !)^{\frac{1}{k+1}}}\right)^{\frac{i+d/2}{\ell +d/2}(\ell+1)} \|D^ku_0\|^{\frac{\ell+1}{k+1} \frac{i+d/2}{\ell +d/2}}
\\
&\qquad\qquad\qquad \times \left(2+ \frac{j-i}{(k !)^{\frac{1}{k+1}}}\right) \frac{c^{j-i}(j-i)!}{\left(c^{\frac{k}{k +1}}(k !)^{\frac{1}{k+1}}\right)^{j-i+1}}\ \|D^ku_0\|^{\frac{j-i+1}{k+1}}
\\
&\lesssim \sum_{0\le i\le \ell} 9\cdot \|u_0\|_2^{\frac{\ell-i}{\ell+d/2}} \frac{c^j\ j!}{\left(c^{\frac{k}{k +1}}(k !)^{\frac{1}{k+1}}\right)^{j+2}}\ \|D^ku_0\|^{\frac{j+2}{k+1}}
\\
&\qquad\qquad\qquad \times \frac{\left(c^{\frac{k}{k +1}}(k!)^{\frac{1}{k+1}}\right)^{i+1}}{c^i\ i!} \left(\frac{c^{\frac{\ell}{\ell +1}}(\ell !)^{\frac{1}{\ell+1}}}{c^{\frac{k}{k +1}}(k !)^{\frac{1}{k+1}}}\right)^{\frac{i+d/2}{\ell +d/2}(\ell+1)} \|D^ku_0\|^{\frac{\ell-i}{\ell +d/2}\frac{d/2-1}{k+1}}.
\end{align*}
Again, without loss of generality, assume $\|D^ku_0\|\lesssim k!\|u_0\|^{k+1}$. Then
\begin{align*}
J&\lesssim   \|u_0\|_2\ \frac{c^j\ j!}{\left(c^{\frac{k}{k +1}}(k !)^{\frac{1}{k+1}}\right)^{j+2}}\ \|D^ku_0\|^{\frac{j+2}{k+1}}
\\
& \qquad \times \sum_{0\le i\le \ell} \frac{\left(c^{\frac{k}{k +1}}(k!)^{\frac{1}{k+1}}\right)^{i+1}}{c^i\ i!} \left(\frac{c^{\frac{\ell}{\ell +1}}(\ell !)^{\frac{1}{\ell+1}}}{c^{\frac{k}{k +1}}(k !)^{\frac{1}{k+1}}}\right)^{\frac{i+d/2}{\ell +d/2}(\ell+1)} \|D^ku_0\|^{\frac{\ell-i}{\ell +d/2}\frac{d/2-1}{k+1}}
\\
&\lesssim \frac{c^j\ j!\ \|u_0\|_2}{\left(c^{\frac{k}{k +1}}(k !)^{\frac{1}{k+1}}\right)^{j+2}}\ \|D^ku_0\|^{\frac{j+2}{k+1}}
\\
& \qquad \times \sum_{0\le i\le \ell}  c^{\frac{d/2(\ell-i)}{\ell+d/2}} \left(c^{\frac{k}{k +1}}(k!)^{\frac{1}{k+1}}\right)^{\frac{(d/2-1)(i-\ell)}{\ell +d/2}} \left(k!\ \|u_0\|^k\right)^{\frac{d/2-1}{k+1}\frac{\ell-i}{\ell +d/2}} \frac{(\ell!)^{\frac{i+d/2}{\ell +d/2}}}{i!}
\\
&\lesssim \frac{c^j\ j!\ \|u_0\|_2}{\left(c^{\frac{k}{k +1}}(k !)^{\frac{1}{k+1}}\right)^{j+2}}\ \|D^ku_0\|^{\frac{j+2}{k+1}} \sum_{0\le i\le \ell} c^{\frac{d/2(\ell-i)}{\ell+d/2}-\frac{k}{k+1} \frac{(d/2-1)(\ell-i)}{\ell +d/2}} \|u_0\|^{\frac{(d/2-1)(\ell-i)}{\ell +d/2}} \frac{(\ell!)^{\frac{i+d/2}{\ell +d/2}}}{i!}
\\
&\lesssim c\|u_0\|^{d/2-1} \frac{(\ell!)^{1/2}\ell}{(\ell/2)!}\cdot \frac{c^j\ j!\ \|u_0\|_2}{\left(c^{\frac{k}{k +1}}(k !)^{\frac{1}{k+1}}\right)^{j+2}}\ \|D^ku_0\|^{\frac{j+2}{k+1}}.
\end{align*}
To sum up we have shown that if \eqref{eq:AscDerCond} is satisfied, for any $j>2\ell$,
\begin{align}\label{eq:NonlinInduct}
I+2J &\lesssim \left(j+ c\|u_0\|_2 \|u_0\|^{d/2-1} \frac{(\ell!)^{1/2}\ell}{(\ell/2)!}\right) \frac{c^{j+1}\ j!\ c^{-1} \|D^ku_0\|^{\frac{j+2}{k+1}}}{\left(c^{\frac{k}{k +1}}(k !)^{\frac{1}{k+1}}\right)^{j+2}}.
\end{align}
Now if $\ell \le j\le 2\ell$,
\begin{align*}
&\left\|\int_0^t e^{(t-s)\Delta}D^j(U^{(n)}\cdot \nabla)U^{(n)}ds\right\| \lesssim t^{\frac{1}{2}} \left(\sum_{j-\ell \le i\le \ell} + 2\sum_{0\le i\le j-\ell} \right)\binom j i L_{n-1}^{(i)}L_{n-1}^{(j-i)}
\\
&\quad \lesssim t^{\frac{1}{2}} \left(\sum_{j-\ell \le i\le \ell} \binom j i K_{n-1}^{\frac{2\ell-j}{\ell+d/2}}\left(L_{n-1}^{(\ell)}\right)^{\frac{j+d}{\ell+d/2}} + 2\sum_{0\le i\le j-\ell} \binom j i K_{n-1}^{\frac{\ell-i}{\ell+d/2}}\left(L_{n-1}^{(\ell)}\right)^{\frac{i+d/2}{\ell+d/2}}L_{n-1}^{(j-i)} \right)
\\
&\quad := t^{\frac{1}{2}} \left(I+2J\right).
\end{align*}
Similarly, via the induction hypothesis and the assumption~\eqref{eq:AscDerCond},
\begin{align*}
I&\lesssim \sum_{j-\ell \le i\le \ell} \binom j i \|u_0\|_2^{\frac{2\ell-j}{\ell+d/2}}\left(2+ \frac{\ell}{(k !)^{\frac{1}{k+1}}}\right)^{\frac{j+d}{\ell+d/2}} \left(\frac{c^{\frac{\ell}{\ell +1}}(\ell !)^{\frac{1}{\ell+1}}}{c^{\frac{k}{k +1}}(k !)^{\frac{1}{k+1}}}\right)^{\frac{j+d}{\ell +d/2}(\ell+1)} \|D^ku_0\|^{\frac{\ell+1}{k+1} \frac{j+d}{\ell +d/2}}
\\
&\lesssim \|u_0\|_2\sum_{j-\ell \le i\le \ell} \binom j i \left(c^{\frac{1}{k+1}-\frac{1}{\ell+1}}\right)^{\frac{\ell+1}{\ell+d/2}(j+d)} \left(\frac{\ell!}{(k !)^{\frac{\ell+1}{k+1}}}\right)^{\frac{j+d}{\ell+d/2}} \|D^ku_0\|^{\frac{\ell+1}{k+1} \frac{j+d}{\ell +d/2}}
\end{align*}
\begin{align*}
&\lesssim \|u_0\|_2\sum_{j-\ell \le i\le \ell} \binom j i c^{-\frac{j+d}{\ell+d/2}} \left(\frac{\ell!}{(k !)^{\frac{\ell+1}{k+1}}}\right)^{\frac{j+d}{\ell+d/2}} \left(k!\ \|u_0\|^k\right)^{\frac{\ell+1}{k+1} \frac{j+d}{\ell +d/2}-\frac{j+2}{k+1}} \|D^ku_0\|^{\frac{j+2}{k+1}}
\\
&\lesssim \|u_0\|^{\frac{j+d}{\ell +d/2}(\ell+1)-(j+2)}\sum_{j-\ell \le i\le \ell} \frac{(\ell!)^{\frac{j+d}{\ell+d/2}}}{i!\ (j-i)!} \ \frac{j!\ \|u_0\|_2}{c^{\frac{j+d}{\ell+d/2}} (k !)^{\frac{j+2}{k+1}}}\ \|D^ku_0\|^{\frac{j+2}{k+1}}
\end{align*}
and the same argument leads to
\begin{align*}
J&\lesssim \frac{c^j\ j!\ \|u_0\|_2}{\left(c^{\frac{k}{k +1}}(k !)^{\frac{1}{k+1}}\right)^{j+2}}\ \|D^ku_0\|^{\frac{j+2}{k+1}} \sum_{0\le i\le j-\ell} c^{\frac{d/2(\ell-i)}{\ell+d/2}-\frac{k}{k+1} \frac{(d/2-1)(\ell-i)}{\ell +d/2}} \|u_0\|^{\frac{(d/2-1)(\ell-i)}{\ell +d/2}} \frac{(\ell!)^{\frac{i+d/2}{\ell +d/2}}}{i!}
\\
&\lesssim c\|u_0\|^{d/2-1} \frac{(\ell!)^{1/2}(j-\ell)}{(\ell/2)!}\cdot \frac{c^j\ j!\ \|u_0\|_2}{\left(c^{\frac{k}{k +1}}(k !)^{\frac{1}{k+1}}\right)^{j+2}}\ \|D^ku_0\|^{\frac{j+2}{k+1}}.
\end{align*}
Then, \eqref{eq:NonlinInduct} still holds for $\ell \le j\le 2\ell$, with the assumption~\eqref{eq:AscDerCond}.
Hence, as long as $|\alpha|t^{1/2}\lesssim 1-M^{-1}$,
\begin{align*}
\|D^jU_n(t)\| &\le M\|D^ju_0\| + T^{1/2} \left(j+c\|u_0\|_2 \|u_0\|^{d/2-1} \frac{(\ell!)^{1/2}\ell}{(\ell/2)!}\right) \frac{c^{j+1}\ j!\ c^{-1}\|D^ku_0\|^{\frac{j+2}{k+1}}}{M^{-1}\left(c^{\frac{k}{k +1}}(k !)^{\frac{1}{k+1}}\right)^{j+2}},
\end{align*}
and if--in addition--$T^{1/2}\lesssim c(M-1)\|D^ku_0\|^{-\frac{1}{k+1}}$, then for all $\ell< j\le k$ and for all $n$
\begin{align*}
\sup_{s<T}\|D^jU_n(s)\| &\le M\|D^ju_0\| + \left(j+c\|u_0\|_2 \|u_0\|^{d/2-1} \frac{(\ell!)^{1/2}\ell}{(\ell/2)!}\right) \frac{c^{j+1}\ j!\ \|D^ku_0\|^{\frac{j+1}{k+1}}}{\left(c^{\frac{k}{k +1}}(k !)^{\frac{1}{k+1}}\right)^{j+2}}.
\end{align*}
Similarly, for all $\ell< j\le k$ and for all $n$
\begin{align*}
\sup_{s<T}\|D^jV_n(s)\| &\le \left(j+c\|u_0\|_2 \|u_0\|^{1/2} \frac{(\ell!)^{1/2}\ell}{(\ell/2)!}\right) \frac{c^{j+1}\ j!\ \|D^ku_0\|^{\frac{j+1}{k+1}}}{\left(c^{\frac{k}{k +1}}(k !)^{\frac{1}{k+1}}\right)^{j+2}}.
\end{align*}
Finally, a standard convergence argument yields \eqref{eq:AscDerUpBdd}.
\end{proof}

\medskip

\begin{theorem}[Descending Chain]\label{le:DescendDer}
Let $u$ be a Leray solution of \eqref{eq:NSE1}-\eqref{eq:NSE3} initiated at $u_0$, and $\epsilon > 0$.
Suppose $\ell$ is sufficiently large such that $\|u_0\|\lesssim (1+\epsilon)^{\ell}$. For a fixed $k\ge\ell$, suppose that
\begin{align}\label{eq:DescDerOrd}
\frac{\|D^ku_0\|^{\frac{1}{k+1}}}{c^{\frac{k}{k+1}}(k!)^{\frac{1}{k+1}}} \ge \frac{\|D^ju_0\|^{\frac{1}{j+1}}}{c^{\frac{j}{j+1}}(j!)^{\frac{1}{j+1}}} \ , \qquad \forall\ j\ge k
\end{align}
for a suitable constant $c=c(k)$ which also satisfies
\begin{align}\label{eq:ParaAdjMaxP}
\lambda h^* + \exp\left((2e/\eta)(1+\epsilon)^{\ell/k} c^{\frac{1}{k+1}}\right) (1-h^*) \le \mu
\end{align}
where
$h^*=\frac{2}{\pi} \arcsin \frac{1-\delta^{2/d}}{1+\delta^{2/d}}$, $(1+\eta)^d=\frac{\delta(1+\lambda)+1}{2}$ and $\mu$ is a positive constant. Then there exist $T_*>t_0$ and a constant $\mu_*$ such that
\begin{align*}
\|D^ku(s)\|\le \mu_* \|D^ku_0\| \ , \qquad \forall\ t_0< s\le T_*\ .
\end{align*}
Here $\mu_*$ is smaller than the threshold $M$ for $D^ku$ given in Theorem~\ref{th:MainThmVel} (and could be less than 1 with proper choices of $c$ and $\mu$). A particular consequence (with an argument by contradiction) of this result is
that--for sufficently small values of $\mu$--\eqref{eq:DescDerOrd} can not coexist with \eqref{eq:ParaAdjMaxP}.
\end{theorem}

\begin{proof}
Pick $k_*$ such that \eqref{eq:AscDerCond} holds for $\ell=k$ and $k=k_*$. According to Theorem~\ref{th:MainThmVel}, there exists
\begin{align*}
T_*=C(M)\|u_0\|_2^2 \cdot \min_{k\le j\le k_*}\! 4^{-j} \|D^{j}u_0\|^{-\frac{d}{j+d/2}}
\end{align*}
such that
\begin{align*}
\sup_{t_0<s<t_0+T_*}\! \|D^{j}u(s)\| \le M\|D^{j}u_0\|\ , \qquad \forall\ k\le j\le k_*\ ,
\end{align*}
i.e. the uniform time span for the real solutions from $k$-th level to $k_*$-th level.

We first consider the case in which the order of `the tail of \eqref{eq:DescDerOrd} after $k_*$' continues for all $s$ up to $t_0+T_*$, that is assuming, for any $t_0<s<t_0+T_*$,
\begin{align}\label{eq:DescDerOrdAlls}
\frac{\|D^ku(s)\|^{\frac{1}{k+1}}}{c^{\frac{k}{k+1}}(k!)^{\frac{1}{k+1}}} \ge \frac{\|D^ju(s)\|^{\frac{1}{j+1}}}{c^{\frac{j}{j+1}}(j!)^{\frac{1}{j+1}}} \ , \qquad \forall\ j\ge k_*\ .
\end{align}
Fix an $x_0\in\bR^d$. Following the assumption~\eqref{eq:DescDerOrdAlls}, if $z\in B_{r_s}(x_0,0)\subset\bC^d$ with
\begin{align}\label{eq:NaScSparse}
r_s= \left(\frac{\sup_{s}\|u(s)\|_{L^2}}{c^*(\zeta,\lambda,\delta,d,p)}\right)^{\frac{1}{k+d/2}} \|D^ku(s)\|^{-\frac{1}{k+d/2}} \approx (\eta/2)^{-1}\|D^ku(s)\|^{-\frac{1}{k+d/2}}
\end{align}
(where $c^*$ is given in \eqref{le:HkSparse} and such choice for the radius becomes natural as we apply Theorem~\ref{th:LerayZalpha} later and Proposition~\ref{prop:HarMaxPrin} at the end of the proof) the complex extension of $D^ku_s(x)$ at any spatial point $x_0$ satisfies (for $z\neq x_0$)
\begin{align*}
\left|D^ku_s(z)\right| &\le \left(\sum_{0\le i\le k_*-k} + \sum_{i>k_*-k}\right)\frac{\left|D^{k+i}u_s(x_0)\right|}{i!} |z-x_0|^i =: \cI_s(z) + \cJ_s(z)
\end{align*}
where (with \eqref{eq:DescDerOrdAlls} in mind )
\begin{align*}
\cJ_s(z) &\le \sum_{i>k_*-k} \left(\frac{\|D^{k+i}u(s)\|^{\frac{1}{k+i+1}}}{\left(c^{k+i}(k+i)! \right)^{\frac{1}{k+i+1}}} \right)^{k+i+1} \frac{(k+i)!}{i!}\ c^{k+i}\ |z-x_0|^i
\\
&\le \sum_{i>k_*-k} \left(\frac{\|D^{k}u(s)\|^{\frac{1}{k+1}}}{c^{\frac{k}{k+1}}\left(k! \right)^{\frac{1}{k+1}}}\right)^{k+i+1} \frac{(k+i)!}{i!}\ c^{k+i}\ |z-x_0|^i\ , \qquad \forall\ t_0<s<t_0+T_*\ .
\end{align*}
Thus, for any $s<t_0+T_*$, (with \eqref{eq:NaScSparse} in mind)
\begin{align*}
\sup_{z\in B_{r_s}(x_0,0)} \cJ_s(z) &\le \|D^ku(s)\| \sum_{i>k_*-k} \left(\frac{\|D^{k}u(s)\|^{\frac{1}{k+1}}}{c^{\frac{k}{k+1}}\left(k! \right)^{\frac{1}{k+1}}}\right)^i \frac{(k+i)!}{k!\ i!}\ c^i\ r_s^i
\\
&\le \|D^ku_s\| \sum_{i>k_*-k} \frac{(k+i)!}{k!\ i!} \left(\frac{c\ r_s\|D^{k}u_s\|^{\frac{1}{k+1}}}{c^{\frac{k}{k+1}}\left(k! \right)^{\frac{1}{k+1}}}\right)^i
\\
&\le \|D^ku_s\| \sum_{i>k_*-k} \frac{(k+i)!}{k!\ i!} \left(\frac{c^{\frac{1}{k+1}}\|D^{k}u_s\|^{\frac{d/2-1}{(k+1)(k+d/2)}}} {(\eta/2)\left(k!\right)^{\frac{1}{k+1}}}\right)^{i}\ .
\end{align*}
By Theorem~\ref{th:MainThmVel} and the assumption \eqref{eq:DescDerOrd}, if $s\le t_0+T_*$ then
\begin{align*}
\sup_{z\in B_{r_s}(x_0,0)} \cI_s(z) &\le M \sum_{0\le i\le k_*-k} \frac{\|D^{k+i}u_0\|}{i!}\ r_s^i
\\
&\le M \sum_{0\le i\le k_*-k} \left(\frac{\|D^{k}u_0\|^{\frac{1}{k+1}}}{c^{\frac{k}{k+1}}\left(k! \right)^{\frac{1}{k+1}}}\right)^{k+i+1} \frac{(k+i)!}{i!}\ c^{k+i}\ r_s^i
\\
&\le M \|D^{k}u_0\|\sum_{0\le i\le k_*-k} \frac{(k+i)!}{k!\ i!} \left(\frac{c\ r_s\|D^{k}u_0\|^{\frac{1}{k+1}}}{c^{\frac{k}{k+1}}\left(k! \right)^{\frac{1}{k+1}}}\right)^i \ .
\end{align*}
We will complete the proof by way of contradiction. Suppose there exists an $t<t_0+T_*$ such that $\|D^ku(t)\|> \mu_* \|D^ku_0\|$, then $r_t\le \mu_*^{-\frac{1}{k+d/2}}r_0$ and
\begin{align*}
\sup_{z\in B_{r_t}(x_0,0)} \cI_t(z) &\le M \|D^{k}u_0\| \sum_{0\le i\le k_*-k} \frac{(k+i)!}{k!\ i!} \left(\frac{c^{\frac{1}{k+1}}\|D^{k}u_0\|^{\frac{d/2-1}{(k+1)(k+d/2)}}} {\mu_*^{\frac{1}{k+d/2}}(\eta/2)\left(k!\right)^{\frac{1}{k+1}}}\right)^{i}\ .
\end{align*}
Combining the estimates for $\cJ_t(z)$ and $\cI_t(z)$ yields
\begin{align*}
\sup_{z\in B_{r_t}(x_0,0)} \left|D^ku_t(z)\right|  &\le M\|D^ku_0\| \sum_{i>k_*-k} \frac{(k+i)!}{k!\ i!} \left(\frac{c^{\frac{1}{k+1}}\left(M\|D^ku_0\|\right)^{\frac{d/2-1}{(k+1)(k+d/2)}}} {(\eta/2)\left(k!\right)^{\frac{1}{k+1}}}\right)^{i}
\\
&\qquad + M \|D^{k}u_0\| \sum_{0\le i\le k_*-k} \frac{(k+i)!}{k!\ i!} \left(\frac{c^{\frac{1}{k+1}}\|D^{k}u_0\|^{\frac{d/2-1}{(k+1)(k+d/2)}}} {\mu_*^{\frac{1}{k+d/2}}(\eta/2)\left(k!\right)^{\frac{1}{k+1}}}\right)^{i}
\\
&\le M\|D^ku_0\| \sum_{i\ge 0} \frac{(k+i)!}{k!\ i!} \left(\frac{c^{\frac{1}{k+1}}\left(M\|D^ku_0\|\right)^{\frac{d/2-1}{(k+1)(k+d/2)}}} {\mu_*^{\frac{1}{k+d/2}}(\eta/2)\left(k!\right)^{\frac{1}{k+1}}}\right)^{i}
\\
&\le M\|D^ku_0\| \left(1-\frac{c^{\frac{1}{k+1}}\left(M\|D^{k}u_0\|\right)^{\frac{d/2-1}{(k+1)(k+d/2)}}} {\mu_*^{\frac{1}{k+d/2}}(\eta/2)\left(k!\right)^{\frac{1}{k+1}}}\right)^{-k-1}.
\end{align*}
Without loss of generality, one can assume $u_0$ evolves from a negative temporal point so that $\|D^ku_0\|\lesssim k!\|u_0\|^{k+1}$;
thus
\begin{align*}
\sup_{z\in B_{r_t}(x_0,0)} \left|D^ku_s(z)\right| &\le M\|D^ku_0\| \left(1-\frac{c^{\frac{1}{k+1}}\left(M\ k!\ \|u_0\|^k\right)^{\frac{d/2-1}{(k+1)(k+d/2)}}} {\mu_*^{\frac{1}{k+d/2}}(\eta/2)\left(k!\right)^{\frac{1}{k+1}}}\right)^{-k-1}.
\end{align*}
Since the above estimates hold for all $x_0$, if $\|u_0\|\lesssim (1+\epsilon)^{\ell}$, $M, \mu_*\approx1$ and $k$ is sufficiently large,
\begin{align*}
&\underset{y\in B_{r_t}(0)}{\sup} \|D^ku(\cdot,y,t)\|_{L^\infty} + \underset{y\in B_{r_t}(0)}{\sup} \|D^kv(\cdot,y,t)\|_{L^\infty}
\\
&\quad \le M \|D^ku_0\| \exp\left(\frac{c^{\frac{1}{k+1}}\left(k!\ \|u_0\|^k\right)^{\frac{d/2}{(k+1)(k+d/2)}}} {(\eta/2)\left(k!\right)^{\frac{1}{k+1}}/(k+1)}\right) \lesssim M \exp\left((2e/\eta)(1+\epsilon)^{\frac{\ell}{k}} c^{\frac{1}{k+1}}\right) \|D^ku_0\|\ .
\end{align*}
By Theorem~\ref{th:LerayZalpha}, for any spatial point $x_0$ there exists a direction $\nu$ along which the super-level set
\begin{align*}
S_{k,\lambda}^{i,\pm}=\left\{x\in\bR^d\ |\ (D^k u_t)_i^\pm(x)>\lambda \|D^k u_t\|_\infty\right\}
\end{align*}
is 1-D $\delta^{1/d}$-sparse at scale $r_t$ given in \eqref{eq:NaScSparse}. Note that the results in Proposition~\ref{prop:HarMaxPrin} are scaling invariant and--for simplicity--assume $r_t=1$ and $\nu$ is a unit vector. Define
\begin{align*}
K=\overline{(x_0-\nu,x_0+\nu)\setminus S_{k,\lambda}^{i,\pm}}\ .
\end{align*}
Then--by sparseness--$|K|\ge 2(1-\delta^{1/d})$. If $x_0\in K$, the result follows immediately. If $x_0\notin K$, then by Proposition~\ref{prop:HarMaxPrin} and the above estimate for $D^ku_t(z)$,
\begin{align*}
|D^ku_t(x_0)|&\le \lambda \|D^k u_t\|_\infty\ h^* + \sup_{z\in B_{r_t}(x_0,0)} \left|D^ku_t(z)\right| (1-h^*)
\\
&\le \lambda M\|D^k u_0\|_\infty\ h^* + M\exp\left((2e/\eta)(1+\epsilon)^{\ell/k}c^{\frac{1}{k+1}}\right)\|D^k u_0\|_\infty (1-h^*)
\end{align*}
where $h^*=\displaystyle \frac{2}{\pi} \arcsin \frac{1-\delta^{2/d}}{1+\delta^{2/d}}$. Hence, if condition~\eqref{eq:ParaAdjMaxP} is satisfied, we observe a contradiction (from the above result) that $\|D^ku(t)\|\le \mu_* \|D^ku_0\|$ with $\mu_*=M\mu$.

Now we consider the opposite case, that is the order \eqref{eq:DescDerOrdAlls} stops at some temporal points $t_\tau<t_0+T_*$ for some indexes $k_\tau>k_*$. For convenience we define
\begin{align}
\cR(j,c,t)&:=\frac{\|D^ju(t)\|^{\frac{1}{j+1}}}{c^{\frac{j}{j+1}}\ (j!)^{\frac{1}{j+1}}}\ ,
\qquad T_{j}(t) := (M_*-1)^2\ c^{\frac{2j}{j+1}}\ \|D^j u(t)\|^{-\frac{2}{j+1}}\ , \label{eq:RealScaleNote}
\\
\cC(j,c,\varepsilon,t_0,t)&:=\left(\|D^ju(\cdot,\varepsilon (t-t_0)^{1/2},t)\|+\|D^jv(\cdot,\varepsilon (t-t_0)^{1/2},t)\|\right)^{\frac{1}{j+1}}\big/ \left(c^{\frac{j}{j+1}}\ (j!)^{\frac{1}{j+1}}\right)\ , \notag 
\end{align}
where $M_*$ is chosen such that
\begin{align*}
T_*=(M_*-1)^2(k_*!)^{-\frac{2}{k_*+1}}\cR(k,c,t_0)^{-2}\ .
\end{align*}
For any such $t_\tau$, one can assume that at least one index $k_\tau$ (at $t_\tau$) satisfies
\begin{align}\label{eq:DescDerOrdOpp}
\cR(k_\tau,c,t_\tau)\ge M^{\frac{1}{k+1}}\cR(k,c,t_0)\ ,
\end{align}
because the opposite for all $k_\tau$ implies $\|D^ku(t_\tau)\|\le \mu_* \|D^ku_0\|$, using the same argument as before. Moreover, we place such indexes in ascending order: $k_*<k_1<k_2<\cdots<k_\tau<k_{\tau+1}<\cdots$ and assume (if such $t_\tau$ exists) $t_p$ is the first time that \eqref{eq:DescDerOrdOpp} occurs for $k_p$ (so the order \eqref{eq:DescDerOrdAlls} persists (for $k_p$) at most until $s=t_p$) while
\begin{align*}
\cR(k_p,c,t_p) = \max_{k\le j< k_p} \cR(j,c,t_p)\ .
\end{align*}
We claim that, for some $n_p\le (k_p/k_*)^2$,
\begin{align}\label{eq:BddCplxExt}
\sup_{t_p<s<t_0+T_*} \cR(k_p,c,s) \le M_*^{\frac{n_p}{k_p+1}} \cR(k_p,c,t_p)\ .
\end{align}
\textit{Proof of the claim:} Based on the choice of $T_*$ and the assumption~\eqref{eq:DescDerOrd},
\begin{align}\label{eq:BddLowerScale}
\sup_{t_0<s<t_0+T_*}\max_{k\le j\le k_*}\! \cR(j,c,s) \le \max_{k\le j\le k_*}\! M^{\frac{1}{j+1}}\cR(j,c,t_0) \le M^{\frac{1}{k+1}}\cR(k,c,t_0)\ .
\end{align}
Recall that $k_*$ is chosen according to the condition~\eqref{eq:AscDerCond}, while $k_1$ and $t_1$ are, respectively, the smallest index and the first temporal point for which \eqref{eq:DescDerOrdOpp} is realized as an equality, implying
\begin{align*}
\sup_{t_0<s<t_1}\max_{k_*< i<k_1}\! \cR(i,c,s) < M^{\frac{1}{k+1}}\cR(k,c,t_0)\ ,
\end{align*}
which, together with \eqref{eq:BddLowerScale}, guarantees \eqref{eq:AscDerOrd} (at $s=t_1$, with $\ell=k$ and $k=k_1$). Then, by Theorem~\ref{le:AscendDer}
\begin{align*}
\sup_{t_1<s<t_1+T_{k_1}}\! \cR(i,c,s)  \le M_*^{\frac{1}{k_1+1}} \cR(k_1,c,t_1) \ ,\qquad \forall\ k_*\le i\le k_1\ .
\end{align*}
If $\displaystyle\sup_{t_1+T_{k_1}<s<t_0+T_*}\cR(k_1,c,s) \le M_*^{\frac{1}{k_1+1}} \cR(k_1,c,t_1)$, then \eqref{eq:BddCplxExt} is achieved immediately; otherwise we repeat the above procedure until the above inequality is attained at some $s=t_1+r\cdot T_{k_1}$ or until $t_1+n_1T_{k_1}\ge t_0+T_*$, and this shall lead to
\begin{align*}
\sup_{t_0<s<t_0+T_*}\max_{k_*\le j\le k_1}\! \cR(j,c,s) \le \left(M_*^{\frac{1}{k_1+1}}\right)^{n_1}\cR(k_1,c,t_1)
\end{align*}
where, based on the choice of $T_*$ and $M_*$,
\begin{align*}
n_1 &\le (t_0+T_*-t_1)/T_{k_1}(t_1)\lesssim T_*\ (M_*-1)^{-2}(k_1!)^{\frac{2}{k_1+1}}\cR(k_1,c,t_1)^2
\\
&\lesssim T_*\ (M_*-1)^{-2}(k_1!)^{\frac{2}{k_1+1}}\cR(k,c,t_0)^2\lesssim (k_1/k_*)^2\ .
\end{align*}
Consequently, if $k_1$ satisfies $M_*^{k_1/k_*^2}\le M$ then
\begin{align*}
\sup_{t_0<s<t_0+T_*}\max_{k_*\le j\le k_1}\! \cR(j,c,s) \le M\ \cR(k_1,c,t_1)\ .
\end{align*}
If $k_p$ is such that $\prod_{\tau=1}^p M_*^{n_\tau/(k_\tau+1)}\le M$, an induction argument leads to
\begin{align*}
\sup_{t_0<s<t_0+T_*}\max_{k_*\le j\le k_p}\! \cR(j,c,s) \le M_*^{\frac{n_p}{k_p+1}} \cR(k_p,c,t_p)\le M\ \cR(k_1,c,t_1)\ ,
\end{align*}
where $n_\tau\le (t_0+T_*-t_\tau)/T_{k_\tau}(t_\tau)\lesssim (k_\tau/k_*)^2$. In fact, for any $k_p$ such that $M_*^{k_p/k_*^2}\le M$,
\begin{align*}
\sup_{t_0<s<t_0+T_*}\max_{k\le j\le k_p}\! \cR(j,c,s) \le M\ \cR(k_1,c,t_1)\lesssim M\ \cR(k,c,t_0)\ .
\end{align*}
This proves the claim (stronger than the claim). On the other hand, we claim that, `for however large index $k_1$ is',
\begin{align}\label{eq:RkbetaBdd}
\sup_{t_0<s<t_0+T_*}\cR(k_1,c,s) \le M^\beta\ \cR(k,c,t_0) \qquad \textrm{with }\ \beta<1\ .
\end{align}
\textit{Proof of the claim:} Recall that $k_1$ is the foremost index for \eqref{eq:DescDerOrdOpp}, so
\begin{align*}
\sup_{s<t_0+T_*} \max_{k_*\le i<k_1}\! \cR(i,c,s) \le \sup_{s<t_0+T_*} \max_{k\le i\le k_*}\! \cR(i,c,s) \le M^{\frac{1}{k+1}}\cR(k,c,t_0)\ .
\end{align*}
The opposite of the claim, together with the above restriction (for $i<k_1$), implies there exists $t<t_0+T_*$ such that, for some $\tilde{\beta}<\beta$,
\begin{align*}
\max_{k\le i<k_1}\! \cR(i,c,t) \le M^{-\tilde{\beta}}\ \cR(k_1,c,t)
\end{align*}
and with a similar argument to the proof of Theorem~\ref{le:AscendDer} we deduce that
\begin{align*}
\sup_{t<s<t+\tilde{T}} \cC(k_1,c,\varepsilon,t,s) \le M^{\tilde{\beta}}\ \cR(k_1,c,t)
\end{align*}
where $|\varepsilon|\gtrsim 1-M^{-\tilde{\beta}}$ and $\tilde{T}\gtrsim (\eta/2)^{-2}\left(1-M^{-\tilde{\beta}}\right)^{-2}\|D^{k_1}u(t)\|^{-\frac{2}{k_1+1}}$. Then, by Theorem~\ref{th:LerayZalpha},  Proposition~\ref{prop:HarMaxPrin} and the above estimate for $D^{k_1}u$,
\begin{align*}
\cR\left(k_1,c,t+\tilde{T}\right) \le \tilde{\mu}\cdot \cR(k_1,c,t) \qquad \textrm{with }\ \tilde{\mu}<1\ ,
\end{align*}
which shows that either spatial intermittency of $D^{k_1}u$ occurs before $s=t+\tilde{T}$($<t_0+T_*$) with $\cR(k_1,c,s) \le M^{2\tilde{\beta}}\ \cR(k,c,t_0)$ or
\begin{align*}
\sup_{s<t_0+T_*}\cR(k_1,c,s) < M^{2\tilde{\beta}}\ \cR(k,c,t_0)\ .
\end{align*}
This proves that \eqref{eq:RkbetaBdd} must hold provided $k_1$( $>k_*$) is the foremost index for \eqref{eq:DescDerOrdOpp} to occur. Summarizing the above two claims (i.e. \eqref{eq:BddCplxExt} and \eqref{eq:RkbetaBdd}), we have shown that
\begin{align*}
\sup_{t_0<s<t_0+T_*}\max_{k\le j\le \ell_*}\! \cR(j,c,s) \le M\ \cR(k,c,t_0)
\end{align*}
where $\ell_*$ is chosen such that $M_*^{\ell_*/k_*^2}\le M$ (where $\ell_*\gg k_*$ since $M_*\ll M$). Finally, we claim that
\begin{align*}
\sup_{t_0<s<t_0+T_*}\max_{\ell_*< j\le 2\ell_*}\! \cR(j,c,s) \le M^\beta\! \sup_{t_0<s<t_0+T_*}\max_{k\le j\le \ell_*}\! \cR(j,c,s) \qquad \textrm{with }\ \beta\lesssim k_*^{-1}\ .
\end{align*}
Assume the opposite. Then there exist $\ell_*<i\le 2\ell_*$ and $t<t_0+T_*$ such that
\begin{align*}
\sup_{t_0<s<t_0+T_*}\max_{k\le j\le \ell_*}\! \cR(j,c,s) < M^{-\beta} \cR(i,c,t)\ .
\end{align*}
With a similar argument to the proof of Theorem~\ref{le:AscendDer} we deduce that
\begin{align*}
\sup_{t<s<t+\tilde{T}} \cC(i,c,\varepsilon,t,s) \le \tilde{M}\ \cR(i,c,t)
\end{align*}
where $|\varepsilon|\gtrsim 1-\tilde{M}^{-1}$ and $\tilde{T}\gtrsim_{\eta, \beta} \|D^iu(t)\|^{-\frac{2}{i+1}}$. Similarly to the above argument for $k_1$, by Theorem~\ref{th:LerayZalpha} and Proposition~\ref{prop:HarMaxPrin}, spatial intermittency of $D^i u$ occurs at $s=t+\tilde{T}$. Thus, $\|D^iu\|$ starts decreasing whenever it reaches the critical state as above, and this proves the claim. Inductively, one can show
\begin{align*}
\sup_{t_0<s<t_0+T_*}\max_{2^n\ell_*< j\le 2^{n+1}\ell_*}\! \cR(j,c,s) \le M^{\beta_n}\!\! \sup_{t_0<s<t_0+T_*}\max_{k\le j\le 2^n\ell_*}\! \cR(j,c,s) \qquad \textrm{with }\ \beta_n\lesssim 2^{-n}k_*^{-1}\ ,
\end{align*}
and therefore
\begin{align*}
\sup_{t_0<s<t_0+T_*}\max_{j>\ell_*}\ \cR(j,c,s) \le M^{2/k_*}\!\! \sup_{t_0<s<t_0+T_*}\max_{k\le j\le \ell_*}\! \cR(j,c,s) \ ,
\end{align*}
which, together with the summary of the previous two claims, yields
\begin{align*}
\sup_{t_0<s<t_0+T_*}\max_{j\ge k}\ \cR(j,c,s) \lesssim M\ \cR(k,c,t_0)
\end{align*}
and the complex extension $\displaystyle\sup_{z\in B_{r_t}(x_0,0)}\left|D^ku_t(z)\right|$ has the same upper estimate as in the initial case; an application of Proposition~\ref{prop:HarMaxPrin} then completes the proof.
\end{proof}

\begin{remark}
If we assume $D^j u(t)\in Z_j(\lambda,\delta,c)$ with $\alpha=1/(k+1)$ for all $j\ge k_*$, then one can prove the statement with the same $\mu_*$ for much longer duration $T_*$.
\end{remark}

\begin{lemma}\label{cor:DerOrdConfig}
Let $u$ be a Leray solution to \eqref{eq:NSE1}-\eqref{eq:NSE3} initiated at $u_0$ and $\ell$ large enough such that
$\|u_0\|\lesssim (1+\epsilon)^{\ell}$.
For any fixed $\kappa>\ell$, if \eqref{eq:ParaAdjMaxP} is satisfied (for $k=\kappa$) with $\mu_*\le 1$ in Theorem~\ref{le:DescendDer}, with the notation introduced in \eqref{eq:RealScaleNote}, one of the following two cases must occur:

(I)$^*$ There exist $t$ and $k \ge \kappa$ such that
\begin{align*}
\cR(j,c,t) \le \cR(k,c,t)\, \qquad \forall \ell\le j\le k\
\end{align*}
and $\cR(k,c,t)\le \displaystyle\max_{\ell\le j\le \kappa} \cR(j,c,t_0)$.

(II)$^*$
\begin{align*}
\sup_{s>t_0}\ \max_{j\ge\ell}\ \cR(j,c,s) \le \max_{\ell\le j\le \kappa}\ \cR(j,c,t_0)\ .
\end{align*}
\end{lemma}

\begin{proof}
At $t=t_0$ assume the opposite of Case (I)$^*$, i.e. there exists $\ell\le k_1<\kappa$ such that
\begin{align}\label{eq:PreOrder}
\cR(j,c,t_0) \le \cR(k_1,c,t_0)\, \qquad \forall \ell\le j\le k_1
\end{align}
while
\begin{align}\label{eq:PostOrder}
\cR(k_1,c,t_0)> \cR(j,c,t_0) \, \qquad \forall j> k_1\ .
\end{align}
Then, one of the following two cases must occur,

(I)$'$ Such order remains for all time;

(II)$'$ Such order remains until $t=t_1$ at which point $\cR(k_2,c,t)$ becomes the maximal among all $\cR(j,c,t)$ for all $j\ge\ell$.

If Case~(I)$'$ occurs, we claim that
\begin{align*}
\sup_{s>t_0}\ \cR(j,c,s) \le \cR(k_1,c,t_0)\ , \qquad \forall j\ge\ell\ .
\end{align*}
\textit{Proof of the claim:} It suffices to show
$$\sup_{s>t_0} \|D^{k_1}u(s)\|\le \|D^{k_1}u_0\|\ .$$
Since the order in \eqref{eq:PostOrder} remains for all $s$, by Theorem~\ref{le:DescendDer} (applied with $\mu_*\le 1$; condition~\eqref{eq:ParaAdjMaxP} holds for $\kappa$, thus also for $k_1$ as $k_1<\kappa$ and $c<1$),
$$\|D^{k_1}u(s)\|\le\|D^{k_1}u_0\|\ , \qquad \forall\ t_0<s<T_*$$
and we can extend the result past $T_*$ given in Theorem~\ref{th:MainThmVel} because $\|u(s)\|$ is restricted by $\|D^{k_1}u(s)\|$ (recall that $u(s)$ is a Leray solution and Lemma~\ref{le:GNIneq}, plus, without loss of generality we can assume $\|D^ju_0\|\lesssim j!\ \|u_0\|^j$ as a consequence of the analyticity result in Theorem~\ref{th:LinftyIVP}) and this in turn restricts the growth of $\|D^{k_1}u(s)\|$. This proves the claim. So, if Case~(I)$'$ occurs, Case~(II)$^*$ is achieved for all $t>t_0$.

If Case~(II)$'$ occurs and $k_2\ge \kappa$, Case~(I)$^*$ is achieved at $t=t_1$. If $k_2<\kappa$, then one of the following
two cases must occur,

(I)$''$ $\cR(k_2,c,t)$ remains the maximal for all $t>t_1$;

(II)$''$ Such order remains until $t=t_2$ at which point $\cR(k_3,c,t)$ becomes the maximal among all $\cR(j,c,t)$ for all $j\ge\ell$.

If Case~(I)$''$ occurs, we claim that
\begin{align}
\cR(k_1,c,t_1) &= \cR(k_2,c,t_1)\ , \label{eq:transitpt}
\\
\sup_{t_0<s<t_1} \cR(j,c,s) &\le \cR(k_1,c,t_0)\ , \qquad \forall j\ge\ell\ , \label{eq:MaxT1}
\\
\sup_{s>t_1}\ \cR(j,c,s) &\le \cR(k_2,c,t_1)\ , \qquad \forall j\ge\ell\ . \label{eq:MaxT2}
\end{align}
\textit{Proof of the claim:} \eqref{eq:transitpt} holds because $t=t_1$ is the transition time between $D^{k_1}u(s)$ and $D^{k_2}u(s)$. An argument similar to the previous step implies \eqref{eq:MaxT1} and \eqref{eq:MaxT2}. In particular, we have
\begin{align*}
\sup_{s>t_1} \|D^{k_2}u(s)\|\le \|D^{k_2}u(t_1)\|\ , \qquad \sup_{t_0<s<t_1} \|D^{k_1}u(s)\|\le \|D^{k_1}u_0\|\ .
\end{align*}
Thus, for all $s>t_0$ and all $j\ge\ell$,
\begin{align*}
\cR(j,c,s) &\le \max\left\{\sup_{t_0<s<t_1} \cR(j,c,s),\ \sup_{s>t_1} \cR(j,c,s)\right\}
\\
&\le \max\left\{\cR(k_1,c,t_0) ,\ \cR(k_2,c,t_1)\right\}
\\
&= \max\left\{\cR(k_1,c,t_0) ,\ \cR(k_1,c,t_1)\right\}
\\
&\le \cR(k_1,c,t_0)\ .
\end{align*}
Hence, if Case~(I)$''$ occurs, Case~(II)$^*$ of the lemma is achieved.

Inductively, if $k_j<\kappa$ in Case~(II)$^{(j-1)}$ for all $j\le i$, a similar argument (utilizing Theorem~\ref{le:DescendDer}) leads to
\begin{align*}
\sup_{t_0<s<t_i} \max_{j\ge\ell}\ \cR(j,c,s) \le \cR(k_1,c,t_0)
\end{align*}
meaning that Case~(II)$^*$ is maintained until $t=t_i$. If $k_i\ge \kappa$ occurs (at the first time) in Case~(II)$^{(i-1)}$, then Case~(I)$^*$ is achieved at $t=t_{i-1}$.
\end{proof}

\begin{corollary}\label{cor:ScaleBound}
Let $u$ be a Leray solution of \eqref{eq:NSE1}-\eqref{eq:NSE3}. Suppose $\ell$ is sufficiently large such that $\|u_0\|\lesssim (1+\epsilon)^{\ell}$. For any $\kappa>\ell$, if there exists a sequence of positive numbers $\{c_j\}_{j=\ell}^{\infty}$ such that $c_{j+1}\le c_j<1$ and for some fixed $p\in\bN^+$
\begin{align}\label{eq:ParaAdjMaxPIndW}
\left(\lambda h^* + \exp\left((2e/\eta)(1+\epsilon)^{\frac{\ell}{j+p}} c_{j}^{\frac{1}{j+p+1}}\right) (1-h^*)\right)M \le 1
\end{align}
is satisfied for each $j$ (where $\eta$ and $h^*$ are defined as in Theorem~\ref{le:DescendDer} and $M$ is given in Theorem~\ref{th:MainThmVel}), then for sufficiently large $t$, one of the following two cases must occur:

(I)$^*$ There exist temporal point $t\ge t_0$, $k\ge\kappa +p$ and constants $\displaystyle \cB_{k,i}\le \prod_{j=i}^k c_j^{-\frac{1}{(j+1)(j+2)}}$ such that
\begin{align*}
\cR(i,c_\ell,t) \le \cB_{k,i} \cdot \cR(k,c_\ell,t) \qquad\textrm{for all}\quad \ell\le i\le k\
\end{align*}
and $\cR(i,c_i,t)\le \displaystyle \max_{\ell\le j\le i} \cR(j,c_i,t_0)$ for all $\ell\le i\le k$.

(II)$^*$ Otherwise there exist $t\ge t_0$, $r<\kappa$ and constants $\displaystyle \cB_{r,i}\le \prod_{j=i}^r c_j^{-\frac{1}{(j+1)(j+2)}}$ such that
\begin{align*}
\sup_{s>t}\ \cR(i,c_\ell,s) \le \cB_{r,i} \cdot \cR(r,c_\ell,t) \qquad\textrm{for all}\quad \ell\le i\le r
\end{align*}
and constants $\displaystyle \cC_{i,r}\le c_r^{\frac{1}{i+1}-\frac{1}{r+1}}$ such that
\begin{align*}
\sup_{s>t}\ \cR(i,c_r,s) \le \cC_{i,r} \cdot \max_{\ell\le j\le \ell+p} \cR(j,c_\ell,t_0) \qquad\textrm{for all}\quad i>r\ .
\end{align*}
\end{corollary}

\begin{proof}
By Lemma~\ref{cor:DerOrdConfig} (applied with $\kappa=\ell+p$), one of the following two cases must occur:

(I) There exist $t_1$ and $k_1\ge\ell+p$ such that
\begin{align}\label{eq:AllDbdt1}
\cR(j,c_\ell,t_1) \le \cR(k_1,c_\ell,t_1)\ , \qquad \forall \ell\le j\le k_1\ ;
\end{align}

(II) Otherwise
\begin{align*}
\sup_{s>t_0}\ \max_{j\ge\ell}\ \cR(j,c_\ell,s) \le \max_{\ell\le j\le \ell+p} \cR(j,c_\ell,t_0)\ .
\end{align*}
If Case (II) occurs, suppose $\cR(r,c_\ell,t_0)$ is the maximal among all $\cR(j,c_\ell,t_0)$ for $\ell\le j\le \ell+p$; then
\begin{align*}
\sup_{s>t_0}\ \max_{j\ge\ell}\ \cR(j,c_\ell,s) \le  \cR(r,c_\ell,t_0)\ .
\end{align*}
Thus, Case (II)$^*$ in the lemma is achieved (for some $\ell\le r\le \ell+p$) at $t=t_0$.

If Case (I) occurs in the above argument (we assume $k_1<\kappa+p$, otherwise if one can find arbitrarily large $k_1$ for \eqref{eq:AllDbdt1}, then Case~(I)$^*$ is achieved at $t=t_1$), by Lemma~\ref{cor:DerOrdConfig} (applied with $\kappa=k_1$), one of the following two cases must occur:

(I)$'$ There exist $t_2$ and $k_2\ge k_1+p$ such that
\begin{align}\label{eq:AllDbdt2}
\cR(j,c_{k_1},t_2) \le \cR(k_2,c_{k_1},t_2)\, \qquad \forall k_1\le j\le k_2\ ;
\end{align}

(II)$'$ Otherwise
\begin{align*}
\sup_{s>t_1}\ \max_{j\ge k_1}\ \cR(j,c_{k_1},s) \le \max_{k_1\le j\le k_1+p} \cR(j,c_{k_1},t_1)\ .
\end{align*}
If Case (II)$'$ occurs, suppose $\cR(r_1,c_{k_1},t_1)$ is the maximal among all $\cR(j,c_{k_1},t_1)$ for $k_1\le j\le k_1+p$; then
\begin{align}\label{eq:PostDbdt1t2}
\sup_{s>t_1}\ \max_{j\ge k_1}\ \cR(j,c_{k_1},s) \le  \cR(r_1,c_{k_1},t_1)
\end{align}
while \eqref{eq:AllDbdt1} holds.
Without loss of generality, we suppose $t_1$ is the first temporal point where Case~(I) is achieved and $k_1$ is the maximal possible index for \eqref{eq:AllDbdt1} (recall that $k_1<\kappa+p$). Then
\begin{align}\label{eq:t1clreverse}
\max_{\ell\le j\le \ell+p} \cR(j,c_\ell,s) > \max_{j> \ell+p}\ \cR(j,c_\ell,s) \ , \qquad \forall t_0<s<t_1
\end{align}
while
\begin{align*}
\max_{j> k_1}\ \cR(j,c_\ell,t_1) < \cR(k_1,c_\ell,t_1) = \max_{\ell\le j\le k_1} \cR(j,c_\ell,t_1)\ .
\end{align*}
By \eqref{eq:t1clreverse} and Theorem~\ref{le:DescendDer} (which is applicable since $\|u_0\|\lesssim (1+\epsilon)^{\ell}$ and \eqref{eq:ParaAdjMaxPIndW} holds),
\begin{align*}
\max_{\ell\le j\le \ell+p}\ \sup_{t_0<s<t_1} \cR(j,c_\ell,s) \le \max_{\ell\le j\le \ell+p} \cR(j,c_\ell,t_0)\ .
\end{align*}
Recall that since $k_1$ is the maximal possible index for \eqref{eq:AllDbdt1},
\begin{align*}
\cR(r_1,c_\ell,t_1) &\le \cR(k_1,c_\ell,t_1) = \max_{\ell\le j\le k_1} \cR(j,c_\ell,t_1)  \le \max_{\ell\le j\le \ell+p} \cR(j,c_\ell,t_0)\ .
\end{align*}
Presently, we are in Case~(II)$'$ under Case~(I); thus
\begin{align}
\sup_{s>t_1}\ &\max_{j\ge k_1}\ \cR(j,c_{k_1},s) \le \cR(r_1,c_{k_1},t_1) \notag
\\
&\le \left(c_\ell/c_{k_1}\right)^{\frac{r_1}{r_1+1}} \cR(r_1,c_\ell,t_1) \le \left(c_\ell/c_{k_1}\right)^{\frac{r_1}{r_1+1}} \max_{\ell\le j\le \ell+p} \cR(j,c_\ell,t_0)\ . \label{eq:PostDbdt1t2r1Pk1}
\end{align}
Keeping in mind that $\displaystyle \max_{\ell\le j< k_1} \cR(j,c_\ell,t_1) \le \cR(k_1,c_\ell,t_1)$ (recall that $k_1$ is the maximal index for \eqref{eq:AllDbdt1}) and $c_j<c_\ell$, applying Theorem~\ref{le:DescendDer} for each $c_j$ (from $j=k_1$ to $j=i$), we deduce that, for any $\ell\le i<k_1$,
\begin{align*}
\sup_{s>t_1}\ \max_{k_1-p\le j<k_1} \cR(j,c_{\scriptscriptstyle k_1-p},s) &\le \sup_{s>t_1}\ \max_{j\ge k_1}\ \cR(j,c_{\scriptscriptstyle k_1-p},s)\ ,
\\
\sup_{s>t_1}\ \max_{k_1-2p\le j<k_1-p} \cR(j,c_{\scriptscriptstyle k_1-2p},s) &\le \sup_{s>t_1}\ \max_{j\ge k_1-p} \cR(j,c_{\scriptscriptstyle k_1-2p},s)\ ,
\\
\vdots &
\\
\sup_{s>t_1}\ \max_{i\le j<k_1-np} \cR(j,c_i,s) &\le \sup_{s>t_1}\ \max_{j\ge k_1-np} \cR(j,c_i,s)\ ,
\end{align*}
where $i<k_1-np\le i+p$.
Combining the above chain of relations with \eqref{eq:PostDbdt1t2} we obtain
\begin{align}
\sup_{s>t_1}\!\ \cR(i,c_i,s) & \le \left(\frac{c_{\spm k_1-np}}{c_i}\right)^{\spm \frac{k_1-np}{k_1-np+1}}\! \left(\frac{c_{\spm k_1-(n-1)p}}{c_{\spm k_1-np}}\right)^{\spm \frac{k_1-(n-1)p}{k_1-(n-1)p+1}}\! \cdots \left(\frac{c_{\spm k_1}}{c_{\spm k_1-p}}\right)^{\spm \frac{k_1}{k_1+1}} \sup_{s>t_1}\!\ \max_{j\ge k_1}\!\ \cR(j,c_{k_1},s) \notag
\\
&\le c_{\spm k_1-np}^{\spm \frac{-p}{(k_1-(n-1)p+1)(k_1-np+1)}} \cdots  c_{\spm k_1-p}^{\scriptscriptstyle \frac{-p}{(k_1-p+1)(k_1+1)}} \left(c_{\spm k_1}^{\spm \frac{k_1}{k_1+1}} \big/ c_i^{\spm \frac{k_1-np}{k_1-np+1}}\right) \cR(r_1,c_{\spm k_1},t_1) \notag
\\
&=: \tilde{\cB}_{k_1,i}^{(p)}\left(c_\ell\big/c_{\spm k_1}\right)^{\frac{r_1}{r_1+1}} \cdot \cR(r_1,c_\ell,t_1) \notag
\\
&\lesssim \cB_{r_1,i} \cdot \max_{\ell\le j\le \ell+p} \cR(j,c_\ell,t_0)\ , \qquad \qquad \qquad \qquad\textrm{for all}\quad \ell\le i<k_1\ . \label{eq:PostDbdt1t2r1}
\end{align}
Thus, in view of \eqref{eq:PostDbdt1t2r1} and \eqref{eq:PostDbdt1t2r1Pk1}, Case (II)$^*$ of the corollary is achieved (at $t=t_1$).

If Case (I)$'$ occurs in the above argument (we assume $k_2<\kappa+p$, otherwise Case~(I)$^*$ is achieved at $t=t_2$), again by Lemma~\ref{cor:DerOrdConfig} (applied with $\kappa=k_2$), one of the following two cases must occur:

(I)$''$ There exist $t_3$ and $k_3\ge k_2+p$ such that
\begin{align*}
\cR(j,c_{k_2},t_3) \le \cR(k_3,c_{k_2},t_3)\, \qquad \forall k_2\le j\le k_3\ ;
\end{align*}

(II)$''$ Otherwise
\begin{align*}
\sup_{s>t_2}\ \max_{j\ge k_2}\ \cR(j,c_{k_2},s) \le \max_{k_2\le j\le k_2+p} \cR(j,c_{k_2},t_2)\ .
\end{align*}
If Case (II)$''$ occurs, suppose $t_2$ is the first temporal point where Case~(I)$'$ is achieved and $k_2$ is the maximal possible index for \eqref{eq:AllDbdt2}, and--in addition--suppose $\cR(r_2,c_{k_2},t_2)$ is the maximal among all $\cR(j,c_{k_2},t_2)$ for $k_2\le j\le k_2+p$. Then,
\begin{align}\label{eq:PostDbdt1t3}
\sup_{s>t_2}\ \max_{j\ge k_2}\ \cR(j,c_{k_2},s) \le \cR(r_2,c_{k_2},t_2)
\end{align}
while \eqref{eq:AllDbdt2} holds.
Suppose $t_2$ is the first temporal point where \eqref{eq:AllDbdt2} is achieved, then
\begin{align}\label{eq:t2c2reverse}
\max_{k_1\le j\le k_1+p} \cR(j,c_{k_1},s) > \max_{j> k_1+p} \cR(j,c_{k_1},s)\ , \qquad \forall t_1<s<t_2
\end{align}
while
\begin{align*}
\max_{j> k_2}\ \cR(j,c_{k_1},t_2) < \cR(k_2,c_{k_1},t_2) = \max_{k_1\le j\le k_2} \cR(j,c_{k_1},t_2)\ .
\end{align*}
Recall that we are in Case~(II)$''$ under Case~(I)$'$ and $k_2$ is the maximal possible index for \eqref{eq:AllDbdt2};
similarly to the previous step, we deduce
\begin{align*}
&\cR(r_2,c_{k_1},t_2) \le \max_{k_1\le j\le k_2} \cR(j,c_{k_1},t_2)  \le \max_{k_1\le j\le k_1+p} \cR(j,c_{k_1},t_1)\ , \notag
\\
&\sup_{s>t_2}\ \max_{j\ge k_2}\ \cR(j,c_{k_2},s) \le \cR(r_2,c_{k_2},t_2) \le \left(c_{k_1}/c_{k_2}\right)^{\frac{r_2}{r_2+1}} \max_{k_1\le j\le k_1+p} \cR(j,c_{k_1},t_1)\ .
\end{align*}
Combining the above estimates with the similar ones in the previous step, we obtain
\begin{align}\label{eq:PostDbdt2t3r2Pk2}
\sup_{s>t_2}\ \max_{j\ge k_2}\ \cR(j,c_{k_2},s) \le \left(\frac{c_{k_1}}{c_{k_2}}\right)^{\frac{r_2}{r_2+1}} \cR(r_1,c_{k_1},t_1)
\le \left(\frac{c_{k_1}}{c_{k_2}}\right)^{\frac{r_2}{r_2+1}} \left(\frac{c_\ell}{c_{k_1}}\right)^{\frac{r_1}{r_1+1}} \max_{\ell\le j\le \ell+p} \cR(j,c_\ell,t_0)\ .
\end{align}
Following an argument similar to the derivation of \eqref{eq:PostDbdt1t2r1} (applying Theorem~\ref{le:DescendDer} step by step), combined with \eqref{eq:PostDbdt1t2r1Pk1}, we deduce
\begin{align}
\sup_{s>t_2}\ \cR(i,c_i,s) & \le \tilde{\cB}_{k_2,i}^{(p)} \left(c_{k_1}\big/c_{k_2}\right)^{\frac{r_2}{r_2+1}} \cdot  \cR(r_2,c_{k_1},t_2) \notag
\\
& \le \tilde{\cB}_{k_2,i}^{(p)} \left(c_{k_1}\big/c_{k_2}\right)^{\frac{r_2}{r_2+1}} \cdot  \cR(r_1,c_{k_1},t_1) \notag
\\
& \le \tilde{\cB}_{k_2,i}^{(p)} \left(c_{k_1}\big/c_{k_2}\right)^{\frac{r_2}{r_2+1}} \left(c_\ell\big/c_{k_1}\right)^{\frac{r_1}{r_1+1}} \max_{\ell\le j\le \ell+p} \cR(j,c_\ell,t_0)\ , \qquad \forall  k_1\le i< k_2\ . \label{eq:PostDbdt2t3r2}
\end{align}
By \eqref{eq:t2c2reverse} and Theorem~\ref{le:DescendDer}
\begin{align*}
\sup_{t_1<s<t_2} \max_{j\ge k_1}\ \cR(j,c_{k_1},s) &\le \max\left\{\sup_{t_1<s<t_2}\max_{j> k_1+p} \cR(j,c_{k_1},s)\ ,\ \sup_{t_1<s<t_2}\max_{k_1\le j\le k_1+p} \cR(j,c_{k_1},s)\right\}
\\
&\le \sup_{t_1<s<t_2}\max_{k_1\le j\le k_1+p} \cR(j,c_{k_1},s) \le \max_{k_1\le j\le k_1+p} \cR(j,c_{k_1},t_1)\ .
\end{align*}
Recall that we are in Case~(II)$''$ under the Subcase~(I)$'$ of Case~(I); utilizing \eqref{eq:t1clreverse}, Theorem~\ref{le:DescendDer} and the above result,
\begin{align*}
\sup_{t_1<s<t_2} & \max_{j\ge k_1}\ \cR(j,c_{k_1},s) \le \left(c_\ell/c_{k_1}\right)^{\frac{r_1}{r_1+1}} \max_{k_1\le j\le k_1+p} \cR(j,c_\ell,t_1)
\\
&\le \left(c_\ell/c_{k_1}\right)^{\frac{r_1}{r_1+1}} \max_{\ell\le j\le \ell+p} \cR(j,c_\ell,t_1) \le \left(c_\ell/c_{k_1}\right)^{\frac{r_1}{r_1+1}} \max_{\ell\le j\le \ell+p} \cR(j,c_\ell,t_0)\ .
\end{align*}
Following an argument similar to the derivation of \eqref{eq:PostDbdt1t2r1} (applying Theorem~\ref{le:DescendDer} `pointwise in $s$'), combined with the above result, we deduce that, for any $t_1<s<t_2$,
\begin{align*}
\cR(i,c_i,s) & \le \tilde{\cB}_{k_1,i}^{(p)} \cdot \max_{j\ge k_1}\ \cR(j,c_{k_1},s) \le \tilde{\cB}_{k_1,i}^{(p)} \left(c_\ell/c_{k_1}\right)^{\frac{r_1}{r_1+1}} \max_{\ell\le j\le \ell+p} \cR(j,c_\ell,t_0)\ , \qquad \forall \ell\le i\le k_1\ .
\end{align*}
In particular, when $s=t_2$, the above result and \eqref{eq:AllDbdt2} (recall that we are in Case~(I)$'$ and $t_2$ is the first time \eqref{eq:AllDbdt2} occurs)
\begin{align}\label{eq:PrfCaseIt2}
\cR(i,c_i,t_2) & \le \tilde{\cB}_{k_1,i}^{(p)} \cdot \max_{j\ge k_1}\ \cR(j,c_{k_1},t_2) = \tilde{\cB}_{k_1,i}^{(p)} \cdot \cR(k_2,c_{k_1},t_2)\ , \qquad \forall \ell\le i\le k_1\ .
\end{align}
The synthesis of \eqref{eq:PostDbdt2t3r2Pk2} and \eqref{eq:PostDbdt2t3r2} with the uniform constant $c_{k_1}$ gives
\begin{align*}
\sup_{s>t_2} \max_{j\ge k_1} \cR(j,c_{k_1},s) &\le \max\left\{\sup_{s>t_2}\max_{j> k_2} \cR(j,c_{k_1},s),\ \sup_{s>t_2}\max_{k_1\le j\le k_2}\! \cR(j,c_{k_1},s)\right\}
\\
&\le \max\left\{\left(\frac{c_{k_2}}{c_{k_1}}\right)^{\frac{k_2}{k_2+1}}\! \sup_{s>t_2}\max_{j> k_2} \cR(j,c_{k_2},s) , \max_{k_1\le j\le k_2}\sup_{s>t_2} \left(\frac{c_j}{c_{k_1}}\right)^{\frac{j}{j+1}}\! \cR(j,c_j,s)\right\}
\\
&\le \tilde{\cB}_{k_2,k_1}^{(p)} \left(c_{k_1}\big/c_{k_2}\right)^{\frac{r_2}{r_2+1}} \left(c_\ell\big/c_{k_1}\right)^{\frac{r_1}{r_1+1}} \max_{\ell\le j\le \ell+p} \cR(j,c_\ell,t_0)\ .
\end{align*}
With \eqref{eq:PrfCaseIt2} in mind and the above restriction, Theorem~\ref{le:DescendDer} implies (similar to the derivation of \eqref{eq:PostDbdt1t2r1})
\begin{align}
\sup_{s>t_2}\ \cR(i,c_i,s) &\le \tilde{\cB}_{k_1,i}^{(p)}\cdot \sup_{s>t_2}\ \max_{j\ge k_1}\ \cR(j,c_{k_1},s) \notag
\\
&\le \tilde{\cB}_{k_2,i}^{(p)} \left(\frac{c_{k_1}}{c_{k_2}}\right)^{\frac{r_2}{r_2+1}} \left(\frac{c_\ell}{c_{k_1}}\right)^{\frac{r_1}{r_1+1}} \max_{\ell\le j\le \ell+p} \cR(j,c_\ell,t_0)\ , \qquad \forall  \ell\le i< k_1\ . \label{eq:PostDbdt2k1ell}
\end{align}
Thus, in view of \eqref{eq:PostDbdt2t3r2Pk2}, \eqref{eq:PostDbdt2t3r2} and \eqref{eq:PostDbdt2k1ell}, Case (II)$^*$ of the corollary is achieved (at $t=t_2$).

Inductively, if Case~(I)$^{(\tau)}$ repeats for multiple times (and $k_\tau<\kappa+p$), then by Lemma~\ref{cor:DerOrdConfig} (applied with $\kappa=k_{\tau}$), one of the following two cases must occur:

(I)$^{(\tau+1)}$ There exist $t_{\tau+1}$ and $k_{\tau+1}\ge k_\tau+p$ such that
\begin{align*}
\cR(j,c_{k_\tau},t_{\tau+1}) \le \cR(k_{\tau+1},c_{k_\tau},t_{\tau+1})\ , \qquad \forall k_\tau\le j\le k_{\tau+1}\ ;
\end{align*}

(II)$^{(\tau+1)}$ Otherwise
\begin{align*}
\sup_{s>t_\tau}\ \max_{j\ge k_\tau}\ \cR(j,c_{k_\tau},s) \le \max_{k_\tau\le j\le k_\tau+p}\ \cR(j,c_{k_\tau},t_\tau)\ .
\end{align*}
Note that the induction terminates when $k_\tau$ reaches $\kappa$ at which point Case~(I)$^*$ is achieved; at this level, \eqref{eq:PrfCaseIt2} (together with \eqref{eq:AllDbdt2} at $s=t_\tau$) becomes
\begin{align*}
\cR(i,c_i,t_\tau) & \le \tilde{\cB}_{k_{\tau-1},i}^{(p)} \cdot \cR(k_\tau,c_{k_{\tau-1}},t_\tau)\ , \qquad \forall \ell\le i\le k_\tau
\end{align*}
which proves the desired inequality in Case~(I)$^*$. If Case~(I)$^{(\tau)}$ stops repeating at some $k_\tau$ and $k_\tau<\kappa+p$, then Case~(II)$^*$ is achieved at $t=t_\tau$.
\end{proof}

\begin{remark}
For the vorticity, the results analogous to Theorem~\ref{le:AscendDer}, Theorem~\ref{le:DescendDer}, Lemma~\ref{cor:DerOrdConfig} and Corollary~\ref{cor:ScaleBound} (with the \emph{a priori} bound in $L^1$) hold as well.
\end{remark}

\medskip

The above four results lead to the main theorem.
\vspace{-0.05in}
\begin{theorem}[Asymptotic Criticality]\label{th:MainResult}
Let $u_0 \in L^\infty \cap L^2$ (resp. $\omega_0 \in L^\infty \cap L^1$) and $u$ in $C((0,T^*), L^\infty)$ where $T^*$ is the first possible blow-up time. Let $c$, $\ell$, $k$ be such that $\|u_0\|\lesssim (1+\epsilon)^{(2/d)\ell}$ (resp. $\|\omega_0\|\lesssim (1+\epsilon)^{(2/d)\ell}$) and \eqref{eq:AscDerCond} holds.
For any index $k\ge \ell$ and temporal point $t$ such that \eqref{eq:AscDerOrd} is satisfied and
\begin{align}\label{eq:tIRRegAscBdd}
&\qquad\quad t+\frac{1}{\cC(\|u_0\|,\ell,k)^2 \|D^ku(t)\|_\infty^{2/(k+1)}} < T^* \\
&\left(\textrm{resp. } t+\frac{1}{\cC(\|\omega_0\|,\ell,k)^2 \|D^k\omega(t)\|_\infty^{2/(k+2)}} < T^* \ \right) \notag
\end{align}
assume that there exists a temporal point
\begin{align*}
&\qquad\quad s=s(t)\in \left[t+\frac{1}{4\cdot\tilde{\cC}(\|u_0\|,\ell,k)\|D^ku(t)\|_\infty^{2/(k+1)}},\ t+\frac{1}{\tilde{\cC}(\|u_0\|,\ell,k)\|D^ku(t)\|_\infty^{2/(k+1)}}\right]
\\
&\left(\textrm{resp. } s=s(t)\in \left[t+\frac{1}{4\cdot\tilde{\cC}(\|\omega_0\|,\ell,k)\|D^k\omega(t)\|_\infty^{2/(k+2)}},\ t+\frac{1}{\tilde{\cC}(\|\omega_0\|,\ell,k)\|D^k\omega(t)\|_\infty^{2/(k+2)}}\right]\ \right)
\end{align*}
such that for any spatial point $x_0$, there exists a scale
\begin{align}\label{eq:kRegScale}
\rho\le \frac{1}{2\cdot\tilde{\cC}(\|u_0\|,\ell,k) \|D^ku(s)\|_\infty^{\frac{1}{k+1}}}\quad \left(\textrm{resp. }\rho\le \frac{1}{2\cdot\tilde{\cC}(\|\omega_0\|,\ell,k) \|D^k\omega(s)\|_\infty^{\frac{1}{k+2}}}\right)
\end{align}
with the property that the super-level set
\begin{align*}
&\qquad\quad V_\lambda^{j,\pm}=\left\{x\in\bR^d\ |\ (D^ku)_j^\pm(x,s)>\lambda \|D^ku(s)\|_\infty\right\}
\\
&\left(\textrm{resp. }\Omega_\lambda^{j,\pm}=\left\{x\in\bR^3\ |\ (D^k\omega)_j^\pm(x,s)>\lambda \|D^k\omega(s)\|_\infty\right\}\ \right)
\end{align*}
is 1D $\delta$-sparse around $x_0$ at scale $\rho$, with each constant $\tilde{\cC}(\|u_0\|,\ell,k)$ chosen such that
\begin{align}\label{eq:CCtildeComparison}
&\qquad\quad \tilde{\cC}(\|u_0\|,\ell,k)\gtrsim k^2\cdot \cC(\|u_0\|,\ell,k)\ ,\qquad \forall\ k\ge \ell
\\
&\left(\textrm{resp. } \tilde{\cC}(\|\omega_0\|,\ell,k)\gtrsim k^2\cdot \cC(\|\omega_0\|,\ell,k)\ ,\qquad \forall\ k\ge \ell \ \right) \notag
\end{align}
where $C(\|u_0\|_2,\ell,k)$'s are given in Theorem~\ref{le:AscendDer}; here the index $(j,\pm)$ is chosen such that $|D^ku(x_0,s)|=(D^ku)_j^\pm(x_0,s)$ (resp. $|D^k\omega(x_0,s)|=(D^k\omega)_j^\pm(x_0,s)$), and the pair $(\lambda,\delta)$ is chosen such that \eqref{eq:ParaAdjMaxP} in Theorem~\ref{le:DescendDer} holds. Then, there exists $\gamma>0$ such that $u\in L^\infty((0,T^*+\gamma); L^\infty)$.

In other words, if $D^ku(s)\in Z_{\alpha_k}(\lambda,\delta,c_0)$ (resp. $D^k\omega(s)\in Z_{\alpha_k}(\lambda,\delta,c_0)$) with $\alpha_k=1/(k+1)$ (resp. $\alpha_k=1/(k+2)$) for all $k\ge \ell$, then $T^*$ is not a blow-up time.
\end{theorem}

\medskip

In order to streamline the proof of the theorem, we start with a definition followed by three lemmas.

\begin{definition}
We divide all the indexes into sections at $\ell=\ell_0<\ell_1<\cdots<\ell_i<\ell_{i+1}<\cdots$ such that $\ell_{i+1}=\phi(\ell_i)$ for some increasing function $\phi(x)\ge 2x$ and each pair $(\ell_i, \ell_{i+q})$ satisfies the condition~\eqref{eq:AscDerCond} (with $\ell=\ell_i$ and $k=\ell_{i+q}$) for some fixed integer $q$. With the notation introduced in \eqref{eq:RealScaleNote}, at any temporal point $t<T^*$ and for each index $i$ we pick $m_i\in [\ell_i,\ell_{i+1}]$ such that
\begin{align*}
\cR(m_i,c(\ell_i),t) = \max_{\ell_i\le j\le \ell_{i+1}} \cR(j,c(\ell_i),t)
\end{align*}
while
\vspace{-0.1in}
\begin{align*}
\cR(m_i,c(\ell_i),t) > \max_{\ell_i\le j< m_i} \cR(j,c(\ell_i),t) \
\end{align*}
where $c(\ell_i):=c_{\ell_{i+1}}$ which is the constant defined by \eqref{eq:ParaAdjMaxPIndW} with $j=\ell_{i+1}$. (If such index $m_i$ does not exist in $[\ell_i,\ell_{i+1}]$ then we let $m_i=\ell_i$.)
Note that $m_i(t)$ may be variant in time, and we will always assume $m_i$ corresponds to the temporal point $t$ in $\cR\left(m_i,\cdot, t\right)$.
Then, we divide the proof into two basic scenarios: (I) either there exists $k_i>\ell_{i+1}$ such that
\begin{align}\label{eq:NonDescAtLi}
\cR(k_i,c(\ell_i),t)\ge \max_{m_i\le j\le k_i} \cR(j,c(\ell_i),t)\ ,
\end{align}
(II) or
\vspace{-0.18in}
\begin{align}\label{eq:MaxDescAtLi}
\cR(m_i,c(\ell_i),t) > \max_{j> m_i} \cR(j,c(\ell_i),t)\ .
\end{align}
We call a section $[\ell_i,\ell_{i+1}]:=\left\{\cR(j,c(\ell_i),t)\right\}_{j=\ell_i}^{\ell_{i+1}}$ Type-$\cA$ if it satisfies \eqref{eq:NonDescAtLi} and Type-$\cB$ if it satisfies \eqref{eq:MaxDescAtLi}. We call the union of sections $[\ell_i, \ell_j]:= \cup_{i\le r\le j-1}[\ell_r,\ell_{r+1}]$ a string if $j-i\ge q$ or the condition~\eqref{eq:AscDerCond} is satisfied with $\ell=\ell_i$ and $k=\ell_j$, and we call a string Type-$\cA$ if it consists of only Type-$\cA$ sections and Type-$\cB$ if it contains at least one Type-$\cB$ section.
\end{definition}

\begin{lemma}\label{le:AstrBdd}
Suppose $\sup_{t>t_0}\|u(t)\|\lesssim (1+\epsilon)^{\ell_i}$, \eqref{eq:AscDerCond} holds at any temporal point with $\ell=\ell_i$ and $k=\ell_{i+q}$, and the assumption~\eqref{eq:kRegScale} holds for all $k\ge \ell_i$. If a string $[\ell_i,\ell_{i+q}]$ is of Type-$\cA$ at an initial time $t_0$, then for any $i\le r<i+q$,
\begin{align}\label{eq:RestrAtoB}
\max_{\ell_r\le j\le \ell_{i+q}} \sup_{t_0<s<\tilde t} \cR(j, c(\ell_r), s) \le (1+\tilde{\epsilon}(\ell_{i+q}))^{1/\ell_{i+q}} \Theta(p^*,r) \max_{i\le p\le i+q} \cR(m_p,c(\ell_p), t_0)
\end{align}
where $\Theta(p^*,r) \lesssim \tilde{\cB}_{\ell_{p^*}, \ell_r}\cdot c(\ell_{p^*})/c(\ell_r)$ is a constant which only depends on $c(\ell_{p^*})$ and $c(\ell_r)$, with $\tilde{\cB}_{i,j}:=\cB_{i,j}$ defined in Corollary~\ref{cor:ScaleBound} if $i>j$ and $\tilde{\cB}_{i,j}:=(\cB_{j,i})^{-1}$ if $i<j$, and $\tilde{\epsilon}(\ell_{i+q})$ is a small quantity which will be given explicitly in the proof; the subscript $p^*$ is the index for the maximal $\cR(m_p,c(\ell_p), t_0)$, and $\tilde t$ is the first time when $[\ell_i,\ell_{i+q}]$ switches to a Type-$\cB$ string; we set $\tilde t= \infty$ if $[\ell_i,\ell_{i+q}]$ is always of Type-$\cA$.
\end{lemma}

\begin{proof}
Several notions and basic results for Type-$\cA$ sections are developed before proceeding to the proof.

We claim that if $[\ell_i,\ell_{i+1}]$ is of Type-$\cA$ and $k_\tau\in [\ell_\tau,\ell_{\tau+1}]$ is one of the indexes in \eqref{eq:NonDescAtLi}, then all the $[\ell_j,\ell_{j+1}]$'s with $i\le j\le \tau$ are Type-$\cA$ sections. Note
that if $k_\tau$ satisfies \eqref{eq:NonDescAtLi}, then
\begin{align*}
\max_{m_i\le r<k_\tau}\cR(r,c(\ell_i),t_0) \le \cR(m_i,c(\ell_i),t_0) < \cR(k_\tau,c(\ell_i),t_0)\ .
\end{align*}
Recall that $c(\ell_j)<c(\ell_i)$ if $j>i$; hence, the above inequality implies
\begin{align*}
\max_{\ell_j\le r<k_\tau}\cR(r,c(\ell_j),t_0) < \cR(k_\tau,c(\ell_j),t_0)\qquad\textrm{for all}\quad i\le j\le \tau\ .
\end{align*}
This ends the proof of the claim.

Next, for any index $k\ge\ell_\tau$ ($\tau\ge q$) define
\begin{align*}
&T_{k}(t) := \tilde{C}(\|u_0\|,\ell_{\tau-q},\ell_\tau)^{-2} (M_{k}-1)^2 \|D^{k}u(t)\|^{-\frac{2}{k+1}} \ ,
\\
&\theta(i,j,\ell,t) := \frac{\cR(i,c(\ell),t)}{\cR(j,c(\ell),t)} \ , \qquad\ \zeta_{k}^{(i)}(r,j):= \theta(j,r,\ell_i,t)^{j+1} + \frac{\tilde{c}(k)}{c(\ell_i)} (M_{k}-1)\ ,
\\
&\mu_{k}^{(i)}(r,j):= \lambda h^* + (1-h^*) \left(M_{k} + \frac{\tilde{c}(k)}{c(\ell_i)}\ \theta(r,j,\ell_i,t)^{j+1} \right)\ 
\end{align*}
where $\tilde{c}(k):=\tilde{C}(\|u_0\|_2,\ell_{\tau-q},\ell_\tau)^{-1} \lesssim 2^{-k}\left(\cB_{\ell_\tau, \ell_{\tau-q}}\right)^{-k}$ and the constants $\lambda, h^*$ are chosen as in Theorem~\ref{le:DescendDer}.
In the rest of the proof we will write $\tilde{c}(k)$ for $\tilde{C}(\|u_0\|_2,\ell_{\tau-q},\ell_\tau)$ and $\theta(i,j,t)$ for $\theta(i,j,\ell,t)$ whenever there is no ambiguity.
For convenience, we write $\zeta_{k}(i)$ and $\mu_{k}(i)$ if $r=j$ (in which case $\theta=1$).
Without loss of generality we assume that each section $[\ell_\tau,\ell_{\tau+1}]$ after $\ell_i$ contains at most one $k_\tau$ which satisfies \eqref{eq:NonDescAtLi} (otherwise one can pick only the maximal one in $[\ell_\tau,\ell_{\tau+1}]$).
So for each $[\ell_i,\ell_{i+1}]$ one can find a section: $\ell_i<k_i<k_{i+1}<\cdots<k_\tau<\cdots$ such that $k_{\tau+1}-k_\tau\ge \ell_{\tau+1}-\ell_{\tau}$ and for all $\tau$
\begin{align}\label{eq:LocalPeaks}
\theta(k_\tau,k_{\tau+1},t_0)\le 1\ , \qquad \min_{k_\tau<j<k_{\tau+1}}\theta(k_\tau,j,t_0)>1\ .
\end{align}
This implies $T_{k_i}>T_{k_{i+1}}>\cdots>T_{k_\tau}>T_{k_{\tau+1}}>\cdots$ and moreover,
\begin{align*}
\cR(j,c(\ell_i), t_0) \le \theta(k_\tau,k_{\tau+1},t_0)\cdot \cR(k_{\tau+1},c(\ell_i), t_0)\ , \qquad \forall\ j\le k_{\tau+1}\  .
\end{align*}
Since $\|u_0\|\lesssim \|u_0\|_2^{\frac{k_\tau}{k_\tau + d/2}} \|D^{k_{\tau}}u_0\|^{\frac{d/2}{k_\tau + d/2}}$ (by Lemma~\ref{le:GNIneq}) and $\|D^{k_{\tau+1}}u_0\|\lesssim (k_{\tau+1}!)\|u_0\|^{k_{\tau+1}}$ (by Theorem~\ref{th:LinftyIVP}), $\theta$ has the lower bound
\begin{align}
\theta(k_\tau,k_{\tau+1},\ell_i,t_0) &= \frac{\|D^{k_\tau}u_0\|^{\frac{1}{k_\tau +1}}}{c(\ell_i)^{\frac{k_\tau}{k_\tau+1}} (k_\tau!)^{\frac{1}{k_\tau +1}}} \bigg/ \frac{\|D^{k_{\tau+1}}u_0\|^{\frac{1}{k_{\tau+1} +1}}}{c(\ell_i)^{\frac{k_{\tau+1}}{k_{\tau+1}+1}} (k_{\tau+1}!)^{\frac{1}{k_{\tau+1} +1}}} \notag
\\
&\ge c(\ell_i)^{\frac{1}{k_\tau+1} - \frac{1}{k_{\tau+1}+1}} \cdot \frac{(k_{\tau+1}!)^{\frac{1}{k_{\tau+1} +1}}}{(k_\tau!)^{\frac{1}{k_\tau +1}}} \cdot  \frac{\|D^{k_\tau}u_0\|^{\frac{1}{k_\tau +1}}}{\|D^{k_{\tau+1}}u_0\|^{\frac{1}{k_{\tau+1} +1}}}  \notag
\\
&\ge c(\ell_i)^{\frac{1}{k_\tau+1} - \frac{1}{k_{\tau+1}+1}} (k_\tau!)^{-\frac{1}{k_\tau +1}} \|u_0\|_2^{-\frac{2k_\tau + d}{d(k_\tau +1)}} \|u_0\|^{2/d-1}\ . \label{eq:ThetaLowerBdd}
\end{align}

First, we assume that for each $[\ell_j,\ell_{j+1}]$ within $[\ell_i,\ell_{i+q}]$ there are finitely many $k_\tau$'s as described in \eqref{eq:LocalPeaks}, and then consider the case where there are infinitely many $k_\tau$'s for at least one $[\ell_j,\ell_{j+1}]$. Note that if $[\ell_i,\ell_{i+q}]$ is a Type-$\cA$ string, then the maximal $k_\tau$ (for all $[\ell_j,\ell_{j+1}]$'s within $[\ell_i,\ell_{i+q}]$) is greater than $\ell_{i+q}$. Without loss of generality we may assume that the maximal $k_\tau\in [\ell_{i+q}, \ell_{i+q+1}]$, otherwise one can make the same argument over the string $[\ell_i,\ell_\tau]$.

\medskip

We prove \eqref{eq:RestrAtoB} in two steps. In the first step, assume that the maximal index $k_i$ for $[\ell_i,\ell_{i+1}]$ as described in \eqref{eq:NonDescAtLi} is greater than $\ell_{i+q}$, and without loss of generality $k_i\in [\ell_{i+q}, \ell_{i+q+1}]$ and such order remains within $[t_0,\tilde t]$ (the proof is the same if $k_i\in[\ell_\tau, \ell_{\tau+1}]$ for some $\tau>i+q$ which is variant as time goes towards $\tilde t$), that is
\begin{align}\label{eq:FiniteLocalPeaks}
\max_{m_i\le j\le k_i}\sup_{t_0\le s\le \tilde{t}}\theta(j, k_i,\ell_i, s)\le 1\ , \qquad \max_{j> k_i} \sup_{t_0\le s\le \tilde{t}}\theta(j, k_i,\ell_i, s)< 1 \ .
\end{align}
Suppose the second maximal index for \eqref{eq:NonDescAtLi} is $k_i^*$ and such order remains within $[t_0,\tilde t]$ (the proof is the same if $k_i^*$ is variant as time goes towards $\tilde t$), that is
\begin{align}\label{eq:2ndFinitePeaks}
\max_{m_i\le j\le k_i^*}\sup_{t_0\le s\le \tilde{t}}\theta(j, k_i^*, \ell_i, s)\le 1\ , \qquad \max_{k_i^*< j<k_i}\sup_{t_0\le s\le \tilde{t}}\theta(j, k_i^*, \ell_i, s)< 1\ .
\end{align}
With the above assumptions, it follows from the proof of Theorem~\ref{le:AscendDer} (with $\ell=\ell_i$ and $k=k_i$) that \begin{align*}
\sup_{t<s<t+T_{k_i}}\! & \cC(k_i, c(\ell_i), \varepsilon, t, s) \le \left(M_{k_i} + \tilde{c}(k_i)\big/c(\ell_i) \right)^{\frac{1}{k_i+1}} \cR(k_i,c(\ell_i), t)
\end{align*}
and, for any $m_i\le j< k_i$,
\begin{align*}
&\sup_{t<s<t+T_{k_i}}\!\! \cR(j,c(\ell_i), s) \le \left(\theta(j,k_i^*,t)^{j+1}+ \frac{\tilde{c}(k_i)}{c(\ell_i)}\ (M_{k_i}-1) \right)^{\frac{1}{j+1}} \cR(k_i^*,c(\ell_i), t)\
\end{align*}
where the constant $M_{k_i}<2$. As the assumption~\eqref{eq:kRegScale} holds at $t_0$, i.e. $D^{k_i}u(s_1)\in Z_{\alpha_{k_i}}(\lambda,\delta,c(\ell_i))$ with $\alpha_{k_i}=1/(k_i+1)$, where $s_1\in[t_0+T_{k_i}/4,t_0+T_{k_i}]$, and without loss of generality we can assume it is located at the right endpoint, that is $s_1=t_0+T_{k_i}$, by the above estimate for the complex solutions and Proposition~\ref{prop:HarMaxPrin} (applied with $\lambda h^* + M_{k_i}(1-h^*)\le \mu$ where $h^*$ is given in Theorem~\ref{le:DescendDer} and the constant $\mu$ is chosen such that $M_{k_i}\mu\le 1$), we know
\begin{align*}
\cR(k_i, c(\ell_i), t_0+T_{k_i}) \le \left(\mu_{k_i}(i)\right)^{\frac{1}{k_i+1}} \cR(k_i, c(\ell_i), t_0).
\end{align*}
At the same time, for any $m_i\le j< k_i$, by the estimate for the real solutions,
\begin{align*}
\cR(j,c(\ell_i), t_0+T_{k_i})&\le \left(\zeta_{k_i}^{(i)}(k_i^*,j)\right)^{\frac{1}{j+1}} \cR(k_i^*,c(\ell_i), t_0)\ .
\end{align*}
As the assumption~\eqref{eq:kRegScale} still holds at $t_0+T_{k_i}$ (with $s_2=t_0+2T_{k_i}$), if the order \eqref{eq:FiniteLocalPeaks} remains at $t_0+T_{k_i}$, then we can repeat the above procedure (Theorem~\ref{le:AscendDer} and Proposition~\ref{prop:HarMaxPrin}) for $n_{k_{i}}$ times (with $s_n=t_0+n\cdot T_{k_i}$) until
\begin{align}\label{eq:equileveltime}
\cR(k_{i},c(\ell_i), t_0+n_{k_{i}}T_{k_{i}})\le \cR(k_i^*,c(\ell_i), t_0+n_{k_{i}}T_{k_i})\ ,
\end{align}
i.e. the order \eqref{eq:FiniteLocalPeaks} remains until some time approximately at $t_0+n_{k_{i}}T_{k_{i}}$ and $\tilde t\lesssim t_0+n_{k_{i}}T_{k_i}$, and $k_i$ must switch to the index $k_i^*$ in \eqref{eq:FiniteLocalPeaks} at $\tilde t$ because $k_i^*$ is always the second maximal index before $\tilde t$ as shown in \eqref{eq:2ndFinitePeaks}.
Note that $\mu_{k_i}(i)<1$ before $t_0+n_{k_{i}}T_{k_{i}}$ while $\zeta_{k_i}(i)>1$, so when applying Proposition~\ref{prop:HarMaxPrin} each time, $\cR(k_{i},c(\ell_i), t)$ is decreasing with possible slight perturbation which is less than $\left(\zeta_{k_i}(i)\right)^{\frac{1}{k_i+1}}$ multiple of the current size, that is, for each $\nu<n_{k_{i}}$,
\begin{align*}
\sup_{t_0+\nu T_{k_{i}}<s<t_0+(\nu+1)T_{k_{i}}}\!\! \cR(k_{i},c(\ell_i), s) \le \left(\zeta_{k_i}(i)\right)^{\frac{1}{k_i+1}} \cR(k_{i},c(\ell_i), t_0+\nu T_{k_{i}})
\end{align*}
while \eqref{eq:FiniteLocalPeaks} and \eqref{eq:2ndFinitePeaks} are preserved. Therefore
\begin{align*}
\max_{\ell_i\le j< k_i} \sup_{t_0<s<\tilde t} \cR(j, c(\ell_i), s) &\lesssim \cR(k_{i},c(\ell_i), t_0+n_{k_{i}}T_{k_{i}}) \lesssim \cR(k_i^*,c(\ell_i), t_0+n_{k_{i}}T_{k_i})\ .
\end{align*}
On the other hand, multiple iterations (for $n_{k_{i}}$ times) in Theorem~\ref{le:AscendDer} imply that
\begin{align*}
\cR(k_{i},c(\ell_i), t_0+n_{k_{i}}T_{k_{i}}) &\lesssim \left(\mu_{k_i}(i)\right)^{n_{k_{i}}/k_i} \cR(k_{i},c(\ell_i), t_0)\ ,
\\
\cR(k_i^*,c(\ell_i), t_0+n_{k_{i}}T_{k_i}) &\lesssim \left(\zeta_{k_i}(i)\right)^{n_{k_{i}}/k_i^*} \cR(k_i^*,c(\ell_i), t_0)\ .
\end{align*}
Note that the right hand side of the second inequality above gives the maximum possible value of $\cR(k_i^*, c(\ell_i), s)$ before $\tilde t$ if \eqref{eq:equileveltime} occurs at $n_{k_{i}}$-th iteration and $\cR(k_i^*,c(\ell_i), s)$ is multiplied by $\left(\zeta_{k_i}(i)\right)^{\frac{1}{k_i^*+1}}$ after each iteration. So the maximal number of iterations $n_{k_{i}}$ that guarantees \eqref{eq:equileveltime} is determined by
\begin{align*}
&\left(\mu_{k_i}(i)\right)^{n_{k_{i}}/k_i} \cR(k_{i},c(\ell_i), t_0) \approx \cR(k_{i},c(\ell_i), t_0+n_{k_{i}}T_{k_{i}})
\\
&\qquad \approx \cR(k_{i}^*,c(\ell_i), t_0+n_{k_{i}}T_{k_{i}}) \approx \left(\zeta_{k_i}(i)\right)^{n_{k_{i}}/k_i^*} \cR(k_i^*,c(\ell_i), t_0)\ .
\end{align*}
From \eqref{eq:ThetaLowerBdd} we know the lower bound for $\theta(k_i^*,k_i,t_0)$ ($=\cR(k_i^*,c(\ell_i), t_0)/\cR(k_{i},c(\ell_i), t_0)\le 1$) is a constant multiple of $(k_i^*!)^{-\frac{1}{k_i^*+1}}\|u_0\|^{\frac{2}{d}-1}$ which is approximately $(k_i^*)^{-1}(1+\epsilon)^{(\frac{2}{d}-1)\ell}$, so
\begin{align*}
\ln\theta(k_i^*,k_i,t_0)\gtrsim -\ln k_i^* + (2/d-1)\ell\cdot \ln(1+\epsilon)\ .
\end{align*}
Recall that $\zeta_{k_i}(i)\approx 1+ (M_{k_i}-1)\cdot \tilde{c}(k_{i})/c(\ell_i)$ and $\left(\zeta_{k_i}(i) \right)^{-k_{i}/k_i^*}\approx 1$, therefore
\begin{align*}
n_{k_{i}} \lesssim \frac{k_{i}\ \ln \theta(k_i^*,k_i,t_0)}{\ln \left(\mu_{k_i}(i) \cdot \left(\zeta_{k_i}(i) \right)^{-k_{i}/k_i^*}\right)} \lesssim  \frac{k_i\ (\ln k_i^* + (1-2/d)\ell)}{-\ln \left(\mu_{k_i}(i)\right)} \
\end{align*}
and $\left(\zeta_{k_i}(i)\right)^{n_{k_{i}}}\le 1+\tilde \epsilon_{k_i}$ for some negligible quantity $\tilde \epsilon_{k_i}$. Since $\cR(k_i^*,c(\ell_i), s)$ is at most multiplied by $\left(\zeta_{k_i}(i)\right)^{\frac{1}{k_i^*+1}}$ within each $[t_0+\nu T_{k_i}, t_0+(\nu+1)T_{k_i}]$,
\begin{align*}
\max_{\ell_i\le j< k_i} \sup_{t_0<s<\tilde t} \cR(j, c(\ell_i), s) &\le \left( \zeta_{k_i}(i) \right)^{n_{k_{i}}/k_i^*} \cR(k_i^*,c(\ell_i), t_0) \le (1+\tilde \epsilon_{k_i})^{1/k_i^*} \cR(k_i^*,c(\ell_i), t_0)\ .
\end{align*}

Now we prove \eqref{eq:RestrAtoB} without assuming $k_i>\ell_{i+q}$, i.e. for each $[\ell_j,\ell_{j+1}]\subset[\ell_i,\ell_{i+q}]$ the maximal index in \eqref{eq:NonDescAtLi} can be less than $\ell_{i+q}$ (the second step). Since $[\ell_i,\ell_{i+1}]$ is a Type-$\cA$ section (recall that we are still in the case where the $k_\tau$'s for each $[\ell_j,\ell_{j+1}]$ are finitely many), similarly to \eqref{eq:FiniteLocalPeaks} one can find the maximal index $k_{\tau_1}\in [\ell_{\tau_1},\ell_{\tau_1+1}]$ (for some $\ell_{\tau_1}>\ell_i$) such that
\begin{align*}
\max_{m_i\le j\le k_{\tau_1}} \sup_{t_0\le s\le \tilde{t}} \theta(j, k_{\tau_1}, \ell_i, s)\le 1\ , \qquad \max_{j> k_{\tau_1}} \sup_{t_0\le s\le \tilde{t}} \theta(j, k_{\tau_1}, \ell_i, s)< 1 \ .
\end{align*}
(The proof is the same if $\tau_1$ is variant as time goes forwards $\tilde t$.) In the rest of the proof we write $\theta(j, \tau, r, s)$ for $\theta(j, k_{\tau}, \ell_r, s)$ if there is no ambiguity. If $\ell_{\tau_1}<\ell_{i+q}$ or $(\ell_i, \ell_{\tau_1})$ does not satisfy the condition~\eqref{eq:AscDerCond}, we repeat the above procedure for $[\ell_{\tau_1},\ell_{\tau_1+1}]$, and in general we find the maximal index $k_{\tau_p}\in [\ell_{\tau_p},\ell_{\tau_p+1}]$ ($\ell_{\tau_p}>\ell_{\tau_{p-1}}$) such that
\begin{align}\label{eq:LocalPeaksAstr}
\max_{m_i \le j\le k_{\tau_p}} \sup_{t_0\le s\le \tilde{t}} \theta(j, \tau_p, \tau_{p-1}, s)\le 1\ , \qquad \max_{j> k_{\tau_p}} \sup_{t_0\le s\le \tilde{t}} \theta(j, \tau_p, \tau_{p-1}, s)< 1 \ ,
\end{align}
until $(\ell_i, \ell_{\tau_n})$ satisfies the condition~\eqref{eq:AscDerCond} or $\ell_{\tau_n}>\ell_{i+q}$, prior to which such $k_{\tau}$ always exists since $[\ell_{\tau},\ell_{\tau+1}]$ is contained in $[\ell_i,\ell_{i+q}]$ which is a Type-$\cA$ string up to $\tilde t$.
Without loss of generality we may assume $k_{\tau_{n}}\in [\ell_{i+q}, \ell_{i+q+1}]$, i.e. $\ell_{\tau_n}=\ell_{i+q}$ (recall that we are still in the case where the $k_\tau$'s for each $[\ell_j,\ell_{j+1}]$ are finitely many) and the order \eqref{eq:LocalPeaksAstr} remains within $[t_0,\tilde t]$, i.e. the indexes $\{\tau_1,\cdots,\tau_n\}$ determined by \eqref{eq:LocalPeaksAstr} remain the same until $\tilde t$, in which case $[\ell_i,\ell_{i+q}]$ switches to a Type-$\cB$ string at $\tilde t$ because one of the sections $[\ell_{\tau_1},\ell_{\tau_1 +1}],\cdots, [\ell_{\tau_{n-1}},\ell_{\tau_{n-1} +1}]$ becomes a Type-$\cB$ section at $\tilde t$.
With the above settings for $k_{\tau}$'s, it follows again by the proof of Theorem~\ref{le:AscendDer} (with $\ell=\ell_i$ and $k=k_{\tau_n}$) that
\begin{align*}
\sup_{t<s<t+T_{k_{\tau_n}}}\!\! \cC(k_{\tau_n},c(\ell_{\tau_{n-1}}), \varepsilon, t, s) \le \left(M_{k_{\tau_n}} + \frac{\tilde{c}(k_{\tau_n})}{c(\ell_{\tau_{n-1}})}\right)^{\frac{1}{k_{\tau_n}+1}} \cR(k_{\tau_n},c(\ell_{\tau_{n-1}}), t)
\end{align*}
($\tau_0:=i$) and, for any $\ell_i\le j<k_{\tau_n}$ and any $1 \le p\le n-1$,
\vspace{-0.05in}
\begin{align*}
&\sup_{t<s<t+T_{k_{\tau_n}}}\!\!\! \cR(j,c(\ell_{\tau_p}), s) \le \left(1+\frac{\tilde{c}(k_{\tau_n})\cdot (M_{k_{\tau_n}}-1)}{c(\ell_{\tau_p})\cdot \theta(j, \tau_{p+1}, \tau_p, t)^{j+1}} \right)^{\frac{1}{j+1}} \cR(j,c(\ell_{\tau_p}), t)\ ,
\end{align*}
where the constant $M_{k_{\tau_n}}<2$ and we used implicitly the fact that $k_{\tau_p}\ge 2 k_{\tau_{p-1}}$ since we assumed (without loss of generality) earlier that each $[\ell_\tau,\ell_{\tau+1}]$ contains at most one $k_\tau$ for \eqref{eq:NonDescAtLi} and from the setup for $\{\ell_j\}$ we know $\ell_{j+1}= \phi(\ell_j) \ge 2\ell_j$. As presented in the previous step, multiple applications of Theorem~\ref{le:AscendDer} and Proposition~\ref{prop:HarMaxPrin} to the above estimates yield
\begin{align*}
\cR(k_{\tau_n},c(\ell_{\tau_{n-1}}), t_0+\nu_{\tau_n}T_{k_{\tau_n}}) &\le \left(\mu_{k_{\tau_n}}(\tau_{n-1})\right)^{\frac{\nu_{\tau_n}}{k_{\tau_n}+1}}\cdot \cR(k_{\tau_n},c(\ell_{\tau_{n-1}}), t_0)\ ,
\\
\cR(j,c(\ell_{\tau_p}), t_0+\nu_{\tau_n}T_{k_{\tau_n}}) &\le \left(\zeta_{k_{\tau_n}}^{(\tau_p)}(k_{\tau_{p+1}},j) \right)^{\frac{\nu_{\tau_n}}{j+1}} \cdot \cR(k_{\tau_{p+1}},c(\ell_{\tau_p}), t_0)\ ,
\end{align*}
for any $\ell_i\le j<k_{\tau_n}$ and any $1 \le p\le n-1$.
Recall that $\{\tau_1,\cdots,\tau_n\}$ in \eqref{eq:LocalPeaksAstr} is retained up to $\tilde t$, so $t_0+\nu_{\tau_n}T_{k_{\tau_n}}$ reaches $\tilde t$ until
\begin{align}\label{eq:TransKpEnd}
\cR(k_{\tau_{p+1}},c(\ell_{\tau_p}), t_0+\nu_{\tau_n}T_{k_{\tau_n}})\le \cR(k_{\tau_p},c(\ell_{\tau_p}), t_0+\nu_{\tau_n}T_{k_{\tau_n}})\ .
\end{align}
As the order \eqref{eq:LocalPeaksAstr} remains until $\tilde t$, the maximum possible value of $\cR(k_{\tau_p}, c(\ell_{\tau_p}), \tilde t)$ is achieved only if $\cR(k_{\tau_{p+1}}, c(\ell_{\tau_p}), \tilde t)$ attains its possible maximum.
By a recursive argument, any of $\cR(k_{\tau_p}, c(\ell_{\tau_p}), \tilde t)$'s ($1 \le p\le n-1$) reaches its possible maximum only if $\cR(k_{\tau_{n-1}}, c(\ell_{\tau_{n-2}}), \tilde t)$ reaches its possible maximum, i.e. it is multiplied by $\left(\zeta_{k_{\tau_n}}(\tau_{n-2})\right)^{\frac{1}{k_{\tau_{n-1}}+1}}$ after each iteration. With the same reasoning as in the previous step, the maximal number of iterations $\nu_{\tau_n}$ that guarantees \eqref{eq:TransKpEnd} ($p=n-1$) is determined by
\begin{align*}
&\left(\mu_{k_{\tau_n}}(\tau_{n-1})\right)^{\frac{\nu_{\tau_n}}{k_{\tau_n}}} \cR(k_{\tau_n},c(\ell_{\tau_{n-1}}), t_0) \approx \cR(k_{\tau_n},c(\ell_{\tau_{n-1}}), t_0+\nu_{\tau_n}T_{k_{\tau_n}})
\\
&\quad \approx \cR(k_{\tau_{n-1}},c(\ell_{\tau_{n-1}}), t_0+\nu_{\tau_n}T_{k_{\tau_n}}) \approx \left(\frac{c(\ell_{\tau_{n-2}})}{c(\ell_{\tau_{n-1}})}\right)^{\frac{k_{\tau_{n-1}}}{k_{\tau_{n-1}}+1}}\! \cR(k_{\tau_{n-1}},c(\ell_{\tau_{n-2}}), t_0+\nu_{\tau_n}T_{k_{\tau_n}})
\\
&\qquad \approx \left(c(\ell_{\tau_{n-2}})/c(\ell_{\tau_{n-1}})\right)^{\frac{k_{\tau_{n-1}}}{k_{\tau_{n-1}}+1}} \left(\zeta_{k_{\tau_n}}(\tau_{n-2})\right)^{\frac{\nu_{\tau_n}}{k_{\tau_{n-1}}}} \cR(k_{\tau_{n-1}},c(\ell_{\tau_{n-2}}), t_0)
\\
&\qquad \quad \approx \left(\zeta_{k_{\tau_n}}(\tau_{n-2})\right)^{\frac{\nu_{\tau_n}}{k_{\tau_{n-1}}}} \cR(k_{\tau_{n-1}},c(\ell_{\tau_{n-1}}), t_0)\ ,
\end{align*}
thus $\nu_{\tau_n}$ has upper bound
\begin{align*}
\nu_{\tau_n} \lesssim \frac{k_{\tau_n}\ \ln \theta(k_{\tau_{n-1}},\tau_n,\tau_{n-1},t_0)} {\ln \left(\mu_{k_{\tau_n}}(\tau_{n-1}) \cdot \left(\zeta_{k_{\tau_n}}(\tau_{n-2}) \right)^{-k_{\tau_n}/k_{\tau_{n-1}}}\right)} \lesssim  \frac{k_{\tau_n} (\ln k_{\tau_{n-1}} + (1-2/d)\ell)}{-\ln \left(\mu_{k_{\tau_n}}(\tau_{n-1})\right)} \ .
\end{align*}
For the same reason as in the previous step,
\begin{align}\label{eq:SharpBddAstr}
\max_{\ell_i\le j<k_{\tau_n}} \sup_{t_0<s<\tilde t} \cR(j, c(\ell_{\tau_{n-2}}), s) &\le (1+\tilde \epsilon_{k_{\tau_n}})^{1/k_{\tau_{n-1}}} \cR(k_{\tau_{n-1}},c(\ell_{\tau_{n-2}}), t_0) \
\end{align}
as $\left(\zeta_{k_{\tau_n}}(\tau_{n-2})\right)^{\nu_{\tau_n}}\le 1+\tilde \epsilon_{k_{\tau_n}}$ with some negligible quantity $\tilde \epsilon$, while, for any $1\le p\le n-3$,
\begin{align*}
\max_{\ell_i\le j<k_{\tau_n}} \sup_{t_0<s<\tilde t} \cR(j, c(\ell_{\tau_p}), s) &\le \sup_{t_0<s<\tilde t} \cR(k_{\tau_{p+1}},c(\ell_{\tau_p}), s)\
\end{align*}
as \eqref{eq:LocalPeaksAstr} is preserved until $\tilde t$. Recursion of the above inequalities together with \eqref{eq:SharpBddAstr} leads to, for any $k_{\tau_{\nu}}\le j<k_{\tau_{\nu+1}}$ with $0\le \nu\le n-2$ ($\tau_0:=i$),
\begin{align*}
\sup_{t_0<s<\tilde t} \cR(j, c(\ell_{\tau_\nu}), s) &\le \prod_{\nu+1\le p\le n-2} \left(\frac{c(\ell_{\tau_p})}{c(\ell_{\tau_{p-1}})} \right)^{k_{\tau_p}/(k_{\tau_p}+1)}\!\! \sup_{t_0<s<\tilde t} \cR(k_{\tau_{n-1}}, c(\ell_{\tau_{n-2}}), s)
\\
&\lesssim  \prod_{\nu+1\le p\le n-2} \left(\frac{c(\ell_{\tau_p})}{c(\ell_{\tau_{p-1}})} \right)^{k_{\tau_p}/(k_{\tau_p}+1)} \cR(k_{\tau_{n-1}}, c(\ell_{\tau_{n-2}}), t_0)\ .
\end{align*}
The estimates for $k_{\tau_{n-1}}\le j<k_{\tau_n}$ was already obtained in \eqref{eq:SharpBddAstr}. For the terminal index $k_{\tau_n}$, referring to \eqref{eq:SharpBddAstr} again,
\vspace{-0.1in}
\begin{align*}
&\cR(k_{\tau_n},c(\ell_{\tau_{n-1}}), \tilde t) \le \cR(k_{\tau_{n-1}},c(\ell_{\tau_{n-1}}), \tilde t) \le \left(\frac{c(\ell_{\tau_{n-2}})}{c(\ell_{\tau_{n-1}})}\right)^{\frac{k_{\tau_{n-1}}}{k_{\tau_{n-1}}+1}} \cR(k_{\tau_{n-1}},c(\ell_{\tau_{n-2}}), \tilde t)
\\
&\qquad\qquad\qquad\le \left(c(\ell_{\tau_{n-2}})/c(\ell_{\tau_{n-1}})\right)^{\frac{k_{\tau_{n-1}}}{k_{\tau_{n-1}}+1}} \cR(k_{\tau_{n-1}},c(\ell_{\tau_{n-2}}), \tilde t) \lesssim \cR(k_{\tau_{n-1}},c(\ell_{\tau_{n-1}}), t_0)\ .
\end{align*}

\medskip

Lastly, we establish the lemma if all the sections after $\ell_{i+q}$ are of Type-$\cA$. For simplicity we assume that the indexes $k_i$ for $[\ell_i,\ell_{i+1}]$ in \eqref{eq:NonDescAtLi} are infinitely many, which implies that the indexes $k_\tau$'s as described in \eqref{eq:LocalPeaks} are infinitely many as well and we can pick an infinite sequence $\{k_p\}_{p=i}^\infty$ such that $k_{p+1}-k_p\ge \ell_{p+1}-\ell_p$ (by the settings for $k_\tau$'s).
Without loss of generality, we assume $\{k_p\}$ remains until $\tilde t$ (otherwise we rearrange the indexes $k_p$'s at the temporal point when the order in \eqref{eq:LocalPeaks} stops).
Similarly to the previous step of the proof, by the proof of Theorem~\ref{le:AscendDer} with $\ell=k_{p-q}$ and $k=k_p$, for any $p\ge i+q$,
\begin{align*}
\sup_{t<s<t+T_{k_p}}\!\! \cC(k_p,c(\ell_{k_{p-1}}), \varepsilon, t, s) \le \left(M_{k_p} + \frac{\tilde{c}(k_p)}{c(k_{p-1})}\right)^{\frac{1}{k_p+1}} \cR(k_p,c(k_{p-1}), t),
\end{align*}
and for any $\ell_i\le j< k_{i+q}$ and any $i\le r<i+q$,
\begin{align*}
&\sup_{t<s<t+T_{k_p}}\!\!\! \cR(j,c(k_r), s) \le \left(1+\frac{\tilde{c}(k_p)\cdot (M_{k_p}-1)}{c(k_r) \cdot \theta(j, k_{r+1}, k_r, t)^{j+1}} \right)^{\frac{1}{j+1}} \cR(j,c(k_r), t) \
\end{align*}
where the constant $M_{k_p}<2$.
By Proposition~\ref{prop:HarMaxPrin} (applied to $D^{k_p}u_t$) and the above estimates,
\begin{align*}
\cR(k_p, c(k_{p-1}), t+T_{k_p}) \le \left(\mu_{k_p}(k_{p-1})\right)^{\frac{1}{k_p+1}} \cR(k_p, c(k_{p-1}), t) \ , \qquad \forall\ p\ge i+q \ .
\end{align*}
Recall that it is shown in the previous step that the maximal number of iterations which guarantees
\begin{align*}
\cR(k_p,c(k_{p-1}), t_0+\nu_{k_p}T_{k_p})\le \cR(k_{p-1},c(k_{p-1}), t_0) \
\end{align*}
is $\displaystyle \nu_{k_p} \lesssim - k_p \ln k_{p-1}/\ln \mu_{k_p}$. Notice that $t_0+\nu_{k_p}T_{k_p}\ll t_0+T_{k_{p-1}}$ because
\begin{align*}
\frac{T_{k_{p-1}}}{T_{k_p}} &= \frac{\tilde{c}(k_{p-1})^2}{\tilde{c}(k_p)^2} \cdot \frac{(M_{k_{p-1}}-1)^{-2}}{(M_{k_p}-1)^{-2}} \cdot \frac{\|D^{k_p}u(t)\|^{\frac{2}{k_p+1}}} {\|D^{k_{p-1}}u(t)\|^{\frac{2}{k_{p-1}+1}}}
\\
&\gtrsim \frac{\tilde{c}(k_{p-1})^2\ c(k_{p-1})^{\frac{2k_{p-1}}{k_{p-1}+1}}k_p^2} {\tilde{c}(k_p)^2\ c(k_{p-1})^{\frac{2k_p}{k_p+1}}k_{p-1}^2} \cdot \theta(k_{p-1}, k_p, t_0)^{-2} \gtrsim 2^{k_p-k_{p-1}}\left(\cB_{k_p, k_{p-q}}\right)^{k_p-k_{p-1}} \
\end{align*}
which is essentially greater than $\nu_{k_p}$ for large $k_p$.
Therefore, for each $p>i+q$, the time span $T_{k_{p-1}}$ is sufficient for $\cR(k_p,c(k_{p-1}),t)$ to decrease to a level comparable to $\cR(k_{p-1},c(k_{p-1}), t_0)$, that is
\begin{align*}
\cR(k_p,c(k_{p-1}), t_0+T_{k_{p-1}}) \le \cR(k_{p-1},c(k_{p-1}), t_0) \ .
\end{align*}
A recursion (with index $p$) of this argument shows that, within $[t_0,t_0+T_{k_{i+q}}]$, all $\cR(k_p,c(k_{p-1}), t)$'s with $p>i+q$ decreases to a level comparable to $\cR(k_{i+q},c(k_{i+q}), t_0)$, more precisely,
\begin{align*}
\max_{p>i+q} \cR(k_p,c(k_{p-1}), t_0+T_{k_{i+q}}) \le \cB_{k_p, k_{i+q}} \cdot \cR(k_{i+q},c(k_{i+q}), t_0) \ .
\end{align*}
At the same time, by the estimates in the previous step, we know
\begin{align*}
\max_{\ell_i\le j<k_{i+q}} \sup_{t_0<s<t_0+T_{k_{i+q}}}\!\! \cR(j, c(k_{i+q}), s) &\le (1+\tilde \epsilon_{k_{i+q}})^{1/k_{i+q}} \cdot \cR(k_{i+q}, c(k_{i+q}), t_0) \
\end{align*}
with some negligible quantity $\tilde \epsilon_{k_{i+q}}$. Then, the argument in the previous step (within the string $[\ell_i,\ell_{i+q}]$) leads to \eqref{eq:RestrAtoB}.

\end{proof}

\begin{lemma}\label{le:BstrBdd}
Suppose $\sup_{t>t_0}\|u(t)\|\lesssim (1+\epsilon)^{\ell_i}$ and \eqref{eq:ParaAdjMaxP} is satisfied at any temporal point with $\ell=\ell_i$ and $(k,c(k))=(\ell_p,c(\ell_p))$ for any $i\le p\le i+q$. If a string $[\ell_i,\ell_{i+q}]$ is of Type-$\cB$ at an initial time $t_0$, then for any $i\le r<i+q$,
\begin{align}\label{eq:RestrBtoA}
\max_{\ell_r\le j\le \ell_{i+q}} \sup_{t_0<s<\tilde t} \cR(j, c(\ell_r), s) \le  \max_{r\le p\le i+q} \cR(m_p,c(\ell_p), t_0)
\end{align}
where $\tilde t$ is the first time when $[\ell_i,\ell_{i+q}]$ switches to a Type-$\cA$ string; we set $\tilde t= \infty$ if $[\ell_i,\ell_{i+q}]$ is always of Type-$\cB$.
\end{lemma}

\begin{proof}
Let $p_1$ be the index for the maximal one in $\{\cR(m_\nu,c(\ell_\nu), t_0)\}_{i\le \nu\le i+q}$, and in general let $p_{j+1}$ be the index for the maximal one in $\{\cR(m_\nu, c(\ell_\nu), t_0)\}_{p_j< \nu\le i+q}$ (pick the minimal $p_j$ if not unique).
In fact, all $[\ell_{p_j},\ell_{p_j+1}]$'s are Type-$\cB$ sections at $t_0$ because for any $\ell_\nu\le r\le \ell_{\nu +1}$ with $\nu>p_j$
\begin{align*}
\cR(r, c(\ell_{p_j}), t_0) < \cR(r, c(\ell_\nu), t_0) \le \cR(m_\nu, c(\ell_\nu), t_0) \le \cR(m_{p_j}, c(\ell_{p_j}), t_0)
\end{align*}
where we used the definition of $m_\nu$ and the fact that $c(\ell_\nu)<c(\ell_{p_j})$.

Let $t_1(r)$ be the first time that $[\ell_r, \ell_{r+1}]$ switches to a Type-$\cA$ section if $[\ell_r, \ell_{r+1}]$ is a Type-$\cB$ section at $t_0$.
Note that, by how $t_1(r)$ is defined, the assumption~\eqref{eq:DescDerOrd} in Theorem~\ref{le:DescendDer} is satisfied for all $s<t_1(p_j)$ with $\mu_*\le 1$ and $k=\ell_{p_j}$, so by the theorem, for any $p_j$,
\begin{align}\label{eq:TypeBTailRestr}
\max_{r\ge \ell_{p_j}} \cR(r,c(\ell_{p_j}),s) \le \cR(m_{p_j},c(\ell_{p_j}),s) \le \cR(m_{p_j},c(\ell_{p_j}),t_0) \ , \qquad \forall\ s\le t_1(p_j)\ .
\end{align}
By how we pick the indexes $p_j$'s,
\begin{align*}
\max_{p_{j-1}<\nu<p_j} \cR(m_\nu, c(\ell_\nu), t_0) \le \cR(m_{p_j},c(\ell_{p_j}),t_0) \
\end{align*}
and by Corollary~\ref{cor:ScaleBound}, for any $p_j$,
\begin{align}\label{eq:TypeBHeadRestr}
\max_{p_{j-1}<\nu<p_j}\! \cR(m_\nu, c(\ell_\nu), s) \le \max_{r\ge \ell_{p_j}} \cR(r,c(\ell_{p_j}),s) \le \cR(m_{p_j}, c(\ell_{p_j}), t_0) \ , \quad \forall\ s\le t_1(p_j) \ .
\end{align}
If $t_1(p_j)<t_1(p_{j+1})$, then the maximal index $k$ for $[\ell_{p_j},\ell_{p_j+1}]$ in \eqref{eq:NonDescAtLi} is contained in $[\ell_{p_j+1},\ell_{p_{j+1}+1}]$ (otherwise $t_1(p_j)>t_1(p_{j+1})$).
Suppose $k\in [\ell_\nu , \ell_{\nu+1}]$ ($\nu\le p_{j+1}$); by how $t_1(p_j)$ is defined and \eqref{eq:TypeBHeadRestr},
\begin{align*}
&\cR(m_{p_j},c(\ell_{p_j}),t_1(p_j)) = \cR(k,c(\ell_{p_j}),t_1(p_j))
\\
&\qquad\qquad\qquad \le \cR(k,c(\ell_\nu),t_1(p_j)) \le  \max_{r\ge \ell_{p_{j+1}}}\!\! \cR(r,c(\ell_{p_{j+1}}),t_1(p_j)) \ .
\end{align*}
Then, by the above result, Corollary~\ref{cor:ScaleBound} and \eqref{eq:TypeBHeadRestr},
\begin{align*}
\max_{p_j\le\nu<p_{j+1}} \cR(m_\nu, c(\ell_\nu), s) \le \max_{r\ge \ell_{p_{j+1}}}\!\! \cR(r,c(\ell_{p_{j+1}}),s) \ , \qquad \forall\ t_1(p_j)\le s\le t_1(p_{j+1}) \ .
\end{align*}
Let $t_{\max}^{(1)}=\max\{t_1(p_j), t_1(p_{j+1})\}$. Collecting \eqref{eq:TypeBTailRestr}, \eqref{eq:TypeBHeadRestr} and the above estimate yields
\begin{align*}
&\max_{\nu \ge p_j} \sup_{t_0<s<t_{\max}^{(1)}} \cR(m_\nu,c(\ell_\nu),s)
\\
&\quad \le \max\left\{\max_{\nu \ge p_{j+1}} \sup_{t_0<s<t_1(p_{j+1})} \! \cR(m_\nu,c(\ell_\nu),s) \ , \ \max_{p_j\le \nu< p_{j+1}} \sup_{t_0<s<t_1(p_{j+1})} \! \cR(m_\nu,c(\ell_\nu),s) \right\}
\\
&\quad \le \max\left\{ \cR(m_{p_{j+1}},c(\ell_{p_{j+1}}),t_0) \ , \ \max_{p_j\le \nu< p_{j+1}} \sup_{t_0<s<t_1(p_j)} \!\! \cR(m_\nu,c(\ell_\nu),s) \ , \right.
\\
&\qquad\qquad\qquad \left. \max_{p_j\le \nu< p_{j+1}} \sup_{t_1(p_j)<s<t_1(p_{j+1})} \!\! \cR(m_\nu,c(\ell_\nu),s) \right\} \le \cR(m_{p_j},c(\ell_{p_j}),t_0) \ .
\end{align*}
If $t_1(p_j)>t_1(p_{j+1})$, then \eqref{eq:TypeBTailRestr} already shows
\begin{align*}
&\max_{\nu \ge p_j} \sup_{t_0<s<t_{\max}^{(1)}}\! \cR(m_\nu,c(\ell_\nu),s) = \max_{\nu \ge p_j} \sup_{t_0<s<t_1(p_j)}\! \cR(m_\nu,c(\ell_\nu),s) \le \cR(m_{p_j},c(\ell_{p_j}),t_0).
\end{align*}
Now let $t_{\max}^{(1)}=\max_{j\ge\alpha} t_1(p_j)$. A recursion (backward) of the above argument leads to
\begin{align*}
&\max_{\nu \ge p_\alpha} \sup_{t_0<s<t_{\max}^{(1)}} \cR(m_\nu,c(\ell_\nu),s) \le \cR(m_{p_\alpha},c(\ell_{p_\alpha}),t_0) \ .
\end{align*}
In particular, if $\alpha=1$, the above result indicates
\begin{align*}
&\max_{\nu \ge p_1} \sup_{t_0<s<t_{\max}^{(1)}} \cR(m_\nu,c(\ell_\nu),s) \le \cR(m_{p_1},c(\ell_{p_1}),t_0) \
\end{align*}
where $t_{\max}^{(1)}=\max_{j\ge1} t_1(p_j)$.
If $t_{\max}^{(1)}=\tilde t$ then the proof is complete. If $t_{\max}^{(1)}<\tilde t$, then we pick new indexes $\tilde{p}_j$'s as shown at the beginning of the proof and repeat the above argument again until some $t_{\max}^{(n)}=\tilde t$, and in each process all $\cR(m_\nu,c(\ell_\nu),s)$'s are restricted by some $\cR\left(m_{\tilde{p}_1},c(\ell_{\tilde{p}_1}),t_{\max}^{(n)}\right)$ which is less than $\cR(m_{p_1},c(\ell_{p_1}),t_0)$.

\end{proof}

%
%

\begin{lemma}\label{le:MaxIndxAStr}
Suppose $\sup_{t>t_0}\|u(t)\|\lesssim (1+\epsilon)^{\ell_i}$ and \eqref{eq:ParaAdjMaxP} is satisfied at any temporal point with $\ell=\ell_i$ and $(k,c(k))=(\ell_p,c(\ell_p))$ for any $i\le p\le i+q$.
If a string $[\ell_i,\ell_{i+q}]$ is of Type-$\cB$ at an initial time $t_0$ and $\tilde t$ is the first time when it switches to a Type-$\cA$ string, then the index $k_p$ described in \eqref{eq:NonDescAtLi} for any $i\le p\le i+q$ has a maximum; more precisely, with the notation in the proof of Lemma~\ref{le:AstrBdd} and $p^*$ being the index for the maximum in $\{\cR(m_p, c(\ell_p), \tilde t)\}_{i\le p\le i+q}$, there exists an index $k_*$ such that
\begin{align*}
\max_{m_{p^*}\le j\le k_*}\theta(j, m_{p^*}, c(\ell_{p^*}), \tilde t) \le 1 \ , \qquad \theta(j, m_{p^*}, c(\ell_{p^*}), \tilde t) <1 \ , \quad \forall j>k_* \ ,
\end{align*}
and at $j=k_*$, in particular, $\theta(k_*, m_{p^*}, c(\ell_{p^*}), \tilde t) =1$. Moreover, $k_*\le \ell_{i+3q}$.
\end{lemma}

\begin{proof}
Suppose there is $k_*> \ell_{i+3q}$ as described above and without loss of generality we assume $k_*\in [\ell_{i+3q},\ell_{i+3q+1}]$. Then $[\ell_{i+q},\ell_{i+2q}]$ and $[\ell_{i+2q},\ell_{i+3q}]$ are both Type-$\cA$ strings at $\tilde t$ and
\begin{align*}
\max_{m_{p^*}\le j < k_*} \cR(j,c(\ell_{p^*}),\tilde t) \le \cR(k_*,c(\ell_{p^*}),\tilde t) \ ;
\end{align*}
thus $\max_{m_{p^*}\le j < k_*} \cR(j,c(\ell_{i+q}),\tilde t) < \cR(k_*,c(\ell_{i+q}),\tilde t)$ as $c(\ell_{i+q})<c(\ell_{p^*})$, which implies
\begin{align*}
\max_{\ell_{i+2q} \le j < k_*} \cR(j, c(\ell_{i+2q}), \tilde t) < \left(\frac{c(\ell_{i+q})}{c(\ell_{i+2q})}\right)^{\frac{1}{k_*+1}-\frac{1}{k_*}} \cR(k_*, c(\ell_{i+2q}), \tilde t) \ .
\end{align*}
Then by the continuity of $D^ju$'s there must exist a temporal point $t_*<\tilde t$ such that
\begin{align*}
\max_{\ell_{i+2q} \le j < k_*} \cR(j, c(\ell_{i+2q}), t_*) \le \left(\frac{c(\ell_{i+q})}{c(\ell_{i+2q})}\right)^{\frac{1}{k_*+1}-\frac{1}{k_*}} \cR(k_*, c(\ell_{i+2q}), t_*) \
\end{align*}
which means $[\ell_{i+2q},\ell_{i+3q}]$ is already a Type-$\cA$ string before $\tilde t$, and by the proof of Lemma~\ref{le:AstrBdd}
\begin{align*}
\max_{\ell_r\le j\le \ell_{i+3q}} \sup_{t_*<s< \hat{t}} \cR(j, c(\ell_r), s) \lesssim \Theta(p^*,r) \max_{i+2q\le p\le i+3q} \cR(m_p,c(\ell_p), t_*) \ ,
\end{align*}
for any $i+2q\le r\le i+3q$, where $\hat{t}$ is the first time $[\ell_{i+2q},\ell_{i+3q}]$ switches to a Type-$\cB$ string, while
\begin{align*}
\cR(k_*, c(\ell_{i+2q}), \hat{t}) \lesssim \max_{i+2q\le p\le i+3q} \cR(m_p,c(\ell_p), t_*) \ .
\end{align*}
If $\hat t <\tilde t$, then by Lemma~\ref{le:BstrBdd} (or Corollary~\ref{cor:ScaleBound})
\begin{align*}
\max_{i\le p\le i+q} \cR(m_p,c(\ell_p), \hat t) \le \max_{i\le p\le i+q} \cR(m_p,c(\ell_p), t_0) \ ,
\end{align*}
and combination of Lemma~\ref{le:AstrBdd} and Lemma~\ref{le:BstrBdd} guarantees
\begin{align*}
\cR(k_*, c(\ell_{i+2q}), s) \lesssim \max_{i+2q\le p\le i+3q} \cR(m_p, c(\ell_{i+2q}), s)  \ , \qquad \forall\ \hat t <s< \tilde t \ .
\end{align*}
In particular, at $\tilde t$, the above restriction contradicts with $\theta(k_*, m_{p^*}, c(\ell_{p^*}), \tilde t) =1$.

If $\hat t >\tilde t$ ($>t_*$), then by Lemma~\ref{le:BstrBdd}
\begin{align*}
\max_{i\le p\le i+q} \cR(m_p,c(\ell_p), \tilde t) \le \max_{i\le p\le i+q} \cR(m_p,c(\ell_p), t_*) \ .
\end{align*}
In general, $\displaystyle\max_{i\le p\le i+q} \cR(m_p,c(\ell_p), s)$ is decreasing within $[t_*, \tilde t]$ as $[\ell_i,\ell_{i+q}]$ is of Type-$\cB$ before $\tilde t$.
Without loss of generality we may assume $k_*$ is invariant in time until $\hat t$; then by the same reasoning as in the proof of Lemma~\ref{le:AstrBdd}, $\cR(k_*, c(\ell_{i+2q}), s)$ is, in general, decreasing (with negligible perturbation) within $[t_*, \hat t]$, more precisely, each time Proposition~\ref{prop:HarMaxPrin} is applied within $[t, t+T_{k_*}]$,
\begin{align*}
\cR(k_*, c(\ell_{i+2q}), t+T_{k_*}) \le \left(\mu_{k_*}(i+2q)\right)^{\frac{1}{k_*+1}} \cR(k_*, c(\ell_{i+2q}), t)
\end{align*}
where $\mu_{k_*}(i+2q)<1$ is a constant defined in the proof of Lemma~\ref{le:AstrBdd}. Assuming $\hat t >\tilde t$, there exists a temporal point $t_*+\tilde \nu T_{k_*}<\tilde t$ such that (for convenience we write $p$ for $p^*$)
\begin{align*}
\cR(m_p, c(\ell_{i+2q}), t_*+\tilde \nu T_{k_*}) &\approx \cR(k_*, c(\ell_{i+2q}), t_*+\tilde \nu T_{k_*}) \ .
\end{align*}
By $\tilde \nu$ times iterations of Proposition~\ref{prop:HarMaxPrin} and the above result,
\begin{align*}
\cR(m_p, c(\ell_{i+2q}), t_*+\tilde \nu T_{k_*}) \le  \left(\mu_{k_*}(i+2q)\right)^{\frac{\tilde \nu}{k_*+1}} \cR(k_*, c(\ell_{i+2q}), t_*) \ .
\end{align*}
Recall that $\cR(k_*, c(\ell_{p^*}), t_*) \le \cR(m_p, c(\ell_{p^*}), t_*)$ and $c(\ell_{i+2q})<c(\ell_{p^*})$, so
\begin{align}\label{eq:MaxDecreBstr}
\cR(m_p, c(\ell_{i+2q}), t_*+\tilde \nu T_{k_*}) \le \left(\mu_{k_*}(i+2q)\right)^{\frac{\tilde \nu}{k_*+1}} \cR(m_p, c(\ell_{i+2q}), t_*) \ .
\end{align}
On the other hand,
\begin{align*}
\|D^{m_p}u(t_*)\| - &\|D^{m_p}u(t_*+\tilde \nu T_{k_*})\| \le \tilde \nu T_{k_*} \sup_{t_*<s<t_*+\tilde \nu T_{k_*}} \|D^{m_p+1}u(s)\|
\\
& \le \tilde \nu T_{k_*} \sup_{t_*<s<t_*+\tilde \nu T_{k_*}} \left(c(\ell_{p})^{\frac{\ell_{p}}{\ell_{p}+1}} (\ell_{p}!)^{\frac{1}{\ell_{p}+1}} \cR\left(m_p+1, c(\ell_{p}), s\right) \right)^{m_p+2}
\\
& \qquad \le \tilde \nu T_{k_*} \left(c(\ell_{p})^{\frac{\ell_{p}}{\ell_{p}+1}} (\ell_{p}!)^{\frac{1}{\ell_{p}+1}} \cR\left(m_p, c(\ell_{p}), t_*\right) \right)^{m_p+2}
\end{align*}
which implies
\begin{align*}
1- \left[\frac{\cR(m_p, c(\ell_p), t_*+\tilde \nu T_{k_*})}{\cR(m_p, c(\ell_p), t_*)}\right]^{m_p+1} \le \tilde \nu T_{k_*} \cdot c(\ell_{p})^{\frac{\ell_{p}}{\ell_{p}+1}} (\ell_{p}!)^{\frac{1}{\ell_{p}+1}} \cR\left(m_p, c(\ell_{p}), t_*\right) \ .
\end{align*}
Recall that $T_{k_*}\lesssim 2^{-2k_*}\left(\cB_{\ell_{i+3q}, \ell_{i+2q}}\right)^{-2k_*} \|D^{k_*}u(t)\|^{-\frac{2}{k_*+1}}$, and the above result contradicts with \eqref{eq:MaxDecreBstr} for any $\tilde \nu$-value.

\end{proof}

\begin{proof}[Proof of Theorem~\ref{th:MainResult}]
The proof is organized as follows. As have been shown in the above three lemmas, a string gets stabilized either by the assumption~\eqref{eq:kRegScale} starting from a Type-$\cA$ string or by Theorem~\ref{le:DescendDer} (and Corollary~\ref{cor:ScaleBound}) from a Type-$\cB$ string. In the following we will prove that on one hand all the higher order derivatives remain within certain ranges up to $T^*$ as a result of the dynamical restriction of a single type or mixing types of strings by Theorem~\ref{le:AscendDer} with the assumption~\eqref{eq:kRegScale}; on the other hand the lower order derivatives are restricted by Corollary~\ref{cor:ScaleBound} and interpolation (Lemma~\ref{le:GNIneq}), thus establishing the solution is regular on $[t_0,T^*]$.

\medskip

Define $\hat{\ell}_i:=\ell_{iq}$ and $\hat{m}_i:=m_{\hat{p}_i}$ where $\hat{p}_i$ is the minimal index within $\{iq, \cdots, (i+1)q\}$ such that
\begin{align*}
\cR\left(m_{\hat{p}_i},c(\ell_{\hat{p}_i}), t\right) = \max_{iq\le p\le (i+1)q} \cR(m_p,c(\ell_p), t) \ .
\end{align*}
Note that $\hat{p}_i(t)$ and $\hat{m}_i(t)$ may be variant in time, and we will always assume $\hat{p}_i$ and $\hat{m}_i$ correspond to the temporal point $t$ in $\cR\left(\cdot,\cdot, t\right)$ if there is no ambiguity.
Let $\hat{t}_1(i)$ be the first time when $[\hat{\ell}_i, \hat{\ell}_{i+1}]$ switches to a Type-$\cA$ string if it is of Type-$\cB$ at $t_0$ (in particular, $\hat{t}_1(i)=t_0$ if $[\hat{\ell}_i, \hat{\ell}_{i+1}]$ is of Type-$\cA$ at $t_0$) and let $\tilde{t}_1(i)$ be the first time when $[\hat{\ell}_i, \hat{\ell}_{i+1}]$ switches to a Type-$\cB$ string after $\hat{t}_1(i)$.
Inductively, we let $\hat{t}_n(i)$ (resp. $\tilde{t}_n(i)$) be the first time when $[\hat{\ell}_i, \hat{\ell}_{i+1}]$ switches to a Type-$\cA$ (resp. Type-$\cB$) string after $\tilde{t}_{n-1}(i)$ (resp. after $\hat{t}_n(i)$).



We will verify in the proof step by step that Lemma~\ref{le:AstrBdd} and Lemma~\ref{le:BstrBdd} are applicable for all $i$ by showing $\sup_{t_0<t<T^*}\|u(t)\|\lesssim (1+\epsilon)^{\ell}$.
With this and the assumption~\eqref{eq:kRegScale} for all $k\ge \ell_0$, in particular, for $i=0$, the proofs of Lemma~\ref{le:BstrBdd} and Lemma~\ref{le:AstrBdd} indicate that, for any $0\le r\le q$,
\begin{align}
\max_{\ell_r\le j\le \ell_q} \sup_{t_0<s<\hat{t}_1(0)} \cR(j, c(\ell_r), s) &\le  \max_{r\le p\le q} \cR(m_p, c(\ell_p), t_0) \ , \notag
\\
\max_{\ell_r\le j\le \ell_q} \sup_{\hat{t}_1(0)<s<\tilde{t}_1(0)} \cR(j, c(\ell_r), s) &\le (1+\tilde{\epsilon}_{\ell_q})^{1/\ell_q} \Theta(\hat{p}_0,r) \cdot \cR(\hat{m}_0, c(\ell_{\hat{p}_0}), \hat{t}_1(0)) \ , \label{eq:StrN0B1}
\end{align}
where $\Theta(\hat{p}_0,r)$ is a constant given by Lemma~\ref{le:AstrBdd}. (Note that the first estimate above can be trivial in a sense that $\hat{t}_1(i)=t_0$.) Application of the above results at $\hat{t}_1(0)$ yields, for any $0\le r\le q$,
\begin{align*}
\max_{\ell_r\le j\le \ell_q} \sup_{t_0<s<\tilde{t}_1(0)} \cR(j, c(\ell_r), s) &\le (1+\tilde{\epsilon}_{\ell_q})^{1/\ell_q} \Theta(\hat{p}_0,r) \cdot \cR(\hat{m}_0, c(\ell_{\hat{p}_0}), t_0) \ .
\end{align*}
In particular, $\Theta(\hat{p}_0,0)\le (1+\tilde{\epsilon}_{\ell_q})^{1/\ell_q} \cB_{\ell_{\hat{p}_0}, \ell}\cdot c(\ell_{\hat{p}_0})/c(\ell)$ where $\tilde \epsilon_{\ell_q}$ is a small quantity given explicitly in the proof of Lemma~\ref{le:AstrBdd}.
By Lemma~\ref{le:GNIneq} and the above result, for any $t_0<s< \tilde{t}_1(0)$,
\begin{align*}
\|u\left(s\right)\| &\lesssim \|u_0\|_2 \|D^{\ell}u\left(s\right)\|^{\frac{d/2}{\ell+d/2}} \notag
\\
&\lesssim \|u_0\|_2 \left(c(\ell)^{\frac{\ell}{\ell+1}} (\ell!)^{\frac{1}{\ell+1}} \cR\left(\ell, c(\ell), s\right) \right)^{\frac{(d/2)(\ell+1)}{\ell+d/2}}  \notag
\\
&\lesssim \|u_0\|_2 \left(c(\ell)^{\frac{\ell}{\ell+1}} (\ell!)^{\frac{1}{\ell+1}} \cB_{\ell_{\hat{p}_0}, \ell} \cdot \cR\left(\hat{m}_0, c(\ell), t_0\right)\right)^{\frac{(d/2)(\ell+1)}{\ell+d/2}} \ .  
\end{align*}
At the same time, by Theorem~\ref{th:LinftyIVP} and the assumption for $\|u_0\|$ imposed by Theorem~\ref{th:MainResult}, we may assume without loss of generality that
\begin{align*}
\|D^{\hat{m}_0}u_0\|^{\frac{1}{\hat{m}_0+1}} \lesssim (\hat{m}_0!)^{\frac{1}{\hat{m}_0+1}} \|u_0\| \lesssim (\hat{m}_0!)^{\frac{1}{\hat{m}_0+1}} (1+\epsilon)^{(2/d)\ell} \ .
\end{align*}
Combining the above results yields
\begin{align*}
\sup_{t_0 <s< \tilde{t}_1(0)} \|u(s)\| &\lesssim \|u_0\|_2 \left(c(\ell)^{\frac{1}{\hat{m}_0+1} - \frac{1}{\ell+1}} (\ell!)^{\frac{1}{\ell+1}} \cB_{\ell_{\hat{p}_0}, \ell} \cdot (1+\epsilon)^{(2/d)\ell} \right)^{d/2}
\\
&\lesssim_{\|u_0\|_2} \left( (\hat{m}_0/\ell)^{\ln(2e/\eta)} (1+\epsilon)^{(2/d)\ell}\right)^{d/2} \lesssim_{\|u_0\|_2} (\ell_q/\ell)^{\frac{d}{2}\ln\left(\frac{2e}{\eta}\right)} (1+\epsilon)^{\ell} \ ,
\end{align*}
which justifies the size assumption in the two lemmas for $t< \tilde{t}_1(0)$.
With in mind that $[\hat{\ell}_0, \hat{\ell}_1]$ is of Type-$\cB$ at $\tilde{t}_1(0)$, the particular restriction of \eqref{eq:StrN0B1} at $\tilde{t}_1(0)$ together with Lemma~\ref{le:BstrBdd} (starting at $\tilde{t}_1(0)$) indicates that, for any $0\le r\le q$,
\begin{align*}
\max_{\ell_r\le j\le \ell_q} \sup_{\tilde{t}_1(0)<s<\hat{t}_2(0)} \cR(j, c(\ell_r), s) &\le (1+\tilde{\epsilon}_{\ell_q})^{1/\ell_q} \Theta(\hat{p}_0,r) \cdot \cR(\hat{m}_0, c(\ell_{\hat{p}_0}), \hat{t}_1(0))
\\
&\le (1+\tilde{\epsilon}_{\ell_q})^{1/\ell_q} \Theta(\hat{p}_0,r) \cdot \cR(\hat{m}_0, c(\ell_{\hat{p}_0}), t_0) \ .
\end{align*}
And the same argument as above leads to
\begin{align*}
\sup_{t_0 <s< \hat{t}_2(0)} \|u(s)\| &\le (1+\tilde{\epsilon}_{\ell_q})^{1/\ell_q} \|u_0\|_2 (\ell_q/\ell)^{\frac{d}{2}\ln\left(\frac{2e}{\eta}\right)} (1+\epsilon)^{\ell} \ ,
\end{align*}
which justifies the size assumption in the two lemmas up to $t< \hat{t}_2(0)$.

In the following, we provide a precise measurement for the slight increment of $\cR(\hat{m}_0, c(\ell_{\hat{p}_0}), s)$ within each time interval $[t,t+T_{k_q}]$ that Proposition~\ref{prop:HarMaxPrin} is applied to; the reason is twofold -- on one hand, although the increment of $\cR(\hat{m}_0, c(\ell_{\hat{p}_0}), s)$ is negligible in a short period, the accumulative effect can be significant in a relatively long time interval since the time span $T_{k_q}$ for each intermittency argument is very small compared to $T^*$ and large increment of $\cR(\hat{m}_0, c(\ell_{\hat{p}_0}), s)$ may cause the increment of $\|u(s)\|$ which would not guarantee the size condition for applying the two lemmas, in which case, it is necessary to transfer the above argument from $[\hat{\ell}_0, \hat{\ell}_1]$ to $[\hat{\ell}_i, \hat{\ell}_{i+1}]$ with larger index $i$ so that the two lemmas can be applied at $[\hat{\ell}_i, \hat{\ell}_{i+1}]$.
On the other hand, the increment of $\cR(\hat{m}_0, c(\ell_{\hat{p}_0}), s)$ will reduce the time span $T_{k_q}$ for the intermittency argument and this, in turn, makes $\cR(\hat{m}_0, c(\ell_{\hat{p}_0}), s)$ get multiplied faster when $s$ approaches  $T^*$.
Our purpose is to quantify such increment and explore how it affects the size of $T_{k_q}$ in order to decide whether a finite repetition of the above argument can lead to the regularity up to $T^*$.


We continue the previous argument at $\hat{t}_2(0)$. Again, by Lemma~\ref{le:AstrBdd}, for any $0\le r\le q$,
\begin{align*}
\max_{\ell_r\le j\le \ell_q} \sup_{\hat{t}_2(0) <s< \hat{t}_2(0) + T_{k_q}}\!\! \cR(j, c(\ell_r), s) \le (1+\tilde{\epsilon}_{\ell_q})^{1/\ell_q} \Theta(\hat{p}_0,r) \cdot \cR(\hat{m}_0, c(\ell_{\hat{p}_0}), \hat{t}_2(0)) \ ,
\end{align*}
where $T_{k_q}\approx 2^{-2k_q}\left(\cB_{\ell_q, \ell}\right)^{-2k_q} \|D^{k_q}u(\hat{t}_2(0))\|^{-\frac{2}{k_q+1}}$, and by Lemma~\ref{le:MaxIndxAStr} we know $k_q\le \ell_{3q}$; thus
\begin{align*}
T_{k_q} \gtrsim 2^{-2\ell_{3q}}\left(\cB_{\ell_q, \ell}\right)^{-2\ell_{3q}} \|D^{\ell_{3q}}u(\hat{t}_2(0))\|^{-\frac{2}{\ell_{3q}+1}} \ .
\end{align*}
Without loss of generality, we assume that $\hat{p}_0$ is invariant in time and that $\tilde t_2(0) \in [\hat{t}_2(0), \hat{t}_2(0) + T_{k_q}]$. In general, we assume $\tilde t_n(0) \in [\hat{t}_n(0), \hat{t}_n(0) + T_{k_q}]$, and by Lemma~\ref{le:AstrBdd} and Proposition~\ref{prop:HarMaxPrin}
\begin{align}
\sup_{\hat{t}_n(0) <s< \hat{t}_n(0) + T_{k_q}}\! \cR(\hat{m}_0, c(\ell_{\hat{p}_0}), s) &\le (1+\tilde{\epsilon}_{\ell_q})^{1/\ell_q} \cdot \cR(\hat{m}_0, c(\ell_{\hat{p}_0}), \hat{t}_n(0)) \ , \label{eq:m0IncreN}
\\
\cR(k_q, c(\ell_{\hat{p}_0}), \hat{t}_n(0) + T_{k_q}) &\le \left(\mu_{k_q}(\hat{p}_0)\right)^{\frac{1}{k_q+1}} \cdot \cR(k_q, c(\ell_{\hat{p}_0}), \hat{t}_n(0)) \ . \label{eq:kqDecreN}
\end{align}
We claim that with the above settings, one of the followings occurs:

(I) $\cR\left(\hat{m}_0, c(\ell_{\hat{p}_0}), \hat{t}_{n+1}(0)\right)\le \cR\left(\hat{m}_0, c(\ell_{\hat{p}_0}), \hat{t}_n(0)\right)$;

\vspace{0.04in}

(II) $\hat{t}_{n+1}(0) - \tilde t_n(0)\ge C\left(M_{k_q}\right) 2^{-2k_q} \|D^{k_q}u\left(\hat{t}_n(0) + T_{k_q}\right)\|^{-\frac{d}{k_q+d/2}}$ with $M_{k_q}=\left(\mu_{k_q}(\hat{p}_0)\right)^{-1}$.

\medskip

\noindent\textit{Proof of the claim:} Assume the opposite of (I), i.e. $\cR\left(\hat{m}_0, c(\ell_{\hat{p}_0}), \hat{t}_{n+1}(0)\right)\ge \cR\left(\hat{m}_0, c(\ell_{\hat{p}_0}), \hat{t}_n(0)\right)$. Without loss of generality we assume that $k_q$ is invariant with $\hat{t}_n(0)$, and that $\hat t_{n+1}(0)>\hat{t}_n(0) + T_{k_q}$.
With in mind that $[\hat{\ell}_0, \hat{\ell}_1]$ is of Type-$\cA$ at $\hat{t}_n(0)$ and at $\hat{t}_{n+1}(0)$, the opposite of (I) indicates that
\begin{align*}
&\cR\left(k_q, c(\ell_{\hat{p}_0}), \hat{t}_{n+1}(0)\right) = \cR\left(\hat{m}_0, c(\ell_{\hat{p}_0}), \hat{t}_{n+1}(0)\right)
\\
&\qquad\qquad\qquad \ge \cR\left(\hat{m}_0, c(\ell_{\hat{p}_0}), \hat{t}_n(0)\right) = \cR\left(k_q, c(\ell_{\hat{p}_0}), \hat{t}_n(0)\right)
\end{align*}
which, combined with \eqref{eq:kqDecreN}, yields
\begin{align*}
\cR\left(k_q, c(\ell_{\hat{p}_0}), \hat{t}_{n+1}(0)\right) \ge \left(\mu_{k_q}(\hat{p}_0)\right)^{-\frac{1}{k_q+1}} \cdot \cR(k_q, c(\ell_{\hat{p}_0}), \hat{t}_n(0) + T_{k_q}) \ ,
\end{align*}
in other words, $\|D^{k_q}u\left(\hat{t}_{n+1}(0)\right)\| \ge \left(\mu_{k_q}(\hat{p}_0)\right)^{-1} \|D^{k_q}u\left(\hat{t}_n(0) + T_{k_q}\right)\|$.
By Theorem~\ref{th:MainThmVel} (applied in a contrapositive form), the time span required for $D^{k_q}u$ to increase by $M_{k_q}=\left(\mu_{k_q}(\hat{p}_0)\right)^{-1}$ is at least (with in mind that $\hat t_{n+1}(0)>\hat{t}_n(0) + T_{k_q}$),
\begin{align*}
T_{k_q}^* := C\left(M_{k_q}\right) 2^{-2k_q} \|D^{k_q}u\left(\hat{t}_n(0) + T_{k_q}\right)\|^{-\frac{d}{k_q+d/2}} \ .
\end{align*}
Recall that $\tilde t_n(0)<\hat{t}_n(0) + T_{k_q}$, so
\begin{align*}
\hat{t}_{n+1}(0) - \tilde t_n(0) \ge T_{k_q}^* := C\left(M_{k_q}\right) 2^{-2k_q} \|D^{k_q}u\left(\hat{t}_n(0) + T_{k_q}\right)\|^{-\frac{d}{k_q+d/2}} \ .
\end{align*}
This ends the proof of the claim.
Moreover, by Lemma~\ref{le:MaxIndxAStr} we know $k_q\le \ell_{3q}$ and
\begin{align*}
\hat{t}_{n+1}(0) - \tilde t_n(0) \ge  T_{k_q}^* \ge C\left(M_{k_q}\right) 2^{-2\ell_{3q}} \|D^{\ell_{3q}}u\left(\hat{t}_n(0) + T_{k_q}\right)\|^{-\frac{d}{\ell_{3q}+d/2}} \ .
\end{align*}
The above claim together with multiple iterations of \eqref{eq:m0IncreN} lead to
\begin{align*}
\cR(\hat{m}_0, c(\ell_{\hat{p}_0}), \hat{t}_{n+1}(0)) &\le (1+\tilde{\epsilon}_{\ell_q})^{\nu/\ell_q} \cdot \cR(\hat{m}_0, c(\ell_{\hat{p}_0}), \hat{t}_2(0))
\end{align*}
where $\nu$ is the total number of times that Case~(II) in the claim occurs within $[\hat{t}_2(0), \hat{t}_{n+1}(0)]$. The worst scenario is $\nu =n$, that is, Case~(II) in the claim occurs throughout $[\hat{t}_2(0), \hat{t}_{n+1}(0)]$, in which case, the above restriction, together with Lemma~\ref{le:AstrBdd} and Lemma~\ref{le:BstrBdd} (applied $n$ times), indicates that, for any $0\le r\le q$,
\begin{align*}
\max_{\ell_r\le j\le \ell_q} \sup_{t_0 <s< \hat{t}_{n+1}(0)}\!\! \cR(j, c(\ell_r), s) \le (1+\tilde{\epsilon}_{\ell_q})^{n/\ell_q} \Theta(\hat{p}_0,r) \cdot \cR(\hat{m}_0, c(\ell_{\hat{p}_0}), t_0) \ .
\end{align*}
Recall that the precise upper bound for $1+\tilde{\epsilon}_{\ell_q}$ was given in the proof of Lemma~\ref{le:AstrBdd}:
\begin{align*}
1+ \tilde{\epsilon}_{\ell_q} &\lesssim \zeta_{\ell_q}(\ell_{q-1}) \lesssim 1 + (M_{\ell_q}-1) \cdot \tilde{c}(\ell_q)/c(\ell_{q-1}) \lesssim 1 + 2^{-\ell_{q}}\left(\cB_{\ell_q, \ell}\right)^{-\ell_{q}} /c(\ell_{q-1}) \ .
\end{align*}

In the rest of the proof, we show that the above iterations of Lemma~\ref{le:AstrBdd} and Lemma~\ref{le:BstrBdd} repeat for finitely many times as $\hat{t}_n(0)$ is approaching $T^*$ by revealing that the time span $\hat{t}_{n+1}(0) - \hat{t}_n(0)$ (or $\hat{t}_{n+1}(i_*) - \hat{t}_n(i_*)$ for some index $i_*$) for each application of Lemma~\ref{le:AstrBdd} and Lemma~\ref{le:BstrBdd} remains greater than a fixed number.
Note that the above argument guarantees that at least for small values of $n$ this is the case:
\begin{align*}
\hat{t}_{n+1}(0) - \hat{t}_n(0) \ge T_{k_q} + T_{k_q}^* \ge 2^{-2\ell_{3q}} \|D^{\ell_{3q}}u\left(\hat{t}_n(0) + T_{k_q}\right)\|^{-\frac{d}{\ell_{3q}+d/2}} \ .
\end{align*}
Assuming this would continue as $\hat{t}_n(0)$ goes towards $T^*$, the maximal number of iterations until $\hat{t}_n(0)$ reaches $T^*$ is
\begin{align*}
n^*:= & T^* / (\hat{t}_{n+1}(0) - \hat{t}_n(0)) \le T^* \cdot 2^{2\ell_{3q}} \|D^{\ell_{3q}}u\left(\hat{t}_n(0) + T_{k_q}\right)\|^{\frac{d}{\ell_{3q}+d/2}}
\\
&\qquad \le T^* \cdot 2^{2\ell_{3q}} \|D^{\ell_{3q}}u\left(\hat{t}_1(0)\right)\|^{\frac{d}{\ell_{3q}+d/2}} \le T^* \cdot 2^{2\ell_{3q}} \cdot \ell_{3q} \cdot \|u_0\|^d
\end{align*}
while $\cR(\hat{m}_0, c(\ell_{\hat{p}_0}), s)$ increases at most by
\begin{align*}
(1+\tilde{\epsilon}_{\ell_q})^{n^*/\ell_q} &\le \left(1 + 2^{-\ell_{q}}\left(\cB_{\ell_q, \ell}\right)^{-\ell_{q}} /c(\ell_{q-1})\right)^{T^* \cdot 2^{2\ell_{3q}} \cdot \ell_{3q} \cdot (1+\epsilon)^{2\ell} /\ell_q}
\\
&\lesssim \exp\left( T^* \cdot 2^{2\ell_{3q}-\ell_{q}} \cdot (\ell_{3q}/ \ell_q) \cdot c(\ell_{q-1}) \cdot (1+\epsilon)^{2\ell}  \left(\cB_{\ell_q, \ell}\right)^{-\ell_{q}} \right)=: C^* ,
\end{align*}
provided $T^*\ll \left(\cB_{\ell_q, \ell}\right)^{\ell_{q}}$. Then, similarly to the estimates for $\|u(s)\|$ within $[t_0, \tilde{t}_1(0)]$,
\begin{align*}
\sup_{t_0 <s< \hat{t}_{n^*}(0)} \|u(s)\| &\lesssim \|u_0\|_2 \sup_{t_0 <s< \hat{t}_{n^*}(0)} \left(c(\ell)^{\frac{\ell}{\ell+1}} (\ell!)^{\frac{1}{\ell+1}} \cB_{\ell_{\hat{p}_0}, \ell} \cdot \cR\left(\hat{m}_0, c(\ell), s\right)\right)^{\frac{(d/2)(\ell+1)}{\ell+d/2}}
\\
&\lesssim \|u_0\|_2 \left(c(\ell_{\hat{p}_0}) \cdot \ell_{\hat{p}_0} \cdot \cB_{\ell_{\hat{p}_0}, \ell} \cdot (1+\tilde{\epsilon}_{\ell_q})^{n^*/\ell_q} \cR\left(\hat{m}_0, c(\ell_{\hat{p}_0}), t_0\right)\right)^{\frac{(d/2)(\ell+1)}{\ell+d/2}} \ ;
\end{align*}
thus
\vspace{-0.1in}
\begin{align*}
\sup_{t_0 <s< \hat{t}_{n^*}(0)} \|u(s)\| \lesssim (C^*)^{\frac{d}{2}} \|u_0\|_2 (\ell_q/\ell)^{\frac{d}{2}\ln\left(\frac{2e}{\eta}\right)} (1+\epsilon)^{\ell} \ .
\end{align*}
As $C^*\approx 1$, this justifies the condition $\sup_{t_0 <s< T^*} \|u(s)\| \lesssim (1+\epsilon)^{\ell}$, Lemma~\ref{le:AstrBdd} and Lemma~\ref{le:BstrBdd} are applicable, and the process described above may continue until $T^*$.

If $T^*\gtrsim \left(\cB_{\ell_q, \ell}\right)^{\ell_{q}}$, we separate $[t_0, T^*]$ at some $\cT_1\ll \left(\cB_{\ell_q, \ell}\right)^{\ell_{q}}$ such that the condition for Lemma~\ref{le:AstrBdd} and Lemma~\ref{le:BstrBdd} is satisfied within $[t_0, \cT_1]$, and regularity of the solution holds up to $\cT_1$.
Then, similarly, we separate $[\cT_1 , T^*]$ at some $\cT_2\ll \left(\cB_{\ell_{2q}, \ell_q}\right)^{\ell_{2q}}$ such that
\begin{align*}
\sup_{\cT_1 <s< \cT_2} \|u(s)\| \lesssim (1+\epsilon)^{\ell_q}
\end{align*}
which justifies the size condition in Lemma~\ref{le:AstrBdd} and Lemma~\ref{le:BstrBdd} applied to the string $[\hat{\ell}_1, \hat{\ell}_2]$, and regularity remains until $\cT_2$. Inductively, we divide $[\cT_i , T^*]$ at some $\cT_{i+1}\ll \left(\cB_{\hat{\ell}_{i+1}, \hat{\ell}_i}\right)^{\hat{\ell}_{i+1}}$ so that
\begin{align*}
\sup_{\cT_i <s< \cT_{i+1}} \|u(s)\| \lesssim (1+\epsilon)^{\hat{\ell}_i}
\end{align*}
and Lemma~\ref{le:AstrBdd} and Lemma~\ref{le:BstrBdd} are applicable to the string $[\hat{\ell}_i, \hat{\ell}_{i+1}]$ up to $\cT_{i+1}$. This dividing process stops at some index $i_*$ such that $T^*\ll \left(\cB_{\hat{\ell}_{i_*+1}, \hat{\ell}_{i_*}}\right)^{\hat{\ell}_{i_*+1}}$ and regularity remains until $T^*$ with $\|u(T^*)\|\lesssim (1+\epsilon)^{\hat{\ell}_{i_*}}$. In particular, $T^*$ is not a blow-up time.

\medskip

The proof for the vorticity is similar.

\end{proof}

\section{Conclusion}

The main goal of this paper was to demonstrate asymptotically critical nature of the NS regularity problem
within the framework of sparseness of the super-level sets of the higher-order derivatives of the
velocity field. 
The principal mechanism behind the proof is weakening the nonlinear effect at high (differential) levels
through the interplay between the spatial intermittency (utilized via the harmonic measure majorization
principle) and the local-in-time monotonicity properties of chains of derivatives (ascending \emph{vs.}
descending). Since the role of the ascending property is replacing the classical Gagliardo-Nirenberg
interpolation inequalities, this process can be thought of as `dynamic interpolation'.

\medskip

The follow up work includes repurposing and refining the techniques presented here to obtain stronger 
manifestations of criticality -- and in particular -- criticality with respect to 
the strength of diffusion in the context
of the 3D hyper-dissipative (HD) NS system. In a work \citet{Grujic2020},
the authors presented a mathematical evidence of \emph{criticality} of the \emph{Laplacian}.
More precisely, it was demonstrated that -- as soon as the hyper-diffusion exponent is greater than 1 and 
the flow is in a suitably defined `turbulent scenario' -- the 3D HD NS 
system does not allow spontaneous formation of singularities. To illustrate the impact of the result in a 
methodology free setting, the authors considered a two-parameter family of the rescaled blow-up
profiles (c.f. \citet{AlBr2022} where the ansatz was used to point out that in the Navier-Stokes case
this type of analysis does not rule out new scaling exponents), and showed that as soon as the hyper-diffusion exponent is 
greater than 1 a new region in the parameter 
space is ruled out. More importantly, the region is a neighborhood of the self-similar profile, i.e., the approximately self-similar 
blow-up -- a prime candidate for the singularity formation -- is ruled out
(\citet{Grujic2020}) for all HD NS models.

\bigskip

\bigskip

\centerline{\textbf{Acknowledgments}}

\medskip

The work of of Z.G. is supported in part by the National Science Foundation grant DMS--2009607,
``Toward criticality of the Navier-Stokes regularity problem''.

\bigskip

\bigskip

\noindent \textbf{COI Statement:} On behalf of all authors, the corresponding author states that there is no conflict of interest.

\bigskip

\bigskip

\noindent \textbf{Data Availability Statement:} This manuscript has no associated data.

\bigskip

\bigskip

\bibliographystyle{abbrvnat}
\bibliographystyle{plainnat}

\begin{thebibliography}{13}
\providecommand{\natexlab}[1]{#1}
\providecommand{\url}[1]{\texttt{#1}}
\expandafter\ifx\csname urlstyle\endcsname\relax
  \providecommand{\doi}[1]{doi: #1}\else
  \providecommand{\doi}{doi: \begingroup \urlstyle{rm}\Url}\fi
  
  
  
  
  

\bibitem[Ahlfors(1973)]{Ahlfors1973}
L.~V. Ahlfors.
\newblock \emph{Conformal invariants: topics in geometric function theory}.
\newblock McGraw-Hill Book Co., New York-D\"{u}sseldorf-Johannesburg, 1973.
\newblock McGraw-Hill Series in Higher Mathematics.




\bibitem[Albritton and Bradshaw (2022)]{AlBr2022}
D. Albritton and Z. Bradshaw.
\newblock Remarks on sparseness and regularity of {N}avier-{S}tokes solutions.
\newblock \emph{Nonlinearity}, 35, \penalty0 2858, 2022.






\bibitem[Bradshaw et~al.(2019)Bradshaw, Farhat, and Gruji\'{c}]{Bradshaw2019}
Z.~Bradshaw, A.~Farhat, and Z.~Gruji\'{c}.
\newblock An {A}lgebraic {R}eduction of the `{S}caling {G}ap' in the
  {N}avier--{S}tokes {R}egularity {P}roblem.
\newblock \emph{Arch. Ration. Mech. Anal.}, 231\penalty0 (3):\penalty0
  1983--2005, 2019.
\newblock ISSN 0003-9527.
\newblock \doi{10.1007/s00205-018-1314-5}.
\newblock URL \url{https://doi.org/10.1007/s00205-018-1314-5}.

\bibitem[Constantin(1990)]{Constantin1990}
P.~Constantin.
\newblock Navier-{S}tokes equations and area of interfaces.
\newblock \emph{Comm. Math. Phys.}, 129\penalty0 (2):\penalty0 241--266, 1990.
\newblock ISSN 0010-3616.
\newblock URL \url{http://projecteuclid.org/euclid.cmp/1104180744}.

\bibitem[Farhat et~al.(2017)Farhat, Gruji\'{c}, and Leitmeyer]{Farhat2017}
A.~Farhat, Z.~Gruji\'{c}, and K.~Leitmeyer.
\newblock The space {$B^{-1}_{\infty,\infty}$}, volumetric sparseness, and 3{D}
  {NSE}.
\newblock \emph{J. Math. Fluid Mech.}, 19\penalty0 (3):\penalty0 515--523,
  2017.
\newblock ISSN 1422-6928.
\newblock \doi{10.1007/s00021-016-0288-z}.
\newblock URL \url{https://doi.org/10.1007/s00021-016-0288-z}.

\bibitem[Gagliardo(1959)]{Gagliardo1959}
E.~Gagliardo.
\newblock Ulteriori propriet\`a di alcune classi di funzioni in pi\`u
  variabili.
\newblock \emph{Ricerche Mat.}, 8:\penalty0 24--51, 1959.
\newblock ISSN 0035-5038.

\bibitem[Gruji\'{c}(2001)]{Grujic2001}
Z.~Gruji\'{c}.
\newblock The geometric structure of the super-level sets and regularity for
  3{D} {N}avier-{S}tokes equations.
\newblock \emph{Indiana Univ. Math. J.}, 50\penalty0 (3):\penalty0 1309--1317,
  2001.
\newblock ISSN 0022-2518.
\newblock \doi{10.1512/iumj.2001.50.1900}.
\newblock URL \url{https://doi.org/10.1512/iumj.2001.50.1900}.

\bibitem[Gruji\'{c}(2013)]{Grujic2013}
Z.~Gruji\'{c}.
\newblock A geometric measure-type regularity criterion for solutions to the
  3{D} {N}avier-{S}tokes equations.
\newblock \emph{Nonlinearity}, 26\penalty0 (1):\penalty0 289--296, 2013.
\newblock ISSN 0951-7715.
\newblock \doi{10.1088/0951-7715/26/1/289}.
\newblock URL \url{https://doi.org/10.1088/0951-7715/26/1/289}.





\bibitem[Gruji\'{c}(2016)]{Grujic2016}
Z.~Gruji\'{c}.
\newblock Vortex stretching and anisotropic diffusion in the 3{D}
  {N}avier-{S}tokes equations.
\newblock In \emph{Recent advances in partial differential equations and
  applications}, volume 666 of \emph{Contemp. Math.}, pages 239--251. Amer.
  Math. Soc., Providence, RI, 2016.
\newblock \doi{10.1090/conm/666/13295}.
\newblock URL \url{https://doi.org/10.1090/conm/666/13295}.
	



\bibitem[Gruji\'{c} and Kukavica(1998)]{Grujic1998}
Z.~Gruji\'{c} and I.~Kukavica.
\newblock Space analyticity for the {N}avier-{S}tokes and related equations
  with initial data in {$L^p$}.
\newblock \emph{J. Funct. Anal.}, 152\penalty0 (2):\penalty0 447--466, 1998.
\newblock ISSN 0022-1236.
\newblock \doi{10.1006/jfan.1997.3167}.
\newblock URL \url{https://doi.org/10.1006/jfan.1997.3167}.




\bibitem[Gruji\'{c} and Xu(2020)]{Grujic2020}
Z.~Gruji\'{c} and L.~Xu.
\newblock Time-global regularity of the {N}avier-{S}tokes system with 
hyper-dissipation--turbulent scenario.
\newblock arXiv \url{https://arxiv.org/abs/2012.05692}.








\bibitem[Guberovi\'{c}(2010)]{Guberovic2010}
R.~Guberovi\'{c}.
\newblock Smoothness of {K}och-{T}ataru solutions to the {N}avier-{S}tokes
  equations revisited.
\newblock \emph{Discrete Contin. Dyn. Syst.}, 27\penalty0 (1):\penalty0
  231--236, 2010.
\newblock ISSN 1078-0947.
\newblock \doi{10.3934/dcds.2010.27.231}.
\newblock URL \url{https://doi.org/10.3934/dcds.2010.27.231}.




\bibitem[Iyer et~al.(2014)Iyer, Kiselev, and Xu]{Iyer2014}
G.~Iyer, A.~Kiselev, and X.~Xu.
\newblock Lower bounds on the mix norm of passive scalars advected by
  incompressible enstrophy-constrained flows.
\newblock \emph{Nonlinearity}, 27\penalty0 (5):\penalty0 973--985, 2014.
\newblock ISSN 0951-7715.
\newblock \doi{10.1088/0951-7715/27/5/973}.
\newblock URL \url{https://doi.org/10.1088/0951-7715/27/5/973}.




\bibitem[Kida(1985)]{Kida1985}
S.~Kida.
\newblock Three-dimensional periodic flows with high-symmetry.
\newblock \emph{J. Phys. Soc. Jpn.}, 54\penalty0 (1):\penalty0 2132--2136,
  1985.
\newblock ISSN 1078-0947.
\newblock \doi{10.3934/dcds.2010.27.231}.
\newblock URL \url{https://journals.jps.jp/doi/abs/10.1143/JPSJ.54.2132}.

\bibitem[Nirenberg(1959)]{Nirenberg1959}
L.~Nirenberg.
\newblock On elliptic partial differential equations.
\newblock \emph{Ann. Scuola Norm. Sup. Pisa (3)}, 13:\penalty0 115--162, 1959.

\bibitem[Rafner et~al.(2019)Rafner, Gruji\'c, Bach, Baerentzen, Gervang, Jia,
  Leinweber, Misztal, and Sherson]{Rafner2019}
J.~Rafner, Z.~Gruji\'c, C.~Bach, J.~A. Baerentzen, B.~Gervang, R.~Jia,
  S.~Leinweber, M.~Misztal, and J.~Sherson.
\newblock Geometry of turbulent dissipation and the Navier-Stokes regularity
  problem.
\newblock{\emph{Sci Rep}, 11: 8824, 2021}
\newblock \doi{10.1038/s41598-021-87774-y}
\newblock URL \url{https://doi.org/10.1038/s41598-021-87774-y}.



\bibitem[Ransford(1995)]{Ransford1995}
T.~Ransford.
\newblock \emph{Potential theory in the complex plane}, volume~28 of
  \emph{London Mathematical Society Student Texts}.
\newblock Cambridge University Press, Cambridge, 1995.
\newblock ISBN 0-521-46120-0; 0-521-46654-7.
\newblock \doi{10.1017/CBO9780511623776}.
\newblock URL \url{https://doi.org/10.1017/CBO9780511623776}.



\bibitem[Solynin(1999)]{Solynin1999}
A.~Y.~ Solynin,
\newblock Ordering of sets, hyperbolic metrics, and harmonic measure.
\newblock \emph{J. Math. Sci.} 95\penalty0 (1):\penalty0 2256, 1999.
\newblock ISSN 1573-8795.
\newblock \doi{10.1007/BF02172470}.
\newblock URL \url{https://link.springer.com/article/10.1007/BF02172470}

\end{thebibliography}

\def\cprime{$'$}

\end{document}